\documentclass[11pt]{article}

\usepackage[margin=1.1in]{geometry}
\usepackage[T1]{fontenc}
\usepackage[utf8]{inputenc}
\usepackage{lmodern}
\IfFileExists{microtype.sty}{\usepackage{microtype}}{}
\usepackage{amsmath,amssymb,amsthm}
\usepackage{booktabs}
\usepackage{graphicx}
\usepackage[numbers,sort&compress]{natbib}
\usepackage[
  colorlinks=true,
  linkcolor=black,
  citecolor=black,
  urlcolor=black,
  pdftitle={U-centering as subset ANOVA: edge regression and higher-order theory},
  pdfauthor={Xianyang Zhang},
  pdfsubject={U-centering, subset ANOVA, edge regression, and Hoeffding-component pairings},
  pdfkeywords={U-centering, subset ANOVA, edge regression, energy distance, Hoeffding decomposition, U-statistic}
]{hyperref}

\allowdisplaybreaks

\theoremstyle{plain}
\newtheorem{theorem}{Theorem}[section]
\newtheorem{proposition}[theorem]{Proposition}
\newtheorem{lemma}[theorem]{Lemma}
\newtheorem{corollary}[theorem]{Corollary}

\theoremstyle{definition}
\newtheorem{definition}[theorem]{Definition}

\theoremstyle{remark}
\newtheorem{remark}[theorem]{Remark}

\newcommand{\R}{\mathbb R}
\newcommand{\E}{\mathbb E}

\newcommand{\Var}{\operatorname{Var}}
\newcommand{\Cov}{\operatorname{Cov}}
\newcommand{\tr}{\operatorname{tr}}
\newcommand{\rank}{\operatorname{rank}}
\newcommand{\diag}{\operatorname{diag}}
\newcommand{\range}{\operatorname{range}}
\newcommand{\one}{\mathbf 1}
\newcommand{\cI}{\mathcal I}
\newcommand{\cU}{\mathcal U}
\newcommand{\cH}{\mathcal H}
\newcommand{\cL}{\mathcal L}
\newcommand{\cV}{\mathcal V}

\newcommand{\inner}[2]{\left\langle #1,#2\right\rangle}
\newcommand{\norm}[1]{\left\lVert #1\right\rVert}

\newcommand{\given}{\,\middle|\,}

\title{U-centering as subset ANOVA:\\ edge regression and higher-order theory}
\author{Xianyang Zhang\thanks{Email: zhangxiany@stat.tamu.edu}\\[2pt]
\normalsize Department of Statistics, Texas A\&M University}
\date{\today}

\begin{document}
\maketitle

\begin{abstract}
The unbiased sample versions of squared distance covariance and the
Hilbert--Schmidt independence criterion (HSIC) are fourth-order U-statistics,
yet U-centering evaluates them from pairwise arrays in $O(n^2)$ operations.
We show that U-centering is exactly the least-squares residual obtained after
fitting additive endpoint effects to a symmetric array with zero diagonal.  This
analysis of variance (ANOVA) interpretation explains the zero row sums and
the denominator $n(n-3)$ through the residual degrees of freedom.

The same residualization gives regression representations for dependence
and two-sample statistics.  It also identifies exact conditional permutation
variances and explains the relation between pooled, common-endpoint, and
fully interacted variance estimates.

For arrays indexed by $r$-element subsets of the $n$ sample labels,
higher-order U-centering removes all effects involving fewer than $r$
labels.  When $n\ge2r$, it projects onto a space of dimension
$\binom nr-\binom n{r-1}$.  For two symmetric square-integrable kernels
evaluated on independent and identically distributed (i.i.d.) paired
observations, the covariance between residual entries has an exact form:
each covariance is the corresponding projection
entry multiplied by the covariance of the kernels' highest-order population
interactions.  Consequences
include unbiased positive semidefinite (PSD) estimates of covariance
matrices for the highest-order Hoeffding components, exact mean-squared
errors for estimating these components, and moment formulas under
independence and permutation.  The normalized
residual inner product equals a direct unbiased estimator involving up to
$2r$ observations, yet costs $O(n^r)$ operations for fixed $r$.
For four-argument matching kernels, a further algebraic simplification gives exactly
unbiased estimators of a fourth spectral moment of kernel cross-covariance
using three matrix products, with $O(n^3)$ arithmetic and $O(n^2)$ storage.
A numerical illustration separates dependence structures with equal HSIC
and different spectral concentration.
\end{abstract}

\begin{center}
\end{center}

\section{Introduction}
\label{sec:introduction}
Many dependence measures are built from pairwise comparisons.  Distance
covariance uses distances between observations
\citep{SzekelyRizzoBakirov2007}, whereas HSIC uses kernel similarities
between observations, arranged in Gram matrices
\citep{GrettonEtAl2008,SejdinovicEtAl2013}.  The usual unbiased estimators of
squared distance covariance and HSIC have a useful but initially surprising
property: although their direct formulas involve as many as four distinct
observations, they can be computed from two $n\times n$ arrays in $O(n^2)$
arithmetic
\citep{SzekelyRizzo2014,SongEtAl2012}.  The same algebra appears in martingale
difference divergence and related conditional-mean dependence measures
\citep{ShaoZhang2014,ZhangYaoShao2018}.  For real-valued observations and ordinary
Euclidean distance, the additional ordering structure permits an
$O(n\log n)$ algorithm \citep{HuoSzekely2016}.  In this paper, we instead
study exact centering and inner-product identities for general subset arrays,
without imposing an ordering structure on the observations.  We use the
following notation.  Write
\[
 [n]=\{1,\ldots,n\},
 \qquad
 \binom{[n]}r=\{I\subseteq[n]:|I|=r\}.
\]
For any positive integer $s$, let $\one_s$ denote the all-ones vector in
$\R^s$.
Let $O_i=(X_i,Y_i)$, $i\in[n]$, be independent copies of $O=(X,Y)$, where
$X$ and $Y$ take values in $\mathcal X$ and $\mathcal Y$, respectively.  If
$I=\{i_1<\cdots<i_r\}$, write
$X_I=(X_{i_1},\ldots,X_{i_r})$, and define $Y_I$ similarly.

For a pairwise statistic, the observed entries are indexed by unordered pairs
of distinct sample labels.  We display them as a symmetric matrix, with the absent diagonal cells
set to zero.  Let $A=(A_{ij})$ be such
a symmetric zero-diagonal array. For $n\ge3$, ordinary U-centering is
\begin{equation}\label{eq:pairwise-ucenter-intro}
 \widetilde A_{ij}
 =A_{ij}
 -\frac{A_{i\cdot}+A_{j\cdot}}{n-2}
 +\frac{A_{\cdot\cdot}}{(n-1)(n-2)},
 \qquad i\ne j,
 \qquad
 \widetilde A_{ii}=0,
\end{equation}
where
\[
 A_{i\cdot}=\sum_{j\in[n]\setminus\{i\}}A_{ij},
 \qquad
 A_{\cdot\cdot}
 =\sum_{\substack{i,j\in[n]\\i\ne j}}A_{ij}
 =2\sum_{1\le i<j\le n}A_{ij}.
\]
For a second symmetric zero-diagonal array $B=(B_{ij})$, let $\widetilde B$ denote
its U-centered version, obtained by applying
\eqref{eq:pairwise-ucenter-intro} with $B$ in place of $A$.  When $n=3$, there
are three pairwise observations and three identifiable
endpoint effects.  The additive endpoint model can therefore fit every zero-diagonal
symmetric array exactly, so its residual---and hence its U-centered
array---is zero.  For $n\ge4$, the normalized cross-product is
\begin{equation}\label{eq:intro-ucov}
 \frac{1}{n(n-3)}
 \sum_{i\ne j}\widetilde A_{ij}\widetilde B_{ij}.
\end{equation}
For squared distance covariance, unbiased HSIC, and related pair-kernel
statistics, \eqref{eq:intro-ucov} is a fourth-order U-statistic even though
its matrix implementation is quadratic in $n$.

The row sums and grand sum in \eqref{eq:pairwise-ucenter-intro} have a direct
least-squares interpretation.  Treat $A_{ij}$ as an observation on the edge
$\{i,j\}$ of a complete graph and fit
\[
 A_{ij}=u_i+u_j+\varepsilon_{ij}.
\]
The least-squares residual is exactly \eqref{eq:pairwise-ucenter-intro}.  Its
normal equations state that the residuals involving each label sum to zero,
which explains the zero row and column sums.  Moreover, the residual space has
dimension $\binom n2-n=n(n-3)/2$.  Since the sum in
\eqref{eq:intro-ucov} runs over ordered pairs and therefore counts each edge
twice, its denominator is twice the residual degrees of freedom, namely
$n(n-3)$.

The same residual projector also gives regression interpretations
of dependence and two-sample statistics.  In both problems, pairwise
quantities are adjusted for endpoint effects before a slope is fitted.  For
dependence, both pair arrays are adjusted before one is regressed on the
other.  For two-sample testing, pairwise dissimilarities are regressed on the
between-group pair indicator.  The common-endpoint model assigns one endpoint
effect to each pooled observation across all pair types.  The fully interacted
regression introduced here allows that effect to differ between within- and
between-group pairs and yields the pooled residual standard error used by
\citet{ChakrabortyZhang2021}.  Section~\ref{sec:edge-regression} also relates
the common-endpoint studentization to the pooled U-centered estimator of
\citet{GaoShao2023}, before we turn to higher-order subsets.

We next extend the same construction to larger subsets.  For $n\ge2r$ and an array
$(A_I:I\in\binom{[n]}r)$, write $\cU_{n,r}$ for the higher-order U-centering
operator: $\cU_{n,r}A$ is the unique least-squares residual obtained after
fitting all effects involving fewer than $r$ labels.  Its explicit formula is
given in Definition~\ref{def:higher-u}.  In particular, its $(r-1)$-way
margins vanish:
  \begin{equation*}
 \sum_{i\notin J}(\cU_{n,r}A)_{J\cup\{i\}}=0,
 \qquad J\in\binom{[n]}{r-1}.
\end{equation*}
The residual degrees of freedom are
\[
 d_{n,r}
 =\binom nr-\binom n{r-1}
 =\binom nr\frac{n-2r+1}{n-r+1}.
\]
At $r=1$ this is ordinary centering; at $r=2$ it is
\eqref{eq:pairwise-ucenter-intro}; and at $r=3$ it removes the grand mean,
one-label effects, and two-label effects from a triple-indexed array.

The finite-sample ANOVA above is deterministic, while its statistical
interpretation follows from the population Hoeffding decomposition.  For
symmetric square-integrable kernels $a:\mathcal X^r\to\R$ and
$b:\mathcal Y^r\to\R$, let $a_r$ and $b_r$ denote their $r$th
Hoeffding components (see Section~\ref{sec:general_hoeffding}).  We study
  \begin{equation*}
 \Theta_r(a,b)
 =\E\!\left[a_r(X_1,\ldots,X_r)b_r(Y_1,\ldots,Y_r)\right].
\end{equation*}
For $r\ge1$, if $X$ and $Y$ are independent, then
$a_r(X_1,\ldots,X_r)$ and $b_r(Y_1,\ldots,Y_r)$ are independent.  Since each
$r$th Hoeffding component has mean zero, $\Theta_r(a,b)=0$.
Different component levels can describe different features of dependence, as
the four-argument matching-kernel example in Section~\ref{sec:example}
illustrates.

A direct unbiased estimator groups products $a(X_I)b(Y_J)$ according to the
overlap size $s=|I\cap J|$ and combines the resulting averages with weights
$(-1)^{r-s}\binom rs$.  The weights remove the Hoeffding components indexed
by $j<r$ and retain only the $r$th component.  For fixed $r$, both subset-margin
inversion and the normalized inner product of higher-order U-centered arrays
evaluate the scalar in $O(n^r)$; the latter route also returns the residual
arrays:
  \begin{equation*}
 \frac{1}{d_{n,r}}
 \sum_{I\in\binom{[n]}r}
 (\cU_{n,r}A)_I(\cU_{n,r}B)_I,
 \qquad A_I=a(X_I),\quad B_I=b(Y_I).
\end{equation*}
The projection representation also yields zero margins and residual degrees
of freedom, and shows that centering either array is sufficient.

Hoeffding decompositions and U-statistics are classical
\citep{Hoeffding1948,Serfling1980}.  Finite-population and exchangeable
extensions are developed by \citet{BloznelisGotze2001} and
\citet{Peccati2004}.  The rank theory of inclusion matrices, which record containment between
subsets, goes back to
\citet{Gottlieb1966}, while projection methods for functions on fixed-size
subsets are discussed by \citet{Dunkl1978,Filmus2016}.  We use this established
projector and rank theory as the algebraic basis for the statistical
development below.  For $n\ge 2r$, \citet[Theorem~1]{WangLindsay2014} give an
exact unbiased estimator of the full variance of a general order-$r$
U-statistic.  Our special case with the same kernel evaluated on the same sample in
both arrays concerns only the
variance of the highest-order Hoeffding component.  In this case, the
projection identity rewrites the overlap-based estimator as a nonnegative
residual mean square.
Exact permutation moments for pairwise matrix statistics have a long
history, including \citet{Mantel1967}; related conditional moments for kernel
two-sample statistics appear in \citet{SongChen2024}.  Our permutation
argument expresses these moments through the same subset projector and its
dimension.  The spectral computation below uses partition inversion for
distinct-label sums, an established graph-counting technique
\citep{Roth2017}, to obtain explicit matrix formulas for the particular
matching-kernel estimators.

Against this background, the paper develops four connected results.
\begin{itemize}
\item Empirical U-centering is exactly subset ANOVA residualization.  The
projection explains zero margins, residual degrees of freedom, and the
regression representations of dependence and generalized-energy statistics.
\item The covariance of two residual arrays is their highest-order Hoeffding
pairing times the projector.  This identifies what each residual estimates,
gives its exact mean-squared estimation error, and supplies a PSD unbiased covariance
estimator for several highest-order components.
\item The projection also gives exact moments under independence and under
permutation.  In the two-sample setting, it identifies the permutation
standard deviation of the slope and shows that the pooled mean square is
the permutation average of either regression mean square.
\item For four-argument matching kernels, the projected statistic measures
spectral concentration beyond total HSIC.  Its unbiased estimator and an
exactly unbiased estimator of the fourth spectral moment are computable
using three dense matrix products and quadratic storage.
\end{itemize}

Sections~\ref{sec:pairwise-anova} and~\ref{sec:pairwise-hoeffding} develop the
pairwise projection and population identities.
Section~\ref{sec:edge-regression} gives the pair-array regressions and
two-sample permutation consequences, while
Section~\ref{sec:higher-order-u-centering} extends the projection to
$r$-subset arrays.  Section~\ref{sec:population-pairings} establishes the
overlap and projection representations, then develops covariance,
component estimation, independence, and permutation results.  It ends with
variance-component and spectral applications, including a numerical
illustration.  Longer algebraic and statistical proofs and the complete
spectral computation are collected in the appendices.

\section{Pairwise U-centering as ANOVA residualization}
\label{sec:pairwise-anova}

We first establish the projection interpretation in the pairwise case.
Because unordered pairs of sample labels are the edges of a complete graph,
a pairwise array can be viewed as an edge-indexed response.  U-centering
removes the additive effects associated with the two endpoints of each edge
and retains the interaction orthogonal to those endpoint effects.

\subsection{Symmetric zero-diagonal arrays}
A symmetric pairwise array assigns one value to each unordered pair
$\{i,j\}\in\binom{[n]}2$.  We display these values in an $n\times n$ matrix,
although its diagonal cells are not observations: a U-statistic uses distinct
sample labels, so the pairs $\{i,i\}$ are absent.  Setting the diagonal to
zero embeds the edge-indexed array in a symmetric matrix without introducing
self-pairs.  Let
\[
 \cV_{n,2}
 =\left\{A\in\R^{n\times n}:A=A^\mathsf T,\ A_{ii}=0
   \text{ for every }i\right\}.
\]
For $A=(A_{ij})$ and $B=(B_{ij})$ in $\cV_{n,2}$, define their edge inner
product by
\begin{equation}\label{eq:edge-inner}
  \inner{A}{B}_{E}=\sum_{i<j}A_{ij}B_{ij}.
\end{equation}
The associated edge norm is $\norm{A}_E=\inner{A}{A}_E^{1/2}$.
\begin{definition}[Pairwise U-centering]\label{def:pairwise-u}
For $n\ge3$ and $A\in\cV_{n,2}$, define $\widetilde A=\cU_{n,2}A$ by
\begin{equation}\label{eq:pairwise-u}
 \widetilde A_{ij}
 =A_{ij}-\frac{A_{i\cdot}+A_{j\cdot}}{n-2}
       +\frac{A_{\cdot\cdot}}{(n-1)(n-2)},
 \quad i\ne j,
 \qquad
 \widetilde A_{ii}=0.
\end{equation}
\end{definition}

The two endpoint terms remove their fitted additive contributions, and the
final grand-sum term corrects for their shared overall mean.  Because the
diagonal entries vanish, this notation agrees
with the standard U-centering formula.  The zero-row-sum and idempotence
properties were recorded by \citet{SzekelyRizzo2014}.  We next show that both
properties follow from an exact least-squares representation.

\subsection{The exact ANOVA characterization}

Define edge vectorization by
$\operatorname{vec}_E(A):=a_E=(A_{ij}:i<j)\in\R^{\binom{n}{2}}$; for a second array $B$,
define $\operatorname{vec}_E(B):=b_E=(B_{ij}:i<j)$.  Let
$Z\in\R^{\binom{n}{2}\times n}$ be the unsigned
edge-by-vertex incidence matrix,
  \begin{equation*}
 Z_{\{i,j\},k}=\mathbf 1\{k\in\{i,j\}\}.
\end{equation*}
Each row of $Z$ corresponds to one edge and contains exactly two ones in the
columns of its endpoints.  Consider the additive pair model
\begin{equation}\label{eq:pair-additive-model}
  A_{ij}=u_i+u_j+\varepsilon_{ij},\qquad i<j.
\end{equation}
Equivalently, writing $u_i=\mu/2+\alpha_i$ and imposing
$\sum_i\alpha_i=0$ gives the familiar ANOVA form
\begin{equation}\label{eq:pair-anova-model}
  A_{ij}=\mu+\alpha_i+\alpha_j+\varepsilon_{ij}.
\end{equation}
This fixed-array decomposition is related to the round-robin ANOVA underlying
the Social Relations Model, which separates actor, partner, and relationship
effects \citep{WarnerKennyStoto1979,KennyLaVoie1984}; in our setting symmetry
constrains the actor and partner roles to a common endpoint effect, and
$\varepsilon_{ij}$ is a deterministic least-squares residual.

\begin{theorem}[U-centering is pairwise ANOVA residualization]
\label{thm:pair-anova}
Let $n\ge3$ and $A\in\cV_{n,2}$.  The U-centered edge vector is the
least-squares residual from \eqref{eq:pair-additive-model}:
\begin{equation}\label{eq:pair-projection}
  \operatorname{vec}_{E}(\widetilde A)
  =M_{n,2}a_E,
  \qquad
  M_{n,2}=I_{\binom{n}{2}}-Z(Z^\mathsf TZ)^{-1}Z^\mathsf T.
\end{equation}
The least-squares estimates in the parametrization
\eqref{eq:pair-anova-model} are
\begin{equation}\label{eq:pair-ls}
 \widehat\mu=\frac{A_{\cdot\cdot}}{n(n-1)},
 \qquad
 \widehat\alpha_i
 =\frac{A_{i\cdot}-A_{\cdot\cdot}/n}{n-2}.
\end{equation}
Consequently, the fitted residual is exactly \eqref{eq:pairwise-u}.
\end{theorem}

\begin{proof}
We compute the Gram matrix of the incidence design and apply its inverse to
the row-sum vector.

Each edge contains two vertices, and two distinct vertices belong together to
exactly one edge.  Therefore
  \begin{equation*}
 Z^\mathsf TZ=(n-2)I_n+\one_n\one_n^\mathsf T.
\end{equation*}
The matrix is nonsingular for $n\ge3$, with inverse
\begin{equation}\label{eq:ztz-inverse}
 (Z^\mathsf TZ)^{-1}
 =\frac{1}{n-2}I_n
 -\frac{1}{2(n-1)(n-2)}\one_n\one_n^\mathsf T.
\end{equation}
Since $(Z^\mathsf Ta_E)_i=A_{i\cdot}$ and
$\one_n^\mathsf TZ^\mathsf Ta_E=A_{\cdot\cdot}$, the fitted vertex coefficient is
\[
 \widehat u_i
 =\frac{A_{i\cdot}}{n-2}
  -\frac{A_{\cdot\cdot}}{2(n-1)(n-2)}.
\]
Thus
\[
 A_{ij}-\widehat u_i-\widehat u_j
 =A_{ij}-\frac{A_{i\cdot}+A_{j\cdot}}{n-2}
  +\frac{A_{\cdot\cdot}}{(n-1)(n-2)},
\]
which is \eqref{eq:pairwise-u}.  Finally,
$\widehat\mu=2n^{-1}\sum_i\widehat u_i$ and
$\widehat\alpha_i=\widehat u_i-\widehat\mu/2$ give
\eqref{eq:pair-ls}.
\end{proof}

\begin{corollary}
\label{cor:pair-basic}
Let
\begin{align*}
 \cH_{n,2}
 &=\left\{C\in\cV_{n,2}:C\one_n=0\right\},\\
 \cL_{n,2}
 &=\left\{C\in\cV_{n,2}:C_{ij}=u_i+u_j
      \text{ for some }u\in\R^n\right\}.
\end{align*}
Then the edge space is the orthogonal direct sum
  \begin{equation*}
 \cV_{n,2}=\cL_{n,2}\oplus\cH_{n,2},
 \qquad \cL_{n,2}\perp\cH_{n,2},
\end{equation*}
and
\[
 \range(\cU_{n,2})=\cH_{n,2},
 \qquad
 \ker(\cU_{n,2})=\cL_{n,2}.
\]
In particular,
  \begin{equation*}
 \widetilde A\one_n=0,
 \qquad
 \one_n^\mathsf T\widetilde A=0,
 \qquad
 \cU_{n,2}^2=\cU_{n,2}.
\end{equation*}
Moreover, $\cU_{n,2}$ is self-adjoint under the edge inner product:
its matrix in edge coordinates is symmetric.
Equivalently, $\widetilde A=\cU_{n,2}A$ is the unique symmetric zero-diagonal array
such that
\[
 \widetilde A\one_n=0
 \qquad\text{and}\qquad
 A-\widetilde A\in\cL_{n,2}.
\]
\end{corollary}

\noindent The proof is given in
Appendix~\ref{app:proof-pair-basic}.

For a concrete example, take $n=4$ and consider the three ways to pair
the four labels into two disjoint edges, called matchings.  Let
\[
 s_1=A_{12}+A_{34},\qquad
 s_2=A_{13}+A_{24},\qquad
 s_3=A_{14}+A_{23}.
\]
Substitution into the U-centering formula gives
\begin{align*}
 \widetilde A_{12}=\widetilde A_{34}
 &=\frac{2s_1-s_2-s_3}{6},\\
 \widetilde A_{13}=\widetilde A_{24}
 &=\frac{2s_2-s_1-s_3}{6},\\
 \widetilde A_{14}=\widetilde A_{23}
 &=\frac{2s_3-s_1-s_2}{6}.
\end{align*}
These three residual values sum to zero, and each vertex is incident to
one edge from each matching.  Thus every row sum is zero and only two
residual values can be chosen freely: six edge observations leave two
residual degrees of freedom after the endpoint fit.

\subsection{Residual degrees of freedom and the normalization}
The edge space has dimension $\binom n2$, whereas the additive fitted space
has dimension $n$.  It follows that
  \begin{equation*}
 d_{n,2}:=\dim(\cH_{n,2})
 =\binom n2-n
 =\frac{n(n-3)}{2}.
\end{equation*}
The same rank was obtained from a vectorized representation of U-centering by
\citet{ZhuZhangYaoShao2020}.  For $n\ge4$, using the edge inner product defined in
\eqref{eq:edge-inner},
\begin{equation}\label{eq:pair-normalized-inner}
 \frac{1}{d_{n,2}}\inner{\widetilde A}{\widetilde B}_E
 =\frac{1}{n(n-3)}\sum_{i\ne j}\widetilde A_{ij}\widetilde B_{ij}.
\end{equation}
When $A=B$, \eqref{eq:pair-normalized-inner} is the residual mean square from
the endpoint ANOVA\@.  Here degrees of freedom count the dimension of a
linear space.  The ordered-pair sum counts each edge product twice, which
explains the factor of two between $d_{n,2}$ and $n(n-3)$.  This dimension
count does not require independent pair entries, and zero residual row sums
do not imply that those entries are statistically independent.
Self-adjointness and idempotence further give the
one-sided-centering identity
  \begin{equation*}
 \inner{\widetilde A}{\widetilde B}_E
 =\inner{A}{\widetilde B}_E
 =\inner{\widetilde A}{B}_E.
\end{equation*}
This identity underlies familiar simplifications for unbiased squared distance
covariance and HSIC \citep{SzekelyRizzo2014,SongEtAl2012}.

By Theorem~\ref{thm:pair-anova}, the least-squares representation gives the
following closed-form inner-product formula:
\[
 \inner{\widetilde A}{\widetilde B}_E
 =a_E^\mathsf TM_{n,2}b_E
 =a_E^\mathsf Tb_E
 -(Z^\mathsf Ta_E)^\mathsf T
   (Z^\mathsf TZ)^{-1}(Z^\mathsf Tb_E).
\]
Here $Z^\mathsf Ta_E=(A_{1\cdot},\ldots,A_{n\cdot})^\mathsf T$, and
\eqref{eq:ztz-inverse} gives
\begin{equation}\label{eq:pair-contraction}
 \frac{1}{n(n-3)}\sum_{i\ne j}\widetilde A_{ij}\widetilde B_{ij}
 =\frac{1}{n(n-3)}\bigg\{
 \sum_{i\ne j}A_{ij}B_{ij}
 -\frac{2}{n-2}\sum_{i=1}^n A_{i\cdot}B_{i\cdot}
 +\frac{A_{\cdot\cdot}B_{\cdot\cdot}}{(n-1)(n-2)}\bigg\}.
\end{equation}
Thus the computation requires only the two arrays, their row sums, and their
grand sums.  Once the arrays have been formed, all terms in
\eqref{eq:pair-contraction} can be accumulated in $O(n^2)$ arithmetic.

\section{\texorpdfstring{{Pairwise Hoeffding decomposition and fourth-order U-statistic identity}}{Pairwise Hoeffding decomposition and fourth-order U-statistic identity}}
\label{sec:pairwise-hoeffding}

We now turn from the sample-label projection to its population target.

\subsection{Population interaction and overlap moments}

Let $a:\mathcal X^2\to\R$ and $b:\mathcal Y^2\to\R$ be symmetric kernels such
that $\E\{a(X_1,X_2)^2\}<\infty$ and
$\E\{b(Y_1,Y_2)^2\}<\infty$.  The pairwise Hoeffding decomposition of $a$ is
\begin{equation}\label{eq:pair-hoeffding}
 a(x_1,x_2)=a_0+a_1(x_1)+a_1(x_2)+a_2(x_1,x_2),
\end{equation}
\begin{samepage}
\begingroup
\allowdisplaybreaks[0]
where
  \begin{align*}
 a_0&=\E a(X_1,X_2)\\
 a_1(x)&=\E a(x,X_2)-a_0\\
 a_2(x_1,x_2)
 &=a(x_1,x_2)-\E a(x_1,X)-\E a(X,x_2)+a_0,
\end{align*}
\endgroup
\end{samepage}
with $\E\{a_2(x,X)\}=0$ for almost every $x$.
Define $b_0,b_1,b_2$ in the same way and set
\begin{equation}\label{eq:theta2}
 \Theta_2(a,b)=\E\{a_2(X_1,X_2)b_2(Y_1,Y_2)\}.
\end{equation}
In general, $\Theta_2(a,b)$ is a cross-moment and may have either sign.  To
express \eqref{eq:theta2} in terms of observable raw kernels, define
  \begin{align*}
 \eta_2&=\E\{a(X_1,X_2)b(Y_1,Y_2)\}\\
 \eta_1&=\E\{a(X_1,X_2)b(Y_1,Y_3)\}\\
 \eta_0
 &=\E\{a(X_1,X_2)b(Y_3,Y_4)\}
  =\E a(X_1,X_2)\,\E b(Y_1,Y_2).
\end{align*}
Here the subscript records the number of shared sample labels.

\begin{proposition}
\label{prop:pair-overlap}
Let
$\Theta_0=a_0b_0$ and
$\Theta_1=\E\{a_1(X_1)b_1(Y_1)\}$.  Then
\[
 \eta_0=\Theta_0,\qquad
 \eta_1=\Theta_0+\Theta_1,\qquad
 \eta_2=\Theta_0+2\Theta_1+\Theta_2,
\]
and hence
\begin{equation}\label{eq:theta2-overlap}
 \Theta_2(a,b)=\eta_2-2\eta_1+\eta_0.
\end{equation}
\end{proposition}

\noindent The proof is given in
Appendix~\ref{app:proof-pair-overlap}.

\subsection{The direct estimator and its U-centered form}

For $n\ge3$, set
\[
 A_{ij}=a(X_i,X_j),\qquad B_{ij}=b(Y_i,Y_j),\qquad i\ne j,
\]
and set $A_{ii}=B_{ii}=0$.  To form the overlap estimators, assume
$n\ge4$.  With $(n)_q=n!/(n-q)!$, define
  \begin{align*}
 \widehat\eta_2
 &=\frac{1}{(n)_2}\sum_{i\ne j}A_{ij}B_{ij}\\
 \widehat\eta_1
 &=\frac{1}{(n)_3}\sum_{i,j,k\ \mathrm{all\ distinct}}
     A_{ij}B_{ik}\\
 \widehat\eta_0
 &=\frac{1}{(n)_4}\sum_{i,j,k,\ell\ \mathrm{all\ distinct}}
     A_{ij}B_{k\ell}.
\end{align*}
Each term is unbiased for the corresponding $\eta_s$, so
\begin{equation}\label{eq:direct-fourth}
 T_{n,2}^{\mathrm{dir}}
 =\widehat\eta_2-2\widehat\eta_1+\widehat\eta_0
\end{equation}
is unbiased for $\Theta_2(a,b)$.  After symmetrization it is a fourth-order
U-statistic.

\begin{theorem}[Pairwise compression identity]
\label{thm:pair-compression}
For $n\ge4$,
\begin{equation}\label{eq:pair-compression}
 T_{n,2}^{\mathrm{dir}}
 =\frac{1}{d_{n,2}}\sum_{i<j}\widetilde A_{ij}\widetilde B_{ij}
 =\frac{1}{n(n-3)}\sum_{i\ne j}\widetilde A_{ij}\widetilde B_{ij}.
\end{equation}
Thus this fourth-order U-statistic can be evaluated in $O(n^2)$ arithmetic once
the two pairwise arrays have been formed.
\end{theorem}

\begin{proof}
Expanding the row-sum cross-product and separating the terms with identical
second endpoints gives
\[
 \sum_{i=1}^n A_{i\cdot}B_{i\cdot}
 =(n)_2\widehat\eta_2+(n)_3\widehat\eta_1.
\]
For the product of the two grand sums, the two ordered edges either have the
same two labels, share exactly one label, or are disjoint.  When they have the
same two labels, the second edge has two possible orientations; when they
share one label, there are four choices for the positions of the common
endpoint.  Therefore,
\[
 A_{\cdot\cdot}B_{\cdot\cdot}
 =2(n)_2\widehat\eta_2
  +4(n)_3\widehat\eta_1
  +(n)_4\widehat\eta_0.
\]
Substituting these identities into \eqref{eq:pair-contraction} and simplifying
the three coefficients yields
\[
 \frac{1}{n(n-3)}
 \sum_{i\ne j}\widetilde A_{ij}\widetilde B_{ij}
 =\widehat\eta_2-2\widehat\eta_1+\widehat\eta_0,
\]
which is \eqref{eq:direct-fourth}.
\end{proof}

Theorem~\ref{thm:general-compression} below extends the compression identity to
arbitrary $r$-subsets.  The next proposition gives a complementary entrywise
interpretation through the conditional expectation of each U-centered entry.

\begin{proposition}
\label{prop:pair-conditional}
For a fixed pair $\{i,j\}$ and $n\ge3$,
\begin{equation}\label{eq:pair-conditional}
 \E\left(\widetilde A_{ij}\mid X_i,X_j\right)
 =\frac{n-3}{n-1}a_2(X_i,X_j).
\end{equation}
\end{proposition}

\noindent The proof is given in
Appendix~\ref{app:proof-pair-conditional}.

The population interaction $a_2(X_i,X_j)$ depends only on the fixed pair and
the underlying distribution.  In contrast, $\widetilde A_{ij}$ also depends
on the other observations through the sample margins.  Holding $X_i,X_j$
fixed and averaging over those other observations gives the factor
$(n-3)/(n-1)$ in \eqref{eq:pair-conditional}.  Thus the formula describes a
conditional mean, rather than a pointwise rescaling of the population
interaction for every realized sample.

The residual dimension also has a statistical interpretation.  Define
the variance of the population interaction by
$\sigma_a^2=\E\{a_2(X_1,X_2)^2\}$.  For $n\ge4$,
\begin{align*}
 \Cov\{\operatorname{vec}_E(\widetilde A)\}
 &=\sigma_a^2 M_{n,2},\\
 \E\norm{\operatorname{vec}_E(\widetilde A)}^2
 &=d_{n,2}\sigma_a^2.
\end{align*}
These are the pairwise cases of Theorem~\ref{thm:residual-covariance}.
Thus dividing the squared residual norm by $d_{n,2}$ gives an unbiased
estimate of $\sigma_a^2$.  This second-moment identity does not make the
residual entries independent.

\subsection{Distance covariance and HSIC}

For distance covariance, take
\[
 a(x,x')=\norm{x-x'},\qquad b(y,y')=\norm{y-y'}.
\]
Then $a_2$ and $b_2$ are the population-centered distance kernels and
$\Theta_2(a,b)=\operatorname{dCov}^2(X,Y)$.  Under the sufficient moment
condition $\E\norm X^2+\E\norm Y^2<\infty$,
\eqref{eq:pair-compression} is the usual unbiased squared sample distance
covariance \citep{SzekelyRizzoBakirov2007,SzekelyRizzo2014}.

For HSIC, let $a=k$ and $b=\ell$ be positive definite kernels.  The empirical
arrays are the Gram matrices with their diagonal entries set to zero:
\[
 A_{ij}=k(X_i,X_j)\mathbf 1\{i\ne j\},\qquad
 B_{ij}=\ell(Y_i,Y_j)\mathbf 1\{i\ne j\}.
\]
Then $\Theta_2(k,\ell)$ is the population HSIC, and
\eqref{eq:pair-compression} gives its unbiased U-statistic estimator
\citep{GrettonEtAl2008,SongEtAl2012}.  Bounded kernels automatically satisfy
the standing second-moment condition.  The correspondence between distance
and kernel formulations is discussed by \citet{SejdinovicEtAl2013}.

\section{Regression on pair arrays after U-centering}
\label{sec:edge-regression}

U-centering has two complementary uses.  First, the normalized cross-product
of two residual arrays is an unbiased estimator of a population Hoeffding
pairing, as shown in Section~\ref{sec:pairwise-hoeffding}.  Second, the
same construction gives a regression interpretation.  Treat the entry of one
pair array on each edge as the response and the corresponding entry of a
second array as the predictor.  U-centering both arrays first removes additive
endpoint effects; the slope is then fitted from the residual entries, with one
observation for each unordered pair.

\subsection{Single-predictor regression after residualization}

Let $y,x\in\R^M$ be the response and predictor vectors, let
$W\in\R^{M\times k}$ be a nuisance design matrix of rank $k$, and let $P_W$
be the orthogonal projector onto
$\operatorname{col}(W)^\perp$.  Thus $P_W$ removes from a vector the part
explained by $W$.  Write
\[
 d=\rank(P_W)=M-k.
\]
Provided the adjusted vectors $P_Wy$ and $P_Wx$ are nonzero, define their
normalized inner product by
\begin{equation}
 R_W(y,x)=
 \frac{y^\mathsf TP_Wx}
 {\{(y^\mathsf TP_Wy)(x^\mathsf TP_Wx)\}^{1/2}}.
 \label{eq:general-partial-edge-correlation}
\end{equation}
This is the cosine between the two adjusted vectors.  If the column space
of $W$ contains the constant vector, both adjusted vectors have mean zero,
so it is also their Pearson correlation.  The ANOVA designs used below
all contain the constant vector.

\begin{proposition}[Single-predictor residual regression]
\label{prop:general-edge-regression}
Suppose $d\ge2$, $P_Wx\ne0$, $P_Wy\ne0$, and $|R_W(y,x)|<1$.  In the linear model
\[
 y=W\gamma+\beta x+e,
\]
the conventional homoskedastic ordinary least squares (OLS) statistic for
testing $H_0:\beta=0$ is
\begin{equation}
 T_W=
 \frac{\sqrt{d-1}\,R_W(y,x)}
 {\{1-R_W(y,x)^2\}^{1/2}},
 \label{eq:general-edge-regression-t}
\end{equation}
and the full-model residual degrees of freedom are $d-1$.
\end{proposition}

This follows from the Frisch--Waugh--Lovell theorem.  First remove from $y$
and $x$ the parts explained by $W$, leaving $y_*=P_Wy$ and $x_*=P_Wx$.
Regressing $y_*$ on $x_*$ through the origin gives the full-model slope and
residual sum of squares:
\[
 \widehat\beta=\frac{x_*^\mathsf Ty_*}{x_*^\mathsf Tx_*},
 \qquad
 \mathrm{RSS}=\norm{y_*-\widehat\beta x_*}^2.
\]
The conventional standard error accounts for the $k$ degrees of freedom
used to fit $W$:
\[
 \widehat\sigma^2=\frac{\mathrm{RSS}}{d-1},
 \qquad
 \operatorname{se}(\widehat\beta)
 =\left\{\frac{\widehat\sigma^2}{x_*^\mathsf Tx_*}\right\}^{1/2},
 \qquad d-1=M-k-1.
\]
Taking $\widehat\beta/\operatorname{se}(\widehat\beta)$ gives
\eqref{eq:general-edge-regression-t}.  A separate regression on the $M$
residual entries must retain the divisor $M-k-1$: fitting $W$ has already
used $k$ degrees of freedom.  A no-intercept routine that treats these
entries as $M$ new responses can instead divide by $M-1$, so its reported
standard error requires a degrees-of-freedom correction.

These are algebraic identities for the conventional OLS statistic.  They do
not by themselves give it a Student $t$ null distribution.  In the pair-array
applications, entries can be dependent because they are computed from the
same observations.  Any reference distribution or resampling calibration
therefore needs a separate justification.  Under independence, however,
Corollary~\ref{cor:conditional-ols} shows that the conventional squared
standard error is conditionally unbiased for the residual-regression slope
variance.

\subsection{Dependence as a slope after endpoint residualization}

Let $n\ge4$ and $A,B\in\cV_{n,2}$, with edge vectorizations $a_E$ and
$b_E$, and let $Z$ be the unsigned edge-by-vertex incidence matrix from
Section~\ref{sec:pairwise-anova}.  Consider
\begin{equation}
 b_E=Z\gamma+\beta a_E+e.
 \label{eq:dyadic-regression}
\end{equation}
The term $Z\gamma$ gives each sample label its own additive effect.  Changing
$\gamma_i$ shifts every pair value that contains label $i$.  No separate
intercept is needed: adding the same constant to every endpoint effect shifts
every fitted edge value by a common amount.
By Theorem~\ref{thm:pair-anova}, projecting away these endpoint effects is
exactly the same as U-centering $A$ and $B$.
Suppose that $\widetilde A$ and $\widetilde B$ are nonzero.  Their partial
correlation is then
\begin{equation}
 R_U(A,B)=
 \frac{\inner{\widetilde A}{\widetilde B}_E}
 {\norm{\widetilde A}_E\norm{\widetilde B}_E}.
 \label{eq:u-residual-correlation}
\end{equation}
The residual space after fitting the endpoint effects has dimension
\[
 d_{n,2}=\rank(M_{n,2})=\binom n2-n=\frac{n(n-3)}2,
\]
so, if $|R_U(A,B)|<1$, Proposition~\ref{prop:general-edge-regression} gives
\begin{equation}
 T_{\mathrm{dep}}
 =\frac{\sqrt{d_{n,2}-1}\,R_U(A,B)}
 {\{1-R_U(A,B)^2\}^{1/2}}.
 \label{eq:edge-regression-t}
\end{equation}
We call $T_{\mathrm{dep}}$ the U-centered dependence $t$-statistic.  By
construction, it is the ordinary slope $t$-statistic in
\eqref{eq:dyadic-regression}.  When $A$ and $B$ are Euclidean distance arrays,
$R_U(A,B)$ is the bias-corrected sample distance correlation $R_n^*$, and
$T_{\mathrm{dep}}$ is the corresponding distance-correlation $t$-statistic of
\citet{SzekelyRizzo2013}.
\citet{ZhuZhangYaoShao2020} extend the same studentization to distance- and
reproducing kernel Hilbert space (RKHS)-based dependence measures.  This slope
regression has the residual degrees-of-freedom count
\[
 \binom n2-n-1:
 \qquad
 \text{pair observations}-\text{endpoint effects}-\text{one slope}.
\]

\subsection{The generalized-energy statistic as a slope on the between-group pair indicator}

Let $P$ and $Q$ be probability distributions on a common space $\mathcal Z$.
Let $X,X'$ be independent draws from $P$ and let $Y,Y'$ be independent draws
from $Q$, with all four random variables mutually independent.  For a
symmetric measurable dissimilarity $\rho:\mathcal Z^2\to\R$ for which the
expectations below are finite, define the population generalized-energy
contrast by
\[
 \mathcal E_\rho(P,Q)
 =2\E\rho(X,Y)-\E\rho(X,X')-\E\rho(Y,Y').
\]
The ordinary energy distance is the special case $\rho(x,y)=\norm{x-y}$.

Let $X_1,\ldots,X_n$ and $Y_1,\ldots,Y_m$, with $n,m\ge2$, be independent
samples from $P$ and $Q$, respectively.  Pool the observations as
\[
 U_1=X_1,\ldots,U_n=X_n,
 \qquad
 U_{n+1}=Y_1,\ldots,U_{n+m}=Y_m,
\]
and put $N=n+m$.  Set $G_i=1$ for $i\le n$ and $G_i=0$ for $i>n$ to record
group membership.  For $i<j$, define
\begin{equation}
 C_{ij}=\mathbf 1\{G_i\ne G_j\}=(G_i-G_j)^2,
 \qquad
 A_{ij}=\rho(U_i,U_j),
 \label{eq:between-group-pair-indicator}
\end{equation}
and let $a_E=\operatorname{vec}_E(A)$ and
$c_E=\operatorname{vec}_E(C)$.  Thus $c_E$ indicates which pairs are
between-group pairs.

The following three block averages give an unbiased sample version of the
population contrast:
\[
 \overline A_{XY}=\frac1{nm}\sum_{i\le n<j}A_{ij},
 \quad
 \overline A_{XX}=\frac{2}{n(n-1)}\sum_{i<j\le n}A_{ij},
 \quad
 \overline A_{YY}=\frac{2}{m(m-1)}\sum_{n<i<j\le N}A_{ij},
\]
and, for this fixed choice of $\rho$, the sample generalized-energy statistic is
\[
 \mathcal E_{n,m}=2\overline A_{XY}-\overline A_{XX}-\overline A_{YY}.
\]
This formula uses the dissimilarity convention, under which large cross-group
values contribute positively.  If a positive-definite kernel $k$ is used
directly as a similarity, define $\overline k_{XX}$, $\overline k_{XY}$, and
$\overline k_{YY}$ by replacing $A_{ij}$ with $k(U_i,U_j)$ in the three block
averages above.  The usual unbiased estimator of squared maximum mean
discrepancy (MMD) then has the opposite sign:
\[
 \widehat{\operatorname{MMD}}_{u,k}^2
 =\overline k_{XX}+\overline k_{YY}-2\overline k_{XY}
 =-\{2\overline k_{XY}-\overline k_{XX}-\overline k_{YY}\}.
\]
Thus, if $A_{ij}=k(U_i,U_j)$, the regression slope below is
$-\widehat{\operatorname{MMD}}_{u,k}^2/2$ under the kernel-similarity
convention \citep{SejdinovicEtAl2013}.

To express this contrast as a regression coefficient, define the
block-constant edge vector
$\ell_E=(\ell_{ij}:1\le i<j\le N)$ by
\begin{equation}
 \begin{gathered}
 \ell_{ij}=
 \begin{cases}
 -\{n(n-1)\}^{-1},&1\le i<j\le n,\\[1mm]
 (nm)^{-1},&1\le i\le n<j\le N,\\[1mm]
 -\{m(m-1)\}^{-1},&n<i<j\le N,
 \end{cases}\\[2mm]
 a_{n,m}:=\norm{\ell_E}>0,
 \qquad
 a_{n,m}^2
 =\frac1{nm}+\frac1{2n(n-1)}+\frac1{2m(m-1)}.
 \end{gathered}
 \label{eq:energy-edge-contrast-vector}
\end{equation}
Let $Z_N\in\R^{\binom{N}{2}\times N}$ be the unsigned edge-by-vertex
incidence matrix for the complete graph on the $N$ pooled labels, and let
$M_{N,2}$ be its residual projector,
defined as in \eqref{eq:pair-projection}.  The common-endpoint edge regression is
\begin{equation}
 a_E=Z_N\gamma+\theta c_E+e.
 \label{eq:simple-pooled-group-regression}
\end{equation}

\begin{proposition}[The generalized-energy statistic is a slope on the between-group pair indicator]
\label{prop:energy-group-slope}
The residualized between-group pair indicator satisfies
\begin{equation}
 M_{N,2}c_E=\frac{\ell_E}{a_{n,m}^2}.
 \label{eq:ucentered-group-indicator}
\end{equation}
Consequently, the least-squares slope in
\eqref{eq:simple-pooled-group-regression} is
\begin{equation}
 \widehat\theta
 =\ell_E^\mathsf Ta_E
 =\overline A_{XY}-\frac{\overline A_{XX}+\overline A_{YY}}2
 =\frac{\mathcal E_{n,m}}2.
 \label{eq:energy-as-group-slope}
\end{equation}
Equivalently, if $\widetilde A$ and $\widetilde C$ denote global U-centering
on all $N$ labels, then
\begin{equation}
 \mathcal E_{n,m}
 =2a_{n,m}^2\inner{\widetilde A}{\widetilde C}_E.
 \label{eq:energy-as-ucentered-label-covariance}
\end{equation}
\end{proposition}

The contrast vector $\ell_E$ is orthogonal to the endpoint-effect space,
so $\ell_E^\mathsf TM_{N,2}a_E=\ell_E^\mathsf Ta_E$.  Removing endpoint
effects therefore leaves this contrast unchanged, which explains why the
fitted slope is expressed using the original block averages.  The
explicit projection of $c_E$ is proved in
Appendix~\ref{app:two-sample-nesting}.

\subsection{Restricted and fully interacted studentizations}

We now compare two studentizations, meaning two ways of dividing the fitted
slope by an estimated standard error.
In the common-endpoint model \eqref{eq:simple-pooled-group-regression}, each
pooled observation has the same endpoint effect in its within- and
between-group edges.  After fitting the generalized-energy slope
$\widehat\theta$, the residual sum of
squares and its degrees of freedom are
\begin{equation}
 \mathrm{RSS}_R
 =a_E^\mathsf TM_{N,2}a_E-\frac{\widehat\theta^2}{a_{n,m}^2},
 \qquad
 \nu_R=\binom N2-N-1.
 \label{eq:restricted-two-sample-rss}
\end{equation}
Provided $\mathrm{RSS}_R>0$, the slope statistic for this common-endpoint
model is
\begin{equation}
 T_R=\frac{\widehat\theta}
 {a_{n,m}\sqrt{\mathrm{RSS}_R/\nu_R}}.
 \label{eq:restricted-energy-t}
\end{equation}
We call the regression the restricted model and the resulting standardization
of $\widehat\theta$ the restricted studentization.  A closely related pooled
studentization is due to \citet{GaoShao2023};
Remark~\ref{rem:gao-shao} gives the exact relation.

For the remainder of this subsection, assume $n,m\ge4$.  We now introduce a
fully interacted regression whose slope statistic
reproduces the pooled studentization of \citet{ChakrabortyZhang2021}.  This
regression uses a more flexible model for endpoint effects.  Put
$Q_{ij}=G_i+G_j-1$, which equals $1$, $0$, and $-1$ on $XX$, $XY$, and $YY$
edges, respectively, and fit
\begin{equation}
 A_{ij}=\eta+\delta Q_{ij}+\theta C_{ij}
 +(1-C_{ij})(\alpha_i+\alpha_j)
   +C_{ij}(\beta_i+\beta_j)+e_{ij},
 \qquad i<j.
 \label{eq:stacked-two-sample-edge-model}
\end{equation}
The three block intercepts are $\eta+\delta$ for $XX$, $\eta+\theta$ for
$XY$, and $\eta-\delta$ for $YY$.  Thus $\delta$ distinguishes the two
within-group blocks, while $\theta$ compares the between-group block with
their average.
For identifiability, the $\alpha$ effects sum to zero separately in the two
groups, and so do the $\beta$ effects.  The $\alpha_i$ are endpoint effects
for within-group edges, whereas the $\beta_i$ are endpoint effects for
between-group edges.  Thus this model allows the effect of an observation to
depend on the type of edge in which it appears.

Partition $A$ into two symmetric zero-diagonal within-group arrays and one
rectangular between-group array:
\[
 A^{XX}=(A_{ij})_{1\le i,j\le n},\qquad
 A^{XY}=(A_{i,n+j})_{\substack{1\le i\le n\\1\le j\le m}},\qquad
 A^{YY}=(A_{n+i,n+j})_{1\le i,j\le m}.
\]
Define the ordinary centering matrix
\[
 \Pi_s=I_s-\frac1s\one_s\one_s^\mathsf T.
\]
Then $\Pi_nA^{XY}\Pi_m$ subtracts the row and column means and adds back
the overall mean.  Write $\norm{\cdot}_F$ for the Frobenius norm, the square
root of the sum of squared matrix entries.

\begin{proposition}[Nested edge regressions for two-sample homogeneity]
\label{prop:energy-anova-t}
Let $n,m\ge4$.  In model \eqref{eq:stacked-two-sample-edge-model}:
\begin{enumerate}
\item the slope is again $\widehat\theta=\mathcal E_{n,m}/2$;
\item the residual sum of squares is
\begin{equation}
 \mathrm{RSS}_F=\mathrm{RSS}_{XY}+\mathrm{RSS}_{XX}+\mathrm{RSS}_{YY},
 \label{eq:stacked-full-rss}
\end{equation}
where
\begin{align*}
 \mathrm{RSS}_{XY}
 &=\norm{\Pi_nA^{XY}\Pi_m}_F^2,\\
 \mathrm{RSS}_{XX}
 &=\operatorname{vec}_E(A^{XX})^\mathsf T
   M_{n,2}\operatorname{vec}_E(A^{XX}),\\
 \mathrm{RSS}_{YY}
 &=\operatorname{vec}_E(A^{YY})^\mathsf T
   M_{m,2}\operatorname{vec}_E(A^{YY});
\end{align*}
\item the residual degrees of freedom are
\begin{equation}
 \nu_F
 =(n-1)(m-1)+\frac{n(n-3)}2+\frac{m(m-3)}2
 =\binom N2-2N+1;
 \label{eq:stacked-two-sample-df}
\end{equation}
\item let $S_{n,m}$ denote the pooled variance estimator of
\citet{ChakrabortyZhang2021}, whose blockwise formula is reproduced in
\eqref{eq:cz-pooled-variance}.  It satisfies
$S_{n,m}=4\mathrm{MS}_F$, where
$\mathrm{MS}_F=\mathrm{RSS}_F/\nu_F$ is the full-model residual mean square.
Consequently, provided
$\mathrm{RSS}_F>0$, the ordinary slope statistic is
\begin{equation}
 T_F=\frac{\widehat\theta}{a_{n,m}\sqrt{\mathrm{MS}_F}}
 =\frac{\mathcal E_{n,m}}{a_{n,m}\sqrt{S_{n,m}}},
 \label{eq:energy-anova-t}
\end{equation}
which, under the same dissimilarity convention, is
exactly their studentized statistic;
\item the restricted model is nested in the fully interacted model and
\begin{equation}
 \mathrm{RSS}_R=\mathrm{RSS}_F+\mathrm{SS}_{\mathrm{int}},
 \qquad
 \mathrm{SS}_{\mathrm{int}}\ge0,
 \qquad
 \nu_R-\nu_F=N-2,
 \label{eq:nested-rss-decomposition}
\end{equation}
where $\mathrm{SS}_{\mathrm{int}}$ is the extra sum of squares explained by
the endpoint-by-pair-type interactions.
\end{enumerate}
\end{proposition}

Thus the two procedures fit different endpoint effects when forming the
residual mean square.  The additional sum of squares can be made explicit.
Let $\widehat\alpha^X,\widehat\alpha^Y$ be the fitted centered endpoint
effects in the two within-group blocks, and let
$\widehat\beta^X,\widehat\beta^Y$ be the fitted centered row and column
effects in the cross-group block.  These are the fitted versions of the
$\alpha$ and $\beta$ effects in
\eqref{eq:stacked-two-sample-edge-model}.  Then
\begin{equation}
 \mathrm{SS}_{\mathrm{int}}
 =\frac{m(n-2)}{N-2}
       \norm{\widehat\alpha^X-\widehat\beta^X}^2
 +\frac{n(m-2)}{N-2}
       \norm{\widehat\alpha^Y-\widehat\beta^Y}^2.
 \label{eq:interaction-ss-explicit}
\end{equation}
Thus the interactions measure how much an observation's fitted within-group
endpoint effect differs from its fitted between-group endpoint effect.  Write
$\mathrm{MS}_R=\mathrm{RSS}_R/\nu_R$.  If the common slope is nonzero and
both residual sums of squares are positive, the nesting identity gives
\begin{equation}
 |T_F|>|T_R|
 \quad\Longleftrightarrow\quad
 \frac{\mathrm{SS}_{\mathrm{int}}}{N-2}>\mathrm{MS}_F.
 \label{eq:interaction-studentization-comparison}
\end{equation}
The comparison is therefore governed by the mean square associated with
the additional interactions.  Appendix~\ref{app:two-sample-nesting} proves
\eqref{eq:interaction-ss-explicit} and the results below.

\subsection{Permutation calibration and pooled mean squares}
\label{sec:two-sample-permutation}

Retain $n,m\ge4$, and write
\[
 d=d_{N,2},\qquad Q_0=a_E^\mathsf TM_{N,2}a_E.
\]
The quantity $Q_0$ is the residual sum of squares after removing endpoint
effects from the pooled array, before fitting the group slope.  The three
mean squares correspond to successive fits: $Q_0/d$ removes only the global
endpoint effects, $\mathrm{MS}_R$ also fits the group slope, and
$\mathrm{MS}_F$ additionally fits the $N-2$ endpoint-by-pair-type
interactions.  Unlike the latter two quantities, $Q_0/d$ does not depend on
how the pooled observations are divided into groups.  We now relate it to
the variation of the slope and mean squares across group assignments.

\begin{remark}[Connection with Gao and Shao]
\label{rem:gao-shao}
Under the constant-diagonal condition of
\citet[Proposition~10]{GaoShao2023}, and after
aligning the dissimilarity convention, let $\widetilde A$ be the globally
U-centered zero-diagonal pooled array.  Write $a_0^k$ for the common
diagonal value of their original, uncentered matrix, before its diagonal
is replaced by zero.  Put $c=a_0^k$.  Their centered off-diagonal entries equal
$\widetilde A_{st}-c/(N-1)$.  Since
$\sum_{s\ne t}\widetilde A_{st}=0$,
\[
 \sum_{s\ne t}
 \left(\widetilde A_{st}-\frac{c}{N-1}\right)^2
 =2Q_0+\frac{Nc^2}{N-1}.
\]
The diagonal correction in their equation~(11) therefore gives
\[
 \frac{2Q_0+Nc^2/(N-1)}{N(N-3)}
 -\frac{c^2}{(N-1)(N-3)}
 =\frac{Q_0}{d}.
\]
Here $\sum_{s\ne t}\widetilde A_{st}^{\,2}=2Q_0$ and $N(N-3)=2d$.
If $Q_0>0$, their studentized statistic can therefore be written as
\begin{equation}
 T_{\mathrm{GS}}
 =\frac{\widehat\theta}{a_{n,m}\sqrt{Q_0/d}}.
 \label{eq:gao-shao-pooled-statistic}
\end{equation}
Gao and Shao establish the unconditional null variance and unbiasedness of
their pooled variance estimator.  The following result gives a conditional
permutation interpretation of the same denominator.
\end{remark}

\begin{proposition}[Permutation moments for the two-sample regressions]
\label{prop:two-sample-permutation}
Fix a real symmetric zero-diagonal pooled array $A$ and group sizes $n,m\ge4$.
Assign the $N=n+m$ labels uniformly to groups of sizes $n$ and $m$.
Use a superscript $\pi$ for a statistic recomputed after this assignment,
and let $\E_\pi$ and $\Var_\pi$ denote expectation and variance over the
assignments, holding $A$ fixed.  Then
\begin{equation}
 \E_\pi\widehat\theta^\pi=0,\qquad
 \Var_\pi(\widehat\theta^\pi)
 =a_{n,m}^2\frac{Q_0}{d},
 \label{eq:two-sample-permutation-slope}
\end{equation}
and
\begin{equation}
 \E_\pi\mathrm{MS}_R^\pi
 =\E_\pi\mathrm{MS}_F^\pi=\frac{Q_0}{d},\qquad
 \E_\pi\mathrm{SS}_{\mathrm{int}}^\pi
 =(N-2)\frac{Q_0}{d}.
 \label{eq:two-sample-permutation-mean-squares}
\end{equation}
\end{proposition}

For $Q_0>0$, $T_{\mathrm{GS}}$ is consequently the slope divided by its
exact conditional permutation standard deviation.  Since
$\mathrm{RSS}_R=Q_0-\widehat\theta^2/a_{n,m}^2$ and $\nu_R=d-1$,
whenever $\mathrm{RSS}_R>0$,
\begin{equation}
 T_R=T_{\mathrm{GS}}
 \left\{\frac{d-1}{d-T_{\mathrm{GS}}^2}\right\}^{1/2}.
 \label{eq:restricted-pooled-monotonicity}
\end{equation}
For fixed $A$, $Q_0$ is invariant under reassignment, and the displayed
transformation is strictly increasing on $(-\sqrt d,\sqrt d)$.  On
assignments for which the statistics are defined, the raw energy statistic,
$T_{\mathrm{GS}}$, and $T_R$ therefore have the same signed ordering; their
absolute values also have the same ordering.  If $Q_0>0$ and
$\mathrm{RSS}_R^\pi>0$ for every assignment, they give identical permutation
tests using the same one-sided or absolute-value two-sided comparison.  The fully interacted mean square depends on the
partition, so $T_F$ need not preserve these orderings.  When $Q_0=0$, all
the moment identities remain valid, but these standardized statistics are
undefined.

The proposition is an identity for a fixed array.  For independent samples
under the null hypothesis $P=Q$, relabeling the pooled observations leaves
their joint distribution unchanged.  This exchangeability justifies a
permutation test.  Under the same hypothesis, the pooled mean square also
has a variance-reduction interpretation.

\begin{corollary}[Averaging the regression mean squares under homogeneity]
\label{cor:pooled-mean-square-averaging}
Suppose $P=Q$, let $Z,Z'$ be independent draws from this common law, and
assume $\E\rho(Z,Z')^2<\infty$.  Let $\rho_2$ be the second
Hoeffding component of $\rho$, defined as in \eqref{eq:pair-hoeffding}, and set
$\sigma_{\rho,2}^2=\E\rho_2(Z,Z')^2$.  Then
\begin{equation}
 \E(Q_0/d)=\E\mathrm{MS}_R=\E\mathrm{MS}_F=\sigma_{\rho,2}^2,
 \qquad
 \E\mathrm{SS}_{\mathrm{int}}=(N-2)\sigma_{\rho,2}^2.
 \label{eq:two-sample-null-mean-squares}
\end{equation}
Let $\mathcal F_{\mathrm{sym}}$ denote the information in the pooled sample
that is unchanged by relabeling.  For $J\in\{R,F\}$,
\begin{equation}
 \E(\mathrm{MS}_J\mid\mathcal F_{\mathrm{sym}})=Q_0/d.
 \label{eq:pooled-mean-square-rao-blackwell}
\end{equation}
If, in addition, $\E|\rho(Z,Z')|^4<\infty$, then
\begin{equation}
 \Var(\mathrm{MS}_J)
 =\Var(Q_0/d)+\E\{\Var_\pi(\mathrm{MS}_J^\pi)\}
 \ge\Var(Q_0/d).
 \label{eq:pooled-mean-square-variance-reduction}
\end{equation}
\end{corollary}

Thus averaging either regression mean square over all group assignments
gives the pooled mean square exactly.  Under homogeneity, this averaging is
a Rao--Blackwell improvement for estimating $\sigma_{\rho,2}^2$.  The fourth
moment condition concerns the pairwise dissimilarity values; a moment
condition on the raw observations alone would require a growth bound on
$\rho$.  None of these identities asserts independence of the sums of
squares or a Student reference distribution.  For the dependence
regression, Corollary~\ref{cor:conditional-ols} gives the corresponding
conditional interpretation of its squared standard error under independence.

\begin{remark}[Fixed-sample, high-dimensional limits]
When the sample sizes are fixed and the ambient dimension grows,
\citet[Corollary~2.1]{ZhuZhangYaoShao2020} and
\citet[Theorem~4.2]{ChakrabortyZhang2021} show, under their
respective null hypotheses and assumptions, that $T_{\mathrm{dep}}$ and $T_F$
converge to Student $t$ distributions.
\end{remark}

\section{Higher-order U-centering}
\label{sec:higher-order-u-centering}

We now extend pairwise residualization to arrays indexed by $r$-subsets.  The
endpoint effects are replaced by all lower-order subset effects, and the zero
row sums become zero $(r-1)$-way margins.

\subsection{\texorpdfstring{{The hierarchy for triple-indexed arrays}}{The hierarchy for triple-indexed arrays}}

Assume $n\ge6$, and consider a symmetric array $A_{ijk}$ indexed by unordered
triples of distinct labels.  The fitted model includes an overall mean,
one-label effects, and pair effects:
\[
 A_{ijk}=\mu+\alpha_i+\alpha_j+\alpha_k
 +\beta_{ij}+\beta_{ik}+\beta_{jk}+\varepsilon_{ijk}.
\]
The parametrization is redundant because pair effects can already represent
the grand mean and the one-label effects.  Indeed, set
\[
 g_{ij}=\beta_{ij}+\frac{\alpha_i+\alpha_j}{2}+\frac{\mu}{3}.
\]
Then $g_{ij}+g_{ik}+g_{jk}$ equals the entire fitted part of the display
above.  Individual coefficients in the hierarchical parametrization are
therefore not unique, but the fitted array and its least-squares residual
are.  For $n=6$, there are $\binom63=20$ triple entries and the pair-effect
design has rank $\binom62=15$, leaving five residual dimensions.  The grand
mean and the six one-label effects add no further fitted directions.  The
general rank statement is given in Theorem~\ref{thm:higher-projection}.

Here $\varepsilon_{ijk}$ denotes the deterministic least-squares residual,
not a stochastic error term.  Its normal equations are the zero pair-margin
conditions
\[
 \sum_{k\notin\{i,j\}}\varepsilon_{ijk}=0
 \qquad\text{for every }\{i,j\}\in\binom{[n]}2.
\]
Thus U-centering for triple-indexed arrays removes every effect involving fewer
than three labels.  The condition $n\ge6$ is the $r=3$ specialization of the
general condition $n\ge2r$.

\subsection{Arrays, margins, and the lower-order ANOVA space}

Fix $r\ge1$ and $n\ge2r$, and let
\[
 \cI_{n,r}=\binom{[n]}r,
 \qquad
 \cV_{n,r}=\R^{\cI_{n,r}}.
\]
An element $A\in\cV_{n,r}$ is an $r$-subset array evaluated only at
distinct sample labels.  For $A,B\in\cV_{n,r}$, define the inner product by
  \begin{equation*}
 \inner{A}{B}=\sum_{I\in\cI_{n,r}}A_I B_I.
\end{equation*}
For $0\le j\le r$, with $\binom{[n]}0=\{\varnothing\}$, define the margin map
$D_{j,r}:\cV_{n,r}\to\R^{\binom{[n]}{j}}$ by
  \begin{equation*}
 (D_{j,r}A)(J)
 =\sum_{\substack{I\in\cI_{n,r}\\I\supseteq J}}A_I,
 \qquad J\in\binom{[n]}j.
\end{equation*}
For $j=0$, this is the grand sum $D_{0,r}A=\sum_I A_I$.  If
$f:\binom{[n]}j\to\R$, the adjoint map is
  \begin{equation*}
 (D_{j,r}^\ast f)(I)
 =\sum_{\substack{J\subseteq I\\|J|=j}}f(J).
\end{equation*}
Both subset spaces use the unweighted sum-of-products inner product, so
the adjoint is simply the transpose of the margin matrix.
Thus $D_{j,r}$ collects $j$-set margins, whereas $D_{j,r}^\ast$ spreads a
function on $j$-sets back to the $r$-sets containing them.
For the six-label example, these operations are concrete sums:
\[
 (D_{2,3}A)(\{1,2\})=A_{123}+A_{124}+A_{125}+A_{126},
\]
and
\[
 (D_{2,3}^\ast g)(\{i,j,k\})=g_{ij}+g_{ik}+g_{jk}.
\]
The first collects all triple values containing a fixed pair.  The second
adds the three pair effects belonging to a fixed triple, exactly as in the
preceding subsection.

The full lower-order ANOVA fit is
\begin{equation}\label{eq:hyper-anova}
 A_I
 =\sum_{j=0}^{r-1}\ \sum_{\substack{J\subseteq I\\|J|=j}}
   \alpha_J+\varepsilon_I.
\end{equation}
Although this parametrization is redundant, its fitted-value space is
  \begin{equation*}
 \cL_{n,r}=\range(D_{r-1,r}^\ast).
\end{equation*}
To see this, let $f$ be a function on $j$-sets with $j<r-1$.  Then
  \begin{equation*}
 D_{r-1,r}^\ast D_{j,r-1}^\ast f
 =(r-j)D_{j,r}^\ast f.
\end{equation*}
For a fixed $r$-set $I$, each $j$-subset $J\subseteq I$ is contained in
exactly $r-j$ of the $(r-1)$-subsets of $I$, which proves the identity.
Therefore, all lower-order fitted effects are contained in
$\range(D_{r-1,r}^\ast)$.  Conversely, the $j=r-1$ term in
\eqref{eq:hyper-anova} spans this range.  Hence the full fitted-value space is
exactly $\cL_{n,r}=\range(D_{r-1,r}^\ast)$.  This nesting property of inclusion
matrices is also discussed in \citet{Gottlieb1966}.  The residual, or pure
$r$-way interaction, space is
  \[
 \cH_{n,r}=\cL_{n,r}^{\perp}=\ker(D_{r-1,r}).
\]
Indeed, since $\cL_{n,r}=\range(D_{r-1,r}^{\ast})$, for any
$H\in\cV_{n,r}$,
\[
\begin{aligned}
 H\in\cL_{n,r}^{\perp}
 &\quad\Longleftrightarrow\quad
 \inner{H}{D_{r-1,r}^{\ast}f}=0
 \quad\text{for every }f\in\R^{\binom{[n]}{r-1}}\\
 &\quad\Longleftrightarrow\quad
 \inner{D_{r-1,r}H}{f}=0
 \quad\text{for every }f
 \quad\Longleftrightarrow\quad
 D_{r-1,r}H=0.
\end{aligned}
\]
Hence $H\in\cH_{n,r}$ if and only if
  \[
 \sum_{i\notin J}H_{J\cup\{i\}}=0
 \qquad\text{for every }J\in\binom{[n]}{r-1}.
\]
\subsection{An explicit higher-order projection}
We seek coefficients $c_{n,r,j}$ for a linear combination of the margin
operators $D_{j,r}^\ast D_{j,r}$.  Since $D_{r,r}$ is the identity, set
$c_{n,r,r}=1$ to retain the raw entry, and choose the remaining coefficients
so that every $(r-1)$-margin vanishes.
The margin count carried out in
Appendix~\ref{app:higher-projection-proof} yields the recurrence
 \[
 c_{n,r,j+1}
 =-\frac{n-r-j+1}{r-j}c_{n,r,j},
 \qquad 0\le j<r.
\]
Solving this recurrence backward from $c_{n,r,r}=1$ gives
  \[
 c_{n,r,j}
 =\frac{(-1)^{r-j}}{\binom{n-r-j+1}{r-j}}.
\]
\begin{definition}[Higher-order U-centering]
\label{def:higher-u}
For $A\in\cV_{n,r}$, define
\begin{equation}\label{eq:higher-u-operator}
 \cU_{n,r}A
 =\sum_{j=0}^r c_{n,r,j}D_{j,r}^\ast D_{j,r}A.
\end{equation}
Equivalently, for $I\in\cI_{n,r}$,
\begin{equation}\label{eq:higher-u-entry}
 (\cU_{n,r}A)_I
 =\sum_{j=0}^r
 \frac{(-1)^{r-j}}{\binom{n-r-j+1}{r-j}}
 \sum_{\substack{J\subseteq I\\|J|=j}}
 \sum_{\substack{I'\in\cI_{n,r}\\I'\supseteq J}}A_{I'}.
\end{equation}
\end{definition}

For $r=1$, \eqref{eq:higher-u-entry} is ordinary sample centering.  For
$r=2$ and $n\ge4$, it becomes
\[
 (\cU_{n,2}A)_{ij}
 =A_{ij}-\frac{A_{i\cdot}+A_{j\cdot}}{n-2}
       +\frac{A_{\cdot\cdot}}{(n-1)(n-2)},
\]
which is pairwise U-centering.

\begin{samepage}
\begin{theorem}[Higher-order U-centering is the ANOVA projection]
\label{thm:higher-projection}
Let $n\ge2r$.  The operator $\cU_{n,r}$ is the orthogonal projection of
$\cV_{n,r}$ onto the residual space
$\cH_{n,r}=\ker(D_{r-1,r})$.  Equivalently, for every
$A\in\cV_{n,r}$, $\cU_{n,r}A$ is the least-squares residual obtained after
fitting every ANOVA effect involving fewer than $r$ labels.  In particular,
\begin{enumerate}
\item $D_{r-1,r}(\cU_{n,r}A)=0$ for every $A$;
\item $\cU_{n,r}A=A$ for every $A\in\cH_{n,r}$;
\item $\cU_{n,r}$ is self-adjoint and idempotent;
\item $\ker(\cU_{n,r})=\cL_{n,r}$; and
\item the residual degrees of freedom are given by
\begin{equation}\label{eq:general-df}
 d_{n,r}:=\dim(\cH_{n,r})
 =\binom nr-\binom n{r-1}
 =\binom nr\frac{n-2r+1}{n-r+1}.
\end{equation}
\end{enumerate}
\end{theorem}
\end{samepage}

\noindent The proof is given in
Appendix~\ref{app:higher-projection-proof}.

\subsection{\texorpdfstring{{The explicit projection for triple-indexed arrays}}{The explicit projection for triple-indexed arrays}}

Assume $n\ge6$.  Let $A_{ijk}$ be symmetric and defined only for distinct
$i,j,k$.  The margins for a triple-indexed array are
\begin{align*}
 M_2(i,j)&=\sum_{k\notin\{i,j\}}A_{ijk},\\
 M_1(i)&=\sum_{\substack{j<k\\j,k\ne i}}A_{ijk},\\
 M_0&=\sum_{i<j<k}A_{ijk}.
\end{align*}
Then, for distinct $i,j,k$,
\begin{equation}\label{eq:r3-formula}
\begin{split}
 (\cU_{n,3}A)_{ijk}
 ={}&A_{ijk}
 -\frac{M_2(i,j)+M_2(i,k)+M_2(j,k)}{n-4}\\
 &+\frac{M_1(i)+M_1(j)+M_1(k)}{\binom{n-3}{2}}
 -\frac{M_0}{\binom{n-2}{3}}.
\end{split}
\end{equation}
The alternating terms follow the ANOVA hierarchy: retain the original triple entry,
subtract the appropriately scaled pair margins, add back the scaled one-label
margins, and finally adjust for the grand sum.  The normal equations for
this fit require the residuals to sum to zero for every fixed pair:
  \begin{equation*}
 \sum_{k\notin\{i,j\}}(\cU_{n,3}A)_{ijk}=0.
\end{equation*}
The residual degrees of freedom are
  \begin{equation*}
 d_{n,3}=\binom n3-\binom n2
 =\frac{n(n-1)(n-5)}{6}.
\end{equation*}

\subsection{Higher-order subset regression}
\label{sec:higher-order-subset-regression}

The pairwise construction extends by replacing edges with $r$-subsets and
endpoint effects with all lower-order subset effects.  Fix an ordering of
$\cI_{n,r}$ and identify
$A,B\in\cV_{n,r}$ with the coordinate vectors
$\mathbf a=(A_I:I\in\cI_{n,r})$ and
$\mathbf b=(B_I:I\in\cI_{n,r})$ in $\R^{\binom nr}$.  By
\eqref{eq:general-df}, $\dim(\cL_{n,r})=\binom n{r-1}$.  Choose a
full-column-rank matrix
\[
 W_{<r}\in\R^{\binom nr\times\binom n{r-1}}
\]
whose column space is the lower-order ANOVA space $\cL_{n,r}$.  For
$\gamma\in\R^{\binom n{r-1}}$, consider
\begin{equation}
 \mathbf b=W_{<r}\gamma+\beta\mathbf a+\mathbf e,
 \label{eq:higher-order-subset-regression}
\end{equation}
where $\mathbf e$ is the residual vector.  Under this coordinate
identification, $P_{W_{<r}}=\cU_{n,r}$ because both
are the orthogonal projector onto $\cL_{n,r}^{\perp}$.  Thus the
regression first U-centers both arrays, removing every effect involving fewer
than $r$ labels, and then fits the slope between the adjusted arrays.

\begin{corollary}[Single-predictor regression on $r$-subsets]
\label{cor:higher-order-subset-regression}
Let $n\ge2r$ and $d_{n,r}\ge2$.  Suppose $\cU_{n,r}A$ and
$\cU_{n,r}B$ are nonzero, and define
\begin{equation}
 R_{n,r}^{\mathrm{sub}}(A,B)
 =\frac{\inner{\cU_{n,r}A}{\cU_{n,r}B}}
 {\norm{\cU_{n,r}A}\norm{\cU_{n,r}B}}.
 \label{eq:higher-order-residual-correlation}
\end{equation}
If $|R_{n,r}^{\mathrm{sub}}(A,B)|<1$, the conventional homoskedastic OLS statistic
for testing $H_0:\beta=0$ in
\eqref{eq:higher-order-subset-regression} is
\begin{equation}
 T_{n,r}^{\mathrm{slope}}
 =\frac{\sqrt{d_{n,r}-1}\,R_{n,r}^{\mathrm{sub}}(A,B)}
 {\{1-(R_{n,r}^{\mathrm{sub}}(A,B))^2\}^{1/2}},
 \qquad
 d_{n,r}=\binom nr-\binom n{r-1},
 \label{eq:higher-order-subset-t}
\end{equation}
and its full-model residual degrees of freedom are $d_{n,r}-1$.
\end{corollary}

The corollary follows from Proposition~\ref{prop:general-edge-regression} with
$M=\binom nr$, $P_W=\cU_{n,r}$, and $d=d_{n,r}$, together with
Theorem~\ref{thm:higher-projection}.
As in Section~\ref{sec:edge-regression}, this identifies the statistic and
its full-model degrees of freedom; it does not establish a Student $t$
null distribution for overlapping subset observations.

The same construction extends to several subset-level predictors: U-center
the response and every predictor column, then fit the resulting ordinary
multiple regression, provided the residualized predictor columns are linearly
independent.  With $p$ such predictors, conventional standard errors use
$d_{n,r}-p$ residual degrees of freedom and require $p<d_{n,r}$ and a positive
residual sum of squares.

\section{Population pairings and projection representation}
\label{sec:population-pairings}

\subsection{General Hoeffding components}\label{sec:general_hoeffding}

We now introduce the population counterpart of the sample-label ANOVA\@.
Fix $r\ge1$, and write $P_X$ for the law of $X$.  Let
$O_i=(X_i,Y_i)$, $i\ge1$, be i.i.d.\ copies of $O=(X,Y)$.  For
$I=\{i_1<\cdots<i_r\}$, write $X_I=(X_{i_1},\ldots,X_{i_r})$ and define
$Y_I$ analogously.  Let $a:\mathcal X^r\to\R$ and
$b:\mathcal Y^r\to\R$ be symmetric kernels satisfying
\[
 \E\{a(X_1,\ldots,X_r)^2\}<\infty,
 \qquad
 \E\{b(Y_1,\ldots,Y_r)^2\}<\infty.
\]
Unlike the sample-label ANOVA of Section~\ref{sec:higher-order-u-centering}, this decomposition takes place in
a population $L^2$ space and is defined by conditional expectation.  We use
the usual Hoeffding components.  To define them, write
\[
 g_j(x_1,\ldots,x_j)
 =\E\{a(x_1,\ldots,x_j,X_{j+1},\ldots,X_r)\},
 \qquad 0\le j\le r,
\]
where $g_0=\E a(X_1,\ldots,X_r)$ and $g_r=a$.  For
$S=\{s_1<\cdots<s_\ell\}\subseteq[j]$, write
$x_S=(x_{s_1},\ldots,x_{s_\ell})$.  Set $a_0=g_0$, interpret
$a_0(x_\varnothing)$ as $a_0$, and define recursively
\[
 a_j(x_1,\ldots,x_j)
 =g_j(x_1,\ldots,x_j)
 -\sum_{\ell=0}^{j-1}
   \sum_{\substack{S\subseteq[j]\\|S|=\ell}}a_\ell(x_S).
\]
The function $g_j$ keeps $j$ arguments fixed and averages over the remaining
$r-j$ arguments.  The component $a_j$ then removes everything already
explained by fewer than $j$ arguments.  In the ANOVA analogy, $g_j$ is a
marginal mean and $a_j$ is the pure $j$-way interaction.  For example,
when $r\ge2$,
\[
 a_1(x)=g_1(x)-g_0,\qquad
 a_2(x,x')=g_2(x,x')-g_1(x)-g_1(x')+g_0.
\]
More generally, the recursion is equivalent to
\[
 g_j(x_1,\ldots,x_j)
 =\sum_{S\subseteq[j]}a_{|S|}(x_S).
\]
Taking $j=r$ gives the Hoeffding decomposition
\begin{equation}\label{eq:general-hoeffding-decomp}
 a(X_I)=\sum_{S\subseteq I}a_{|S|}(X_S),
 \qquad I\in\binom{[n]}r.
\end{equation}
A symmetric $j$-argument kernel is called \emph{canonical} (or completely
degenerate) if integrating out any one argument gives zero for
$P_X^{j-1}$-almost every choice of the remaining arguments.  The components
in the standard Hoeffding decomposition above are symmetric and canonical;
in particular,
\begin{equation}\label{eq:general-canonical}
 \E\{a_j(x_1,\ldots,x_{j-1},X_j)\}=0,
 \qquad j\ge1.
\end{equation}
See \citet{Hoeffding1948} and \citet{Serfling1980} for the standard
derivation and canonicality properties.
We define $b_0,\ldots,b_r$ analogously.

For $0\le j\le r$, we define the bilinear interaction pairings
\begin{equation*}
 \Theta_j(a,b)
 =\E\{a_j(X_1,\ldots,X_j)b_j(Y_1,\ldots,Y_j)\},
\end{equation*}
where $\Theta_0=a_0b_0$.  The target of interest is the $r$th
Hoeffding-component pairing $\Theta_r(a,b)$.
If $X\perp Y$, then for every $j\ge1$ the two canonical components are
independent and have mean zero, so $\Theta_j(a,b)=0$.

\subsection{Population moments indexed by overlap size}

For $0\le s\le r$, take $2r-s$ i.i.d.\ paired observations and set
\[
 I=\{1,\ldots,r\},
 \qquad
 J_s=\{1,\ldots,s,r+1,\ldots,2r-s\}.
\]
Both $I$ and $J_s$ contain $r$ labels, and $|I\cap J_s|=s$.  We define
\[
 \eta_s(a,b)=\E\{a(X_I)b(Y_{J_s})\}.
\]
Equivalently,
\[
 \eta_s(a,b)
 =\E\left[
 a(X_1,\ldots,X_r)
 b(Y_1,\ldots,Y_s,Y_{r+1},\ldots,Y_{2r-s})
 \right].
\]
The subscripts denote sample labels rather than argument positions: the two
kernel evaluations share the first $s$ paired observations and use distinct
observations for their remaining arguments.  The following triangular identity
gives a bilinear form of the classical
covariance decomposition for Hoeffding components
\citep{Hoeffding1948,Serfling1980}.
When $X=Y$ and $a=b=h$, subtracting
$\eta_0=(\E h)^2$ gives the familiar relation between an overlap covariance
and the Hoeffding variance components.  In the present
notation, it also provides the link to the finite-sample projection below.

\begin{theorem}
\label{thm:overlap-anova}
For every $0\le s\le r$,
\begin{equation}\label{eq:eta-theta-triangular}
 \eta_s(a,b)=\sum_{j=0}^s\binom sj\Theta_j(a,b).
\end{equation}
Consequently,
\begin{equation}\label{eq:theta-binomial-inversion}
 \Theta_r(a,b)
 =\sum_{s=0}^r(-1)^{r-s}\binom rs\eta_s(a,b).
\end{equation}
\end{theorem}

\begin{proof}
Expand the two raw kernels using
\eqref{eq:general-hoeffding-decomp}.  Consider a term
$a_{|S|}(X_S)b_{|T|}(Y_T)$ with $S\subseteq I$ and $T\subseteq J_s$.
If $S\ne T$, some sample label belongs to exactly one of the two sets.  On
conditioning on all other labels, canonicality
\eqref{eq:general-canonical} makes the expectation zero.  A nonzero term must
therefore have $S=T\subseteq I\cap J_s=[s]$.  For each $j$, there are
$\binom sj$ shared subsets of size $j$, and each contributes
$\Theta_j(a,b)$.  This proves \eqref{eq:eta-theta-triangular};
\eqref{eq:theta-binomial-inversion} is ordinary binomial inversion.
\end{proof}
For $r=2$, Theorem~\ref{thm:overlap-anova} reduces to
Proposition~\ref{prop:pair-overlap}.
More generally, the overlap-$s$ moment contains Hoeffding components through
order $s$, and binomial inversion isolates the order-$r$ component.

\subsection{The direct overlap estimator}

We now turn from population interpretation to finite-sample estimation.
Assume $n\ge2r$, so that disjoint $r$-subsets exist and every overlap size
from zero to $r$ is represented.  For $I,J\in\cI_{n,r}$, let $A_I=a(X_I)$ and
$B_J=b(Y_J)$.  For $0\le s\le r$, the number of ordered pairs $(I,J)$ with
$|I\cap J|=s$ is
\begin{equation}\label{eq:N-s}
 N_{n,r,s}
 =\binom nr\binom rs\binom{n-r}{r-s}.
\end{equation}
Here $I$ and $J$ are themselves unordered subsets, but $(I,J)$ is an ordered
pair: choose $I$, choose the $s$ labels it shares with $J$, and then choose
the remaining $r-s$ labels of $J$.
Define
\begin{equation}\label{eq:eta-s-hat}
 \widehat\eta_{n,r,s}
 =\frac{1}{N_{n,r,s}}
  \sum_{\substack{I,J\in\cI_{n,r}\\|I\cap J|=s}}A_I B_J.
\end{equation}
By exchangeability, $\E\widehat\eta_{n,r,s}=\eta_s(a,b)$.  We therefore define
the direct unbiased estimator by
\begin{equation}\label{eq:direct-general}
 T_{n,r}^{\mathrm{dir}}
 =\sum_{s=0}^r(-1)^{r-s}\binom rs
  \widehat\eta_{n,r,s}.
\end{equation}
The term $s=0$ alone contains on the order of $n^{2r}$ products, so a literal
evaluation costs $O(n^{2r})$ arithmetic operations.  This count treats the
arrays $A$ and $B$ as already formed; evaluating the kernels on the
$\binom nr$ subsets is a separate, kernel-dependent cost.

Each $\widehat\eta_{n,r,s}$ is a U-statistic of order $2r-s$ after
symmetrization.  A U-statistic of order $q<2r$ can always be lifted to order
$2r$ by averaging its kernel over the $q$-subsets of $2r$ arguments.
Therefore \eqref{eq:direct-general} admits an order-$2r$ U-statistic
representation.  For $r=2$, this is the familiar fourth-order representation;
for $r=3$, it is a sixth-order representation.

\subsection{Direct evaluation by subset-margin inversion}
We first derive the $O(n^r)$ scalar evaluation directly, without using the
projection identity.  Define the unnormalized overlap sums
\begin{equation*}
 \Sigma_s(A,B)
 =\sum_{\substack{I,J\in\cI_{n,r}\\|I\cap J|=s}}A_I B_J,
 \qquad 0\le s\le r,
\end{equation*}
and, for $0\le k\le r$, the cross-products of the $k$-way margins
\begin{equation}\label{eq:Ck-margin}
 C_k(A,B)
 =\sum_{K\in\binom{[n]}k}
   (D_{k,r}A)(K)(D_{k,r}B)(K).
\end{equation}
For $k=0$, recall that $(D_{0,r}A)(\varnothing)=\sum_I A_I$.

For the observed kernel arrays, these quantities have distinct roles:
$\eta_s$ is a population moment, $\widehat\eta_{n,r,s}$ is its sample
average, and $\Sigma_s=N_{n,r,s}\widehat\eta_{n,r,s}$ is the corresponding
unnormalized sum.  By contrast, $C_k$ is computed from the $k$-way margins,
which can be accumulated in one pass through the array entries without
listing all pairs of overlapping subsets.
The next identity converts these readily accumulated margin cross-products
into the overlap sums needed by the estimator.

\begin{proposition}
\label{prop:margin-inversion}
For $0\le k\le r$,
\begin{equation}\label{eq:Ck-Sigma}
 C_k(A,B)=\sum_{s=k}^r\binom{s}{k}\Sigma_s(A,B).
\end{equation}
Consequently,
\begin{equation}\label{eq:Sigma-Ck-inversion}
 \Sigma_s(A,B)
 =\sum_{k=s}^r(-1)^{k-s}\binom{k}{s}C_k(A,B),
 \qquad 0\le s\le r.
\end{equation}
\end{proposition}

\noindent The proof is given in
Appendix~\ref{app:proof-margin-inversion}.

Since $\widehat\eta_{n,r,s}=\Sigma_s/N_{n,r,s}$,
Proposition~\ref{prop:margin-inversion} evaluates the direct estimator
\eqref{eq:direct-general} from the $r+1$ quantities $C_0,\ldots,C_r$.
Each $r$-subset contributes to the margins indexed by its $2^r$ subsets, so
all margins can be accumulated in $O(2^r\binom nr)$ additions.  Forming their
cross-products and applying \eqref{eq:Sigma-Ck-inversion} has the same
$O(n^r)$ order for fixed $r$.

For $r=2$, these margin cross-products are related to the terms in
\textup{(\ref{eq:pair-contraction})} by
\[
 C_0(A,B)=\tfrac14A_{\cdot\cdot}B_{\cdot\cdot},\qquad
 C_1(A,B)=\sum_{i=1}^n A_{i\cdot}B_{i\cdot},\qquad
 C_2(A,B)=\tfrac12\sum_{i\ne j}A_{ij}B_{ij}.
\]
Consequently, the right-hand side of \textup{(\ref{eq:pair-contraction})} is
\[
 \frac{1}{n(n-3)}
 \left\{
 2C_2(A,B)-\frac{2}{n-2}C_1(A,B)
 +\frac{4}{(n-1)(n-2)}C_0(A,B)
 \right\},
\]
which recovers the earlier pairwise inner-product formula.

\subsection{Projection representation by higher-order U-centering}
The margin cross-products from the preceding subsection also enter the ANOVA
projector, since
\begin{equation*}
 C_k(A,B)=\inner{A}{D_{k,r}^{\ast}D_{k,r}B}.
\end{equation*}
The population inversion removes lower-order Hoeffding contributions from
the overlap moments.  The margin inversion evaluates the corresponding
sample overlap sums without enumerating all pairs of subsets.  We now show
that projecting the arrays and taking their normalized inner product gives
exactly the same estimator.  The key is to group the projector entries by
overlap size and compare their weights with those in
\eqref{eq:direct-general}.

The entries of the projector depend only on the overlap between two
subsets.  Indeed,
\[
 (D_{j,r}^{\ast}D_{j,r})_{I,J}=\binom{|I\cap J|}{j},
\]
because this entry counts the $j$-subsets contained in both $I$ and $J$.
Substitution into the projection formula gives, for $I,J\in\cI_{n,r}$ with
$|I\cap J|=s$,
\begin{equation}\label{eq:p-s-def}
 p_{n,r,s}
 =(\cU_{n,r})_{I,J}
 =\sum_{j=0}^s
   \frac{(-1)^{r-j}\binom sj}{\binom{n-r-j+1}{r-j}}.
\end{equation}
The following closed form gives
the coefficient normalization used in Theorem~\ref{thm:general-compression}
to identify the projected inner product in \textup{(\ref{eq:general-compression})}
with the direct overlap estimator in \textup{(\ref{eq:direct-general})}.

\begin{lemma}
\label{lem:p-s}
Let $r\ge1$, $n\ge2r$, and $\kappa=n-2r+1$.  For $0\le s\le r$,
\begin{equation}\label{eq:p-s-closed}
 p_{n,r,s}
 =(-1)^{r-s}
 \frac{\kappa(r-s)!(\kappa+s-1)!}{(\kappa+r)!}
 =(-1)^{r-s}
 \frac{(n-2r+1)(r-s)!(n-2r+s)!}{(n-r+1)!}.
\end{equation}
Moreover,
\begin{equation}\label{eq:group-coeff}
 \frac{N_{n,r,s}p_{n,r,s}}{d_{n,r}}
 =(-1)^{r-s}\binom rs.
\end{equation}
\end{lemma}

The proof is given in Appendix~\ref{app:p-s-proof}.
For $r=2$, the three entries are
\[
 p_{n,2,2}=\frac{n-3}{n-1},\qquad
 p_{n,2,1}=-\frac{n-3}{(n-1)(n-2)},\qquad
 p_{n,2,0}=\frac{2}{(n-1)(n-2)}.
\]
They are the weights for the same pair, pairs sharing one label, and
disjoint pairs, respectively.

\begin{samepage}
\begin{theorem}[Projection representation and unbiasedness]
\label{thm:general-compression}
Let $a$ and $b$ be symmetric square-integrable kernels with $r$ arguments,
and let $n\ge2r$.  Form the arrays $A_I=a(X_I)$ and $B_I=b(Y_I)$ on
$\cI_{n,r}$.  Then
\begin{equation}\label{eq:general-compression}
 T_{n,r}^{\mathrm{dir}}
 =\frac{1}{d_{n,r}}
   \inner{\cU_{n,r}A}{\cU_{n,r}B}
 =\frac{1}{d_{n,r}}\inner{A}{\cU_{n,r}B}
 =\frac{1}{d_{n,r}}\inner{\cU_{n,r}A}{B}.
\end{equation}
In particular,
\begin{equation*}
 \E\left[
 \frac{1}{d_{n,r}}
 \inner{\cU_{n,r}A}{\cU_{n,r}B}
\right]
 =\Theta_r(a,b).
\end{equation*}
\end{theorem}
\end{samepage}

\begin{proof}
Since $\cU_{n,r}$ is self-adjoint and idempotent,
\[
 \inner{\cU_{n,r}A}{\cU_{n,r}B}
 =\inner{A}{\cU_{n,r}B}.
\]
Expanding the latter inner product and grouping the ordered pairs $(I,J)$
according to $s=|I\cap J|$, the definitions of $p_{n,r,s}$,
$\Sigma_s(A,B)$, and $\widehat\eta_{n,r,s}$ give
\begin{align*}
 \frac{1}{d_{n,r}}\inner{A}{\cU_{n,r}B}
 &=\frac{1}{d_{n,r}}\sum_{s=0}^r
   p_{n,r,s}\Sigma_s(A,B)\\
 &=\sum_{s=0}^r
   \frac{N_{n,r,s}p_{n,r,s}}{d_{n,r}}
   \widehat\eta_{n,r,s}\\
 &=\sum_{s=0}^r(-1)^{r-s}\binom rs
   \widehat\eta_{n,r,s}
 =T_{n,r}^{\mathrm{dir}}.
\end{align*}
The final line uses \eqref{eq:group-coeff} and the definition
\eqref{eq:direct-general}.
Unbiasedness follows from Theorem~\ref{thm:overlap-anova}.
\end{proof}

From the raw arrays, accumulating the margins and forming $\cU_{n,r}B$ each
cost $O(2^r\binom nr)$, since every entry touches $2^r$ margins.  The one-sided
identity $\inner{A}{\cU_{n,r}B}/d_{n,r}$ adds only $O(\binom nr)$, giving
$O(2^r\binom nr)=O(n^r)$ arithmetic for fixed $r$.

\subsection{\texorpdfstring{{The explicit estimator for three-argument kernels}}{The explicit estimator for three-argument kernels}}

We now specialize the general estimator to $r=3$.  Assume $n\ge6$.  For
symmetric three-argument kernels $a$ and $b$, define the arrays
$A_{ijk}=a(X_i,X_j,X_k)$ and $B_{ijk}=b(Y_i,Y_j,Y_k)$ on distinct triples.
The target is
\[
 \Theta_3(a,b)=\E\{a_3(X_1,X_2,X_3)b_3(Y_1,Y_2,Y_3)\}.
\]
The direct estimator is
\begin{equation*}
 \widehat\Theta_{n,3}
 =\widehat\eta_{n,3,3}
 -3\widehat\eta_{n,3,2}
 +3\widehat\eta_{n,3,1}
 -\widehat\eta_{n,3,0},
\end{equation*}
which involves up to six distinct observations.  By
Theorem~\ref{thm:general-compression},
\begin{equation*}
 \widehat\Theta_{n,3}
 =\frac{1}{\binom n3-\binom n2}
   \sum_{i<j<k}
   (\cU_{n,3}A)_{ijk}(\cU_{n,3}B)_{ijk},
\end{equation*}
where each centered array is given by \eqref{eq:r3-formula}.  After the
arrays for the three-argument kernels have been formed, the exact centered
calculation uses $O(n^3)$ arithmetic operations instead of the $O(n^6)$
products in a literal enumeration over pairs of triples.

\subsection{Covariance and estimation of canonical components}
\label{sec:residual-covariance-recovery}

The projection identity describes what the residual arrays estimate, as
well as how to compute their inner product.  Retain the assumptions of
Theorem~\ref{thm:general-compression}, and write
\[
 u=\cU_{n,r}A,\qquad v=\cU_{n,r}B,\qquad
 \sigma_a^2=\E\{a_r(X_{1:r})^2\},\qquad
 \sigma_b^2=\E\{b_r(Y_{1:r})^2\}.
\]
Here $X_{1:r}=(X_1,\ldots,X_r)$, and similarly for $Y_{1:r}$.

\begin{theorem}[Residual covariance]
\label{thm:residual-covariance}
The residual arrays have mean zero, and their cross-covariance matrix is
\begin{equation}\label{eq:residual-cross-covariance}
 \Cov(u,v)=\Theta_r(a,b)\cU_{n,r}.
\end{equation}
In particular, $\Cov(u)=\sigma_a^2\cU_{n,r}$ and
$\Cov(v)=\sigma_b^2\cU_{n,r}$.  These identities allow arbitrary dependence
between $X$ and $Y$ within each observation.
\end{theorem}

\noindent The proof is given in
Appendix~\ref{app:proof-residual-covariance}.  It uses the fact that the
projection removes every lower-order Hoeffding contribution, while
canonical components on distinct $r$-subsets are uncorrelated.  Taking the
trace in \eqref{eq:residual-cross-covariance} recovers
$\E\inner{u}{v}=d_{n,r}\Theta_r(a,b)$.  Thus the residual dimension is also
the trace factor in the covariance identity.  In an orthonormal basis of
$\cH_{n,r}$, the coordinates of $u$ have common variance $\sigma_a^2$ and
are uncorrelated; they need not be independent or Gaussian.

\begin{corollary}[Estimating a canonical component]
\label{cor:general-conditional}
For a fixed $I\in\cI_{n,r}$, put
$\lambda_{n,r}=d_{n,r}/\binom nr=(n-2r+1)/(n-r+1)$.  Then
\begin{equation}\label{eq:general-conditional}
 \E\left[u_I\given X_I\right]=\lambda_{n,r}a_r(X_I).
\end{equation}
The rescaled residual $\widehat a_{r,I}=u_I/\lambda_{n,r}$ has the
observable representation
\begin{equation}\label{eq:complement-recovery}
 \widehat a_{r,I}
 =\sum_{S\subseteq I}
   \frac{(-1)^{r-|S|}}{\binom{n-r}{r-|S|}}
   \sum_{\substack{K\subseteq[n]\setminus I\\|K|=r-|S|}}
     a(X_{S\cup K}).
\end{equation}
The mean-square errors of the unscaled and rescaled residuals are
\begin{align}
 \E\{(u_I-a_r(X_I))^2\}
 &=\frac{r}{n-r+1}\sigma_a^2,
 \label{eq:unscaled-recovery-mse}\\
 \E\{(\widehat a_{r,I}-a_r(X_I))^2\}
 &=\frac{r}{n-2r+1}\sigma_a^2.
 \label{eq:rescaled-recovery-mse}
\end{align}
\end{corollary}

\noindent The proof is given in
Appendix~\ref{app:proof-general-conditional}.  Formula
\eqref{eq:complement-recovery} replaces each integration in the population
inclusion--exclusion formula by an average over observations outside $I$,
using distinct labels.  For $r=1$, it becomes
\[
 \widehat a_{1,\{i\}}
 =a(X_i)-\frac{1}{n-1}\sum_{j\ne i}a(X_j).
\]
The expectation in \eqref{eq:general-conditional} averages over those
outside observations with $X_I$ fixed.  A realized residual therefore need
not equal $\lambda_{n,r}a_r(X_I)$.
For $r=2$, the shrinkage factor is
$\lambda_{n,2}=(n-3)/(n-1)$, recovering the pairwise conditional-mean
identity in Section~\ref{sec:pairwise-hoeffding}.

The rescaled residual is conditionally unbiased.  When $\sigma_a^2>0$,
removing this conditional bias increases variation enough to raise the
unconditional mean-square error.  This is a bias--variance tradeoff in
estimating the random component $a_r(X_I)$.  More generally, for every real constant
$c$,
\begin{equation}\label{eq:scalar-recovery-risk}
 \E\{(cu_I-a_r(X_I))^2\}
 =\sigma_a^2\{1-\lambda_{n,r}+\lambda_{n,r}(c-1)^2\}.
\end{equation}
When $\sigma_a^2>0$, the unscaled choice $c=1$ minimizes this mean-squared
error among constant multiples of the residual.  At $n=2r$, the errors divided by
$\sigma_a^2$ are $r/(r+1)$ for the unscaled residual and $r$ for the
rescaled residual.  These ratios compare estimation error with the
variance of the component being estimated.  Along sequences with
$n\to\infty$ and $n\ge2r$, both ratios tend to zero if and only if
$r/n\to0$.  For fixed $r$,
both are $r/n+O(n^{-2})$.  If $\sigma_a^2=0$, both errors are zero.

\subsection{Exact moments under independence}
\label{sec:independence-moments}

We now impose the stronger condition $X\perp Y$.  The arrays built from
the entire $X$-sample and the entire $Y$-sample are then independent.
This condition is essential for the results below; the single equality
$\Theta_r(a,b)=0$ is not sufficient.  For example, take $r=1$, the
identity kernels, $X\sim N(0,1)$, and $Y=X^2-1$.  Then
$\Theta_1=\Cov(X,Y)=0$, although $Y$ is determined by $X$.  The statistic
$T_{n,1}^{\mathrm{dir}}$ is the usual unbiased sample covariance, whose variance here is
$10/n+2/\{n(n-1)\}$, rather than the independence value $2/(n-1)$.
Appendix~\ref{app:proof-independence-moments} verifies this calculation.
Write
\[
 M=\binom nr,\qquad d=d_{n,r},\qquad
 T_{n,r}=\frac{\inner{u}{v}}d,
\]
so that $T_{n,r}=T_{n,r}^{\mathrm{dir}}$.

\begin{proposition}[Null variance and comparison with population components]
\label{prop:null-oracle-moments}
Under $X\perp Y$ and the assumptions of
Theorem~\ref{thm:general-compression},
\begin{equation}\label{eq:exact-null-variance}
 \E T_{n,r}=0,\qquad
 \Var(T_{n,r})=\frac{\sigma_a^2\sigma_b^2}{d},\qquad
 \E\widehat V_0=\Var(T_{n,r}),\quad
 \widehat V_0:=\frac{\norm{u}^2\norm{v}^2}{d^3}.
\end{equation}
For comparison, consider the statistic based on the generally unknown
population components:
\begin{equation}\label{eq:oracle-component-statistic}
 V_{n,r}=\frac1M\sum_{I\in\cI_{n,r}}a_r(X_I)b_r(Y_I).
\end{equation}
It is the order-$r$ term in the sample Hoeffding decomposition of
$T_{n,r}$, and
\begin{align}
 \Var(V_{n,r})=\Cov(T_{n,r},V_{n,r})
 &=\frac{\sigma_a^2\sigma_b^2}{M},
 \label{eq:oracle-moments}\\
 \E\{(T_{n,r}-V_{n,r})^2\}
 &=\sigma_a^2\sigma_b^2\left(\frac1d-\frac1M\right).
 \label{eq:oracle-discrepancy}
\end{align}
For fixed $r$ and fixed kernels, $n^{r/2}(T_{n,r}-V_{n,r})\to0$ in $L^2$.
\end{proposition}

\noindent The proof is given in
Appendix~\ref{app:proof-independence-moments}.  Although the direct estimator
can involve $2r$ observations, under independence every Hoeffding term
involving fewer than $r$ paired observations vanishes.  The statistic $V_{n,r}$
is its first possibly nonzero term.  It is generally unavailable because
the population canonical components are unknown, but
\eqref{eq:oracle-discrepancy} shows exactly how closely the observable
statistic approximates it.  When
$\sigma_a^2\sigma_b^2>0$, the variance of the discrepancy divided by the
variance of $V_{n,r}$ is exactly $r/(n-2r+1)$.

The same covariance calculation gives an exact, but limited, inferential
meaning to the conventional regression standard error.

\begin{corollary}[Conditional second moments of the residual regression]
\label{cor:conditional-ols}
Under the assumptions of Proposition~\ref{prop:null-oracle-moments},
condition on $X_{1:n}$ and suppose $u\ne0$ and $d>1$.  Define
\[
 \widehat\beta=\frac{\inner{u}{v}}{\norm{u}^2},\qquad
 \mathrm{RSS}=\norm{v-\widehat\beta u}^2.
\]
Then
\begin{equation}\label{eq:conditional-ols-moments}
 \E(\widehat\beta\mid X_{1:n})=0,\qquad
 \Var(\widehat\beta\mid X_{1:n})=\frac{\sigma_b^2}{\norm{u}^2},\qquad
 \E\left(\frac{\mathrm{RSS}}{d-1}\given X_{1:n}\right)=\sigma_b^2.
\end{equation}
Consequently, $\mathrm{RSS}/\{(d-1)\norm{u}^2\}$ is conditionally
unbiased for the slope variance.
\end{corollary}

\noindent The proof is also in
Appendix~\ref{app:proof-independence-moments}.  This statement concerns the
squared standard error, not its square root.  It neither makes the
residual coordinates independent nor gives the studentized slope a
Student $t$ law.  A zero residual norm or zero residual sum of squares can
also leave the studentized statistic undefined.

\subsection{Permutation covariance and a non-Gaussian null example}
\label{sec:permutation-moments}

Permutation moments give a complementary finite-sample interpretation of
$d_{n,r}$.  Here the arrays are fixed, and randomness comes only from
relabeling observations.

\begin{theorem}[Permutation covariance]
\label{thm:permutation-covariance}
Let $r\ge1$, $n\ge2r$, and let $A,B\in\cV_{n,r}$ be arbitrary real arrays.  Put
$u=\cU_{n,r}A$ and $v=\cU_{n,r}B$.  For a uniform permutation $\pi$ of
$[n]$, write $v^\pi_I=v_{\pi(I)}$.  Then
\begin{equation}\label{eq:permutation-covariance}
 \E_\pi v^\pi=0,\qquad
 \E_\pi\{v^\pi(v^\pi)^\mathsf T\}
 =\frac{\norm{v}^2}{d_{n,r}}\cU_{n,r}.
\end{equation}
Consequently, including when either residual is zero,
\begin{equation}\label{eq:permutation-statistic-variance}
 \Var_\pi\left\{\frac{\inner{u}{v^\pi}}{d_{n,r}}\right\}
 =\frac{\norm{u}^2\norm{v}^2}{d_{n,r}^3}
 =\widehat V_0.
\end{equation}
If $u\ne0$ and $v\ne0$, their permuted residual correlation satisfies
\begin{equation}\label{eq:permutation-correlation-moments}
 R^\pi=\frac{\inner{u}{v^\pi}}{\norm{u}\norm{v}},\qquad
 \E_\pi R^\pi=0,\qquad
 \E_\pi\{(R^\pi)^2\}=\frac1{d_{n,r}}.
\end{equation}
\end{theorem}

\noindent The proof is given in
Appendix~\ref{app:proof-permutation-covariance}.  It uses the zero margins
and the overlap inversion already derived above.  No distributional
assumptions on the fixed arrays are needed.  Thus the nonnegative null
variance estimator in \eqref{eq:exact-null-variance} is also the exact
conditional permutation variance of the statistic.  Under independence,
the joint law is invariant under these relabelings of the
$Y$-observations, which justifies permutation calibration.  The displayed moments alone do
not specify its tail probabilities or imply a Student reference law.

\paragraph{A simple non-Gaussian null limit.}
Let $X$ and $Y$ be independent symmetric Rademacher variables and use
ordinary Euclidean distance, so $|X_i-X_j|=1-X_iX_j$ and similarly for $Y$.
U-centering removes the constant arrays, and the two minus signs cancel
in the residual correlation.  For this fixed distribution and fixed pair
kernels, the dependence statistic in \eqref{eq:edge-regression-t} obeys
\begin{equation}\label{eq:rademacher-null-limit}
 T_{\mathrm{dep}}\ \Longrightarrow\
 \frac{\chi_1^2-1}{\sqrt2},\qquad n\to\infty.
\end{equation}
Here $T_{\mathrm{dep}}$ is the slope divided by its conventional standard
error from Corollary~\ref{cor:conditional-ols} with $r=2$.
For this convergence statement, set $T_{\mathrm{dep}}=0$ whenever a
residual norm vanishes or $|R_U|=1$; the probability of these events tends
to zero.  The limit differs from the standard normal limit of
$t_{d_{n,2}-1}$.  Appendix~\ref{app:proof-rademacher-limit} derives
\eqref{eq:rademacher-null-limit} by comparison with $V_{n,r}$.  Thus exact
regression formulas and unbiased squared standard errors can coexist with
a non-Gaussian null limit.

\subsection{Application: residual representation for the variance of the
highest-order Hoeffding component}
\label{subsec:variance-components}

Assume $n\ge 2r$.  Let $Z_1,\ldots,Z_n$ be i.i.d.\ copies of a random variable
$Z$ taking values in $\mathcal Z$.  For $I=\{i_1<\cdots<i_r\}$, write
$Z_I=(Z_{i_1},\ldots,Z_{i_r})$.  Let $h:\mathcal Z^r\to\R$ be a symmetric
square-integrable kernel, and define
\begin{equation*}
 U_n(h)=\binom nr^{-1}\sum_{I\in\cI_{n,r}}h(Z_I).
\end{equation*}
For $r=1$, removing the lower-order effect means subtracting the sample mean
$\overline h=n^{-1}\sum_i h(Z_i)$.  The residual mean square is
\[
 \frac{1}{n-1}\sum_{i=1}^n\{h(Z_i)-\overline h\}^2,
\]
the usual unbiased sample variance of $h(Z)$.  The construction below
extends this idea to the highest-order component of an $r$-argument kernel.

Write $h_1,\ldots,h_r$ for its Hoeffding components and
\begin{equation*}
 \sigma_j^2=\E\{h_j(Z_1,\ldots,Z_j)^2\}.
\end{equation*}
Orthogonality of the Hoeffding decomposition gives the exact finite-sample
variance formula
\begin{equation*}
 \Var\{U_n(h)\}
 =\sum_{j=1}^r\frac{\binom rj^2}{\binom nj}\sigma_j^2.
\end{equation*}
This variance decomposition is classical; see
\citet{Hoeffding1948,Serfling1980}.  Here we focus only on $\sigma_r^2$, the
variance of the highest-order Hoeffding component.

\begin{samepage}
Let $H\in\cV_{n,r}$ be the kernel array with $H_I=h(Z_I)$.  Taking $A=B=H$
in Theorem~\ref{thm:general-compression} identifies the corresponding
overlap-based estimator with the normalized squared residual norm
\begin{equation*}
 \widehat\sigma_{r,n}^2
 :=T_{n,r}^{\mathrm{dir}}
 =\frac{1}{d_{n,r}}\norm{\cU_{n,r}H}^2.
\end{equation*}
Its expectation is $\Theta_r(h,h)=\sigma_r^2$, and it is nonnegative.
\end{samepage}
The direct overlap formula is an alternating sum, but it is the same
estimator as this squared residual norm.  Its nonnegativity is therefore an
exact identity for the same kernel and sample; no truncation is needed.

For an order-three U-statistic,
\begin{equation*}
 \Var\{U_n(h)\}
 =\frac{9}{n}\sigma_1^2
  +\frac{9}{\binom n2}\sigma_2^2
  +\frac{1}{\binom n3}\sigma_3^2,
\end{equation*}
and
\begin{equation*}
 \widehat\sigma_{3,n}^2
 =\frac{1}{\binom n3-\binom n2}
   \sum_{i<j<k}(\cU_{n,3}H)_{ijk}^2.
\end{equation*}
If $h_1=h_2=0$---equivalently, if $h-\E h$ is canonical of order three---then
$\Var\{U_n(h)\}=\sigma_3^2/\binom n3$, so
$\widehat\sigma_{3,n}^2/\binom n3$ is a nonnegative unbiased variance
estimator.

\paragraph{A PSD covariance estimator.}
The scalar variance estimator extends to several kernels evaluated on the
same sample.  Let $h^{(1)},\ldots,h^{(p)}$ be symmetric square-integrable
order-$r$ kernels, and let $H^{(j)}_I=h^{(j)}(Z_I)$.  Define
\begin{equation}\label{eq:psd-component-covariance}
 \widehat\Gamma_{jk}
 =\frac1{d_{n,r}}
   \inner{\cU_{n,r}H^{(j)}}{\cU_{n,r}H^{(k)}},
 \qquad 1\le j,k\le p.
\end{equation}
This is a Gram matrix divided by a positive constant, so
$\widehat\Gamma\succeq0$.  Theorem~\ref{thm:residual-covariance} gives
\[
 \E\widehat\Gamma_{jk}
 =\E\{h_r^{(j)}(Z_{1:r})h_r^{(k)}(Z_{1:r})\}.
\]
Hence it is an unbiased PSD estimator of the covariance
matrix of the highest-order components, with
$\rank(\widehat\Gamma)\le\min(p,d_{n,r})$.  Singular matrices are allowed;
no independence among the kernels is required.  This target is not the
full covariance matrix of the original, generally noncanonical U-statistics.

\paragraph{Why the sample-size and positivity qualifications matter.}
The condition $n\ge2r$ is necessary for unbiased estimation in general,
not only for the projection construction.  For the product kernel
$h(z_1,\ldots,z_r)=\prod_{i=1}^rz_i$ under a Bernoulli$(p)$ law, the target
is
\begin{equation}\label{eq:sample-size-sharpness-target}
 \sigma_r^2=\{p(1-p)\}^r.
\end{equation}
No statistic of $n<2r$ observations can have this expectation for every
$p\in(0,1)$.

Nonnegative unbiased estimation is also special to the highest component.
For the same product kernel under a law on $\{-1,1\}$ with mean $\mu$,
the order-$j$ variance is
\begin{equation}\label{eq:lower-component-positivity-target}
 \sigma_j^2=\mu^{2(r-j)}(1-\mu^2)^j.
\end{equation}
For each $1\le j<r$ and every finite sample size, no nonnegative statistic
can be unbiased for this target for all $\mu\in(-1,1)$.  These are
impossibility statements over the indicated distribution families; they
do not rule out estimators under additional restrictions.  The arguments
use classical unbiased-estimation arguments for distributions with
finitely many possible outcomes
\citep{Halmos1946}; the short proofs are given in
Appendix~\ref{app:proof-variance-sharpness}.

\subsection{\texorpdfstring{{A spectral example based on four-argument matching kernels}}{A spectral example based on four-argument matching kernels}}\label{sec:example}
As a second application, we ask what higher-order pairings reveal about
kernel dependence.  Ordinary HSIC measures total squared cross-covariance.
The construction below also describes how that total is distributed across
singular directions: two dependence structures can have the same HSIC but
differ in whether their cross-covariance is concentrated in a few directions
or spread across many.

\subsubsection{Population spectral quantities}
Let $k$ and $\ell$ be bounded measurable real-valued positive-definite kernels with
separable RKHSs $\mathcal F$ and $\mathcal G$, respectively.  Here separability
means that each space has a finite or countable orthonormal basis.  For four
arguments, write $k_{ij}=k(x_i,x_j)$ and
$\ell_{ij}=\ell(y_i,y_j)$, and define the observable symmetric matching
kernels
\begingroup
\allowdisplaybreaks[0]
\begin{align}
 m_k(x_1,x_2,x_3,x_4)
 &=\frac13\{k_{12}k_{34}+k_{13}k_{24}+k_{14}k_{23}\},\label{eq:matching-k}\\
 m_\ell(y_1,y_2,y_3,y_4)
 &=\frac13\{\ell_{12}\ell_{34}+\ell_{13}\ell_{24}
                  +\ell_{14}\ell_{23}\}.\label{eq:matching-l}
\end{align}
\endgroup
The three terms correspond to the three ways to partition four labels into
two disjoint pairs (perfect matchings).  Let $k_c$ and $\ell_c$ denote the
population-centered kernels.  For example,
\[
 k_c(x,x')
 =k(x,x')-\E k(x,X)-\E k(X,x')+\E k(X,X'),
\]
where $X$ and $X'$ are independent with distribution $P_X$; define
$\ell_c$ analogously.  Write
$k_{c,ij}=k_c(x_i,x_j)$ and $\ell_{c,ij}=\ell_c(y_i,y_j)$.  The
fourth Hoeffding components of the raw matching kernels are
\begin{align*}
 (m_k)_4(x_1,x_2,x_3,x_4)
 &=\frac13\{k_{c,12}k_{c,34}+k_{c,13}k_{c,24}
                  +k_{c,14}k_{c,23}\},\\
 (m_\ell)_4(y_1,y_2,y_3,y_4)
 &=\frac13\{\ell_{c,12}\ell_{c,34}+\ell_{c,13}\ell_{c,24}
                  +\ell_{c,14}\ell_{c,23}\}.
\end{align*}
Appendix~\ref{app:rkhs} verifies these projections.

Before introducing operators, consider finite-dimensional kernel feature
vectors.  Subtracting their means gives centered vectors $\phi_c(X)$ and
$\psi_c(Y)$, whose cross-covariance matrix is
\[
 C_{XY}=\E\{\phi_c(X)\psi_c(Y)^\mathsf T\}.
\]
Each entry is the covariance of one feature of $X$ with one feature of $Y$.
Population HSIC is the sum of the squared entries of this matrix, or
equivalently its squared Frobenius norm.  Its singular values describe how
this total is distributed across feature directions.  The operator notation
below extends these matrix quantities to possibly infinite-dimensional
feature spaces: the transpose becomes the adjoint, and the Frobenius norm
becomes the Hilbert--Schmidt norm.

For the general RKHS setting, let
\[
 \phi_c(x)=k(x,\cdot)-\E k(X,\cdot),
 \qquad
 \psi_c(y)=\ell(y,\cdot)-\E \ell(Y,\cdot)
\]
be the centered feature maps.  For $u\in\mathcal F$ and $v,g\in\mathcal G$,
define the rank-one operator $u\otimes v:\mathcal G\to\mathcal F$ by
$(u\otimes v)g=u\inner{v}{g}_{\mathcal G}$.  For real vectors this is
the outer-product matrix $uv^{\mathsf T}$, since
$(uv^{\mathsf T})g=u(v^{\mathsf T}g)$.  Define the cross-covariance
operator
\[
 C_{XY}=\E\{\phi_c(X)\otimes\psi_c(Y)\}:\mathcal G\to\mathcal F.
\]
Let $C_{XY}^{\ast}$ denote its adjoint, characterized by
\[
 \inner{C_{XY}g}{f}_{\mathcal F}
 =\inner{g}{C_{XY}^{\ast}f}_{\mathcal G}.
\]
For any orthonormal basis $(f_q)$ of $\mathcal G$, define the
Hilbert--Schmidt norm by
\[
 \norm{C_{XY}}_{\mathrm{HS}}^2
 =\sum_q\norm{C_{XY}f_q}_{\mathcal F}^2,
\]
and let $s_1,s_2,\ldots$ be the singular values of $C_{XY}$.  Then population
HSIC is
\begin{equation}\label{eq:H2-spectrum}
 H_2:=\Theta_2(k,\ell)
 =\norm{C_{XY}}_{\mathrm{HS}}^2
 =\sum_{j\ge1}s_j^2;
\end{equation}
see \citet{GrettonEtAl2005}.  The pairing of the fourth Hoeffding
components of the matching kernels is
\[
 H_4:=\Theta_4(m_k,m_\ell)
 =\E\{(m_k)_4(X_1,\ldots,X_4)
       (m_\ell)_4(Y_1,\ldots,Y_4)\}.
\]

Expanding the two three-term averages gives nine pairs of matchings.  For each
of the three pairs using the same matching, the two matched edges involve
disjoint groups of observations, so the expectation factors as $H_2^2$.  For
each of the six different matching pairs, their edges form a cycle through
all four labels, called a four-cycle.  By
relabeling, all six four-cycle terms have the same expectation, and
Appendix~\ref{app:rkhs} shows that
\[
 \E\{k_{c,12}k_{c,34}\ell_{c,13}\ell_{c,24}\}
 =\tr\{(C_{XY}^{\ast}C_{XY})^2\}
 =\sum_{j\ge1}s_j^4.
\]
Combining the three terms from identical matchings and the six four-cycle
terms with their common
factor $1/9$ gives
\begin{equation}\label{eq:H4-spectrum}
 H_4=\frac13H_2^2+\frac23\sum_{j\ge1}s_j^4.
\end{equation}
Both terms on the right are nonnegative, so $H_4=0$ if and only if
$H_2=0$, or equivalently $C_{XY}=0$.  If, in addition, $\mathcal X$ and
$\mathcal Y$ are topological spaces with countable bases of open sets
(second-countable spaces), equipped with their Borel sigma-algebras, and
$k$ and $\ell$ are characteristic, these conditions
are equivalent to $X\perp Y$; see
\citet[Theorem~3(i)]{SzaboSriperumbudur2018}.  Characteristic kernels are
those whose mean embeddings distinguish probability distributions.
Without such an additional condition, zero cross-covariance in the chosen
feature spaces need not imply independence.
Equation~\eqref{eq:H4-spectrum} recovers the fourth spectral moment from the
two population pairings:
\[
 S_4:=\sum_{j\ge1}s_j^4=\frac{3H_4-H_2^2}{2}.
\]
When $H_2>0$, define the effective rank by
\begin{equation}\label{eq:spectral-participation}
 d_{\mathrm{eff}}
 :=\frac{(\sum_j s_j^2)^2}{\sum_j s_j^4}
 =\frac{2H_2^2}{3H_4-H_2^2}.
\end{equation}
To interpret it, write $w_j=s_j^2/H_2$.  These nonnegative weights sum to one
and give the fraction of total squared cross-covariance in each singular
direction.  Then $d_{\mathrm{eff}}=1/\sum_j w_j^2$: it equals one when all
the weight is in a single direction and $q$ when it is spread equally across
$q$ directions.
This ratio is invariant to the choice of bases in the two RKHSs, but it
depends on the chosen kernels.  It should therefore be interpreted as the
effective number of dependence directions in the chosen RKHS geometry, not
as an intrinsic dimension of the joint distribution.  At independence,
$H_2=H_4=0$ and the ratio is undefined.

The distinction can be seen at fixed $H_2=\omega^2>0$.  If the spectrum is
rank one, with $s_1=\omega$, then
\[
 H_4=\omega^4,\qquad d_{\mathrm{eff}}=1.
\]
If instead $s_1=\cdots=s_q=\omega/\sqrt q$ and all remaining singular values
vanish, ordinary HSIC is still $\omega^2$, whereas
\[
 H_4=\omega^4\left(\frac13+\frac{2}{3q}\right),
 \qquad d_{\mathrm{eff}}=q.
\]
Thus $H_2$ records total squared cross-covariance, while $H_4$ distinguishes
whether that total is concentrated in one spectral direction or spread
across several directions.

\subsubsection{Unbiased spectral estimation}

We observe i.i.d.\ pairs $O_i=(X_i,Y_i)$ and keep the bounded kernels
$k$ and $\ell$ fixed throughout this subsection.
Our target is $S_4$, the sum of the fourth powers of the singular values of
the cross-covariance operator.  The matching construction estimates
$H_4=(H_2^2+2S_4)/3$, so isolating $S_4$ also requires an estimate of
$H_2^2$.  Squaring an unbiased HSIC estimator generally introduces bias.  We instead
average products of HSIC statistics evaluated on disjoint four-subsets.

Let $\widehat H_2=T_{n,2}^{\mathrm{dir}}(k,\ell)$ and
$\widehat H_4=T_{n,4}^{\mathrm{dir}}(m_k,m_\ell)$.
For a four-subset $I$, let $g(O_I)$ be the unbiased HSIC statistic computed
on those four observations.  Equivalently, $g$ is the average over all
$4!$ permutations of the arguments in
\[
 k(X_1,X_2)\ell(Y_1,Y_2)
 -2k(X_1,X_2)\ell(Y_1,Y_3)
 +k(X_1,X_2)\ell(Y_3,Y_4).
\]
For $n\ge8$, define
\begin{equation}\label{eq:spectral-Q}
 \widehat Q
 =\frac{1}{\binom n4\binom{n-4}4}
 \sum_{\substack{I,J\in\cI_{n,4}\\I\cap J=\varnothing}}
 g(O_I)g(O_J).
\end{equation}
The sum is over ordered pairs of disjoint four-subsets.

\begin{proposition}[Unbiased fourth spectral-moment estimation]
\label{prop:spectral-unbiased}
Under the bounded-kernel assumptions above and $n\ge8$,
$\E\widehat H_2=H_2$, $\E\widehat H_4=H_4$, and
$\E\widehat Q=H_2^2$.  Consequently,
\begin{equation}\label{eq:S4-unbiased}
 \widehat S_4=\frac{3\widehat H_4-\widehat Q}{2},
 \qquad \E\widehat S_4=S_4.
\end{equation}
The plug-in version
$\widetilde S_4=(3\widehat H_4-\widehat H_2^2)/2$ instead satisfies
\begin{equation}\label{eq:S4-plugin-bias}
 \E\widetilde S_4-S_4=-\frac12\Var(\widehat H_2).
\end{equation}
All five estimators are consistent for their corresponding population
quantities as $n\to\infty$ with the kernels fixed.
\end{proposition}

\begin{proof}
The projection representation gives the first two expectations.
For disjoint $I,J$, the two copies of $g$ are independent and each has
expectation $H_2$, proving $\E\widehat Q=H_2^2$.
Equations~\eqref{eq:S4-unbiased} and~\eqref{eq:S4-plugin-bias} follow from
\eqref{eq:H4-spectrum} and
$\E\widehat H_2^2=H_2^2+\Var(\widehat H_2)$.
The fixed-order U-statistic representations of
$\widehat H_2,\widehat H_4,\widehat Q$, with bounded kernels, give consistency;
the two estimators of $S_4$ then converge by their displayed formulas.
\end{proof}

The unbiased estimators in this subsection can be negative.  This differs
from the same-array squared norm in the preceding variance-component
application.  In fact, no estimator can be both nonnegative and universally
unbiased for these spectral targets at a finite sample size.
Moreover, the threshold $n\ge8$ for $H_4$, $H_2^2$, and $S_4$ is necessary
in general when population centering is unknown; $H_2$ itself requires only
four observations.  Appendix~\ref{app:spectral-computation} gives elementary
arguments using distributions with finitely many possible outcomes for
both limitations.  We retain negative estimates to preserve unbiasedness.

The population ratio $d_{\mathrm{eff}}$ remains useful for interpreting the
examples above.  Estimation of its uncertainty is a separate problem: the
ratio is undefined at independence, and its first derivative as a functional
of $C_{XY}$ vanishes when the nonzero singular values are all equal,
including a rank-one spectrum.  Appendix~\ref{app:spectral-computation}
verifies this derivative statement.  These features are reasons for using
the ratio for population interpretation here, without developing inference
for it.

\subsubsection{Cubic matrix computation}

Formula~\eqref{eq:spectral-Q} defines an estimator rather than an efficient
algorithm.  Its apparent eight-label calculation can be reorganized because
every summand is a product of pairwise kernel entries.  The following
calculation exploits this additional structure of matching kernels.

\begin{proposition}[Three-product spectral computation]
\label{prop:cubic-spectral}
Let $n\ge8$, form the two Gram matrices, and set their diagonal entries to zero.
Apply ordinary pairwise U-centering to both, and call the resulting matrices
$K,L$, so $K\one=L\one=0$.  From the three ordinary matrix products
\[
 K^2,\qquad L^2,\qquad KL
\]
and $O(n^2)$ additional arithmetic, one can evaluate exactly the same
$\widehat H_4$ and $\widehat Q$ defined above, and hence $\widehat S_4$.
The calculation uses $O(n^3)$ arithmetic and $O(n^2)$ storage after the
pairwise kernel entries have been formed.
\end{proposition}

The complete formulas and proof are given in
Appendix~\ref{app:spectral-computation}.  Pairwise U-centering changes an
off-diagonal entry only by additive endpoint terms.  Within each four-subset,
these terms cancel in $g$; in the matching array, their contributions
involve at most three labels and vanish under fourth-order U-centering.
Thus this preprocessing preserves the estimators.  Zero row sums then
reduce the subset sums to nine scalar quantities obtained by summing products of matrix entries.
No array indexed by triples or quadruples needs to be constructed.

The algebraic identities hold for arbitrary real symmetric zero-diagonal matrices;
positive definiteness is used for the population spectral interpretation.
The matrix products need not commute: in particular,
$\tr(KLKL)=\sum_{ij}(KL)_{ij}(KL)_{ji}$ generally differs from
$\|KL\|_F^2$.  Sample U-centering here must also be distinguished from the
population centering used to define $k_c,\ell_c$.

\subsubsection{Numerical illustration}

We use a model with fixed marginal distributions and known singular values.
Let $X$ be uniform on $\{-1,1\}^4$, and set $Y_j=X_jE_j$, where the signs
$E_1,\ldots,E_4$ are independent of one another and of $X$, with
$\E E_j=\rho_j$.  Both $X$ and $Y$ are uniform on $\{-1,1\}^4$.
For linear kernels $k(x,x')=x^\mathsf Tx'$ and
$\ell(y,y')=y^\mathsf Ty'$, the cross-covariance matrix is
$\diag(\rho_1,\ldots,\rho_4)$.  We compare
\[
 \rho_{\mathrm{concentrated}}=(0.6,0,0,0),\qquad
 \rho_{\mathrm{distributed}}=(0.3,0.3,0.3,0.3).
\]
Both models have $H_2=0.36$.  Their respective values are
$H_4=0.1296,0.0648$ and $S_4=0.1296,0.0324$.
Thus the example varies spectral concentration while holding total HSIC
and both marginal laws fixed.

For each model, we use sample sizes $n=16,32,64,128$, with 1,000 independent
replications at each size.  Figure~\ref{fig:spectral-illustration} compares estimator
means with the known targets.  Error bars are twice the Monte Carlo standard
error of the simulation mean; they are not confidence intervals for a
single-data-set estimate.  Table~\ref{tab:spectral-rmse} reports the root
mean-squared error (RMSE) for the unbiased and plug-in $S_4$ estimators.
Negative estimates are retained.
The bias identity~\eqref{eq:S4-plugin-bias} predicts a downward shift for
the plug-in estimator, while unbiasedness alone does not determine which
estimator has smaller RMSE.  For the distributed spectrum at
$n=16$, the plug-in mean is $-0.091$ despite a positive target $S_4=0.0324$.

\begin{figure}[htbp]
 \centering
 \includegraphics[width=\linewidth]{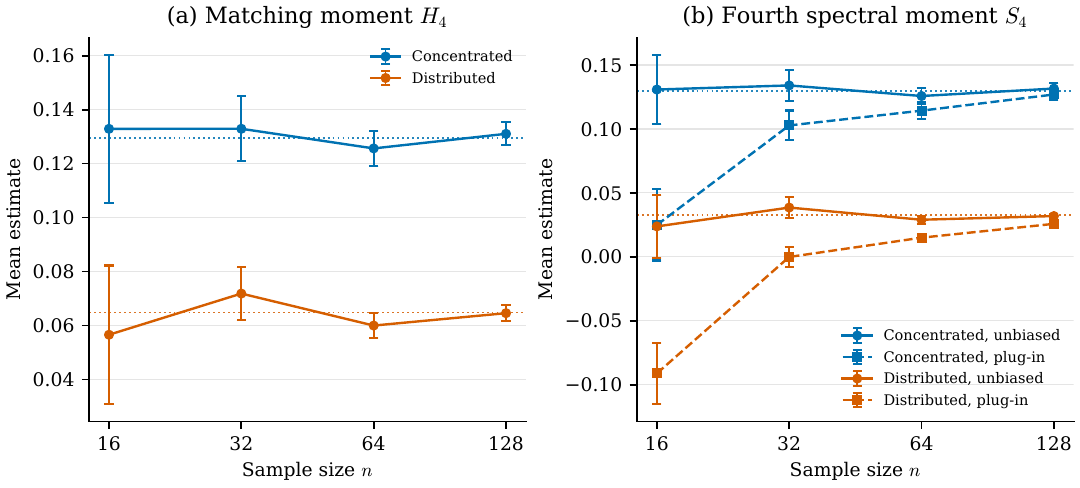}
 \caption{Spectral estimation for two models with the same population HSIC.
 Points are means over 1,000 replications; error bars show twice their Monte
 Carlo standard errors.  Horizontal lines mark the population targets.
 The $S_4$ panel includes both the unbiased and plug-in estimators.}
 \label{fig:spectral-illustration}
\end{figure}

\begin{table}[htbp]
 \centering
 \caption{RMSE of the two $S_4$ estimators over
 1,000 replications.  Smaller values indicate less estimation error.}
 \label{tab:spectral-rmse}
 \begin{tabular}{rcccc}
\toprule
 & \multicolumn{2}{c}{Concentrated spectrum} & \multicolumn{2}{c}{Distributed spectrum}\\
$n$ & Unbiased & Plug-in & Unbiased & Plug-in\\
\midrule
16 & 0.4273 & 0.4512 & 0.3870 & 0.3971\\
32 & 0.1879 & 0.1863 & 0.1346 & 0.1290\\
64 & 0.1022 & 0.1017 & 0.0518 & 0.0514\\
128 & 0.0675 & 0.0671 & 0.0291 & 0.0286\\
\bottomrule
\end{tabular}

\end{table}

In both models, RMSE decreases as the sample size grows.  The unbiased
estimator has lower reported RMSE at $n=16$, while the plug-in estimator
has slightly lower RMSE at $n=32,64,128$.  Thus eliminating bias does not
necessarily reduce RMSE, which reflects both bias and variability.
The distributed spectrum has a smaller target, so its lower absolute
RMSE does not by itself imply better accuracy relative to the target.
The illustration concerns estimation of spectral quantities and makes no
claim about comparative testing power.

\section{Discussion}
\label{sec:discussion}
U-centering is least-squares residualization on sample subsets.  Its
normalized residual cross-product equals the direct overlap estimator of a
highest-order Hoeffding pairing, and its covariance identifies the statistical
role of the residual dimension.  The same geometry gives mean-squared
errors for estimating the highest-order components, PSD covariance estimates
for these components, and exact second moments under independence and
under permutation.

The distinction between these statements matters in applications.
Residual coordinates need not be independent, and an unbiased squared
standard error does not supply a Student reference law.  Exact permutation
tests use exchangeability and the permutation distribution itself.  In the
two-sample problem, the common-endpoint and fully interacted regressions
estimate the same slope but remove different endpoint variation; their mean
squares have the same pooled permutation average.

For general $r$-subset arrays, the scalar pairing costs $O(n^r)$ operations
at fixed $r$.  The matching-kernel application has additional structure:
after pairwise U-centering, three matrix products suffice for exactly
unbiased fourth spectral-moment estimation.  The numerical illustration
shows how these quantities distinguish spectra having the same HSIC.
The sample-size and nonnegativity boundaries delimit what can hold
universally.  Effective-rank inference near independence or at spectra
with equal nonzero singular values requires separate analysis.

\appendix

\section{The higher-order ANOVA projection}
\label{app:higher-projection-proof}
\begin{proof}
\smallskip
\noindent\emph{Symmetry and zero margins.}
Each operator $D_{j,r}^\ast D_{j,r}$ has a symmetric matrix, so
$\cU_{n,r}$ does too.  We now compute its $(r-1)$-margins.

Fix $S\in\binom{[n]}{r-1}$.  For $0\le j\le r$, write
\[
 m_j(J)=(D_{j,r}A)(J),
 \qquad
 T_j(S)=\sum_{\substack{J\subseteq S\\|J|=j}}m_j(J),
\]
with $T_{-1}(S)=0$.  The contribution of
$D_{j,r}^\ast D_{j,r}A$ to the $S$-margin is
\begin{equation}\label{eq:Lj-def}
 L_j(S)
 :=\sum_{i\notin S}
   \sum_{\substack{J\subseteq S\cup\{i\}\\|J|=j}}m_j(J).
\end{equation}
Split the inner $j$-sets according to whether they contain the new label $i$.
Those not containing $i$ are the $j$-subsets of $S$; each is repeated for all
$n-r+1$ choices of $i\notin S$.  Their total contribution is
$(n-r+1)T_j(S)$.

For $j=0$, there is no second part.  For $j\ge1$, a $j$-set containing
$i$ has the form $K\cup\{i\}$ with $K\subseteq S$ and $|K|=j-1$.
For a fixed such $K$, we evaluate the sum over labels outside $S$.  Since
$K\subseteq S$, the complement of $K$ has the disjoint decomposition
\[
 [n]\setminus K
 =\bigl([n]\setminus S\bigr)\,\dot\cup\,\bigl(S\setminus K\bigr).
\]
Consequently,
\begingroup
\interdisplaylinepenalty=10000
\begin{align}
 \sum_{i\notin S}m_j(K\cup\{i\})
 &=\sum_{i\notin K}m_j(K\cup\{i\})
   -\sum_{u\in S\setminus K}m_j(K\cup\{u\}) \notag\\
 &=\sum_{i\notin K}
   \sum_{\substack{I\in\cI_{n,r}\\I\supseteq K\cup\{i\}}}A_I
   -\sum_{u\in S\setminus K}m_j(K\cup\{u\}) \notag\\
 &=\sum_{\substack{I\in\cI_{n,r}\\I\supseteq K}}
   \sum_{i\in I\setminus K}A_I
   -\sum_{u\in S\setminus K}m_j(K\cup\{u\}) \notag\\
 &=(r-j+1)m_{j-1}(K)
   -\sum_{u\in S\setminus K}m_j(K\cup\{u\}).
 \label{eq:fixed-K-count}
\end{align}
\endgroup
The first equality separates the labels outside $K$ into those outside $S$
and those in $S\setminus K$.  The second equality expands the definition of
$m_j$, and the third reverses the order of summation.  For every fixed
$r$-set $I\supseteq K$, the inner sum contains one copy of $A_I$ for each
$i\in I\setminus K$.  Since $|K|=j-1$,
\[
 |I\setminus K|=r-(j-1)=r-j+1,
\]
so each $A_I$ is counted exactly $r-j+1$ times.  This gives
$(r-j+1)m_{j-1}(K)$.  Finally, subtracting the terms indexed by
$u\in S\setminus K$ removes precisely those labels that are outside $K$ but
still inside $S$, leaving only the labels $i\notin S$ required in
\eqref{eq:Lj-def}.

Next, sum \eqref{eq:fixed-K-count} over all $(j-1)$-subsets $K$ of $S$.
The first term becomes $(r-j+1)T_{j-1}(S)$.  In the final double sum, each
$j$-subset $J\subseteq S$ appears once for each choice of the removed element
$u\in J$, hence exactly $j$ times.  Combining the two parts yields
  \begin{equation*}
 L_j(S)
 =(n-r+1-j)T_j(S)+(r-j+1)T_{j-1}(S).
\end{equation*}

Because
\[
 D_{r-1,r}(\cU_{n,r}A)(S)=\sum_{j=0}^r c_{n,r,j}L_j(S),
\]
the coefficient of $T_j(S)$, for $0\le j\le r-1$, is
\[
 (n-r+1-j)c_{n,r,j}+(r-j)c_{n,r,j+1}.
\]
Setting these coefficients to zero yields
\[
 c_{n,r,j+1}
 =-\frac{n-r-j+1}{r-j}c_{n,r,j},
 \qquad 0\le j<r.
\]
This is the recurrence stated before Definition~\ref{def:higher-u}; the
coefficients used there satisfy it with $c_{n,r,r}=1$.  Since $T_r(S)=0$
for $|S|=r-1$, all terms cancel.  Therefore
$D_{r-1,r}(\cU_{n,r}A)=0$, proving that the range of $\cU_{n,r}$ is
contained in $\cH_{n,r}$.

\medskip
\noindent\emph{Identity on the residual space.}
Conversely, suppose $A\in\cH_{n,r}$.  For a $j$-set $J$ with $j<r-1$,
sum the zero $(r-1)$-margins over all $(r-1)$-sets $K$ containing $J$:
\[
 0=\sum_{\substack{K\supseteq J\\|K|=r-1}}
   (D_{r-1,r}A)(K)
  =(r-j)(D_{j,r}A)(J).
\]
Each $r$-set containing $J$ is counted $r-j$ times, so every lower margin
$D_{j,r}A$ vanishes.  In \eqref{eq:higher-u-operator}, only the $j=r$ term
remains, and that term is $A$.  Hence $\cU_{n,r}A=A$ on $\cH_{n,r}$.

\medskip
\noindent\emph{Projection and null space.}
A map with a symmetric matrix whose range lies in $\cH_{n,r}$ and which
leaves every element of $\cH_{n,r}$ unchanged is the orthogonal projection
onto that space.  Its null space, consisting of the arrays it maps to zero,
is therefore $\cH_{n,r}^{\perp}=\cL_{n,r}$.

\medskip
\noindent\emph{Dimension.}
The inclusion matrix $D_{r-1,r}$ has full row rank
$\binom n{r-1}$ when $n\ge2r-1$; this is the classical rank theorem of
\citet{Gottlieb1966}.  Therefore
\[
 \dim\cH_{n,r}
 =\dim\ker(D_{r-1,r})
 =\binom nr-\binom n{r-1},
\]
which is \eqref{eq:general-df}.
\end{proof}

\section{\texorpdfstring{{RKHS identities for four-argument matching kernels}}{RKHS identities for four-argument matching kernels}}
\label{app:rkhs}
\noindent\emph{Centered feature maps.}
Let $k$ and $\ell$ be the bounded measurable real-valued positive-definite kernels from Section~\ref{sec:example}, with separable RKHSs $\mathcal F$ and
$\mathcal G$.  Their feature maps are
\[
 \phi(x)=k(x,\cdot)\in\mathcal F,
 \qquad
 \psi(y)=\ell(y,\cdot)\in\mathcal G.
\]
The reproducing property gives
\[
 k(x,x')=\inner{\phi(x)}{\phi(x')}_{\mathcal F},
 \qquad
 \ell(y,y')=\inner{\psi(y)}{\psi(y')}_{\mathcal G}.
\]
Boundedness ensures that the mean feature vectors
$\mu_X=\E\phi(X)$ and $\mu_Y=\E\psi(Y)$ exist.  With
$\phi_c=\phi-\mu_X$ and $\psi_c=\psi-\mu_Y$, the population-centered kernels
satisfy
\[
 k_c(x,x')=\inner{\phi_c(x)}{\phi_c(x')}_{\mathcal F},
 \qquad
 \ell_c(y,y')=\inner{\psi_c(y)}{\psi_c(y')}_{\mathcal G}.
\]

\medskip
\noindent\emph{Fourth Hoeffding projections.}
In each matching term, the two pair kernels involve disjoint arguments, so
the full four-coordinate Hoeffding projection replaces each raw pair kernel
by its second Hoeffding component.  Let $\mathbb E_i$ denote integration in
the $i$th argument.  For example,
\[
 \left\{\prod_{i=1}^4(I-\mathbb E_i)\right\}
       \{k_{12}k_{34}\}
 =\{(I-\mathbb E_1)(I-\mathbb E_2)k_{12}\}
  \{(I-\mathbb E_3)(I-\mathbb E_4)k_{34}\}
 =k_{c,12}k_{c,34}.
\]
Integrating out any one argument sets to zero the unique centered factor
containing that argument, which also verifies canonicality.  Relabeling the
arguments and averaging over the three matchings gives the two displayed
fourth-component formulas in Section~\ref{sec:example}.

\medskip
\noindent\emph{Cross-covariance operator.}
Recall from Section~\ref{sec:example} that the cross-covariance operator is
\[
 C_{XY}=\E\{\phi_c(X)\otimes\psi_c(Y)\}.
\]
The operator has finite Hilbert--Schmidt norm: boundedness makes the centered
features uniformly bounded, while separability supplies the required
measurability.  Consequently, $C_{XY}^{\ast}C_{XY}$ is PSD and has finite trace:
its nonnegative eigenvalues have a finite sum.  Such an operator is called
trace-class.  To see its covariance interpretation, select one scalar
feature $f\in\mathcal F$ of $X$ and one scalar feature $g\in\mathcal G$ of
$Y$.  Then
\[
 \inner{f}{C_{XY}g}_{\mathcal F}
 =\E\!\left[
   \inner{f}{\phi_c(X)}_{\mathcal F}
   \inner{\psi_c(Y)}{g}_{\mathcal G}
 \right]
 =\Cov\{f(X),g(Y)\}.
\]

\medskip
\noindent\emph{Traces and spectral identities.}
For a positive trace-class operator $T$ on $\mathcal G$, its trace is
\[
 \tr(T)=\sum_q\inner{f_q}{Tf_q}_{\mathcal G},
\]
where the value is independent of the orthonormal basis $(f_q)$.
The definitions in Section~\ref{sec:example} give
\[
 \norm{C_{XY}}_{\mathrm{HS}}^2
 =\tr(C_{XY}^{\ast}C_{XY})
 =\sum_j s_j^2,
 \qquad
 \tr\{(C_{XY}^{\ast}C_{XY})^2\}=\sum_j s_j^4.
\]
Finally, for independent copies $(X_1,Y_1)$ and $(X_2,Y_2)$ of $(X,Y)$,
\[
 \norm{C_{XY}}_{\mathrm{HS}}^2
 =\E\{k_c(X_1,X_2)\ell_c(Y_1,Y_2)\}.
\]
The two coordinates within each pair may remain dependent.  These identities
justify \eqref{eq:H2-spectrum}.

\medskip
\noindent\emph{Matching expansion for \eqref{eq:H4-spectrum}.}
For the representative four-cycle term in Section~\ref{sec:example}, choose
orthonormal bases $(e_p)$ and $(f_q)$ of $\mathcal F$ and $\mathcal G$, and
set
\[
 c_{pq}
 =\E\!\left[
   \inner{\phi_c(X)}{e_p}
 \inner{\psi_c(Y)}{f_q}
 \right].
\]
Put
\[
 \xi_{ip}=\inner{\phi_c(X_i)}{e_p},
 \qquad
 \zeta_{iq}=\inner{\psi_c(Y_i)}{f_q}.
\]
Parseval's identity gives
\[
 k_{c,ij}=\sum_p\xi_{ip}\xi_{jp},
 \qquad
 \ell_{c,ij}=\sum_q\zeta_{iq}\zeta_{jq}.
\]
For finite truncations of the basis expansions, the calculation below is
elementary.  Each truncated kernel expansion is bounded by the corresponding
product of centered feature norms, and these norms are uniformly bounded
because the kernels are bounded.  Dominated convergence therefore permits
passage to the infinite sums.  Using independence across the four
paired-observation labels, we obtain
\begin{align*}
 &\E\{k_{c,12}k_{c,34}\ell_{c,13}\ell_{c,24}\}\\
 &\quad=
 \sum_{p,p',q,q'}
 \E(\xi_{1p}\zeta_{1q})
 \E(\xi_{2p}\zeta_{2q'})
 \E(\xi_{3p'}\zeta_{3q})
 \E(\xi_{4p'}\zeta_{4q'})\\
 &\quad=
 \sum_{p,p',q,q'}
 c_{pq}c_{pq'}c_{p'q}c_{p'q'}.
\end{align*}
Here dependence between $X_i$ and $Y_i$ within a pair is retained.  Since
$c_{pq}=\inner{C_{XY}f_q}{e_p}$, let
\[
 T:=C_{XY}^{\ast}C_{XY}:\mathcal G\to\mathcal G.
\]
This is a positive self-adjoint operator.  Its coefficients relative to the
orthonormal basis $(f_q)$ are
\[
\begin{aligned}
 t_{qq'}
 &:=\inner{Tf_{q'}}{f_q}_{\mathcal G}
 =\inner{C_{XY}f_{q'}}{C_{XY}f_q}_{\mathcal F}\\
 &=\sum_p c_{pq'}c_{pq}.
\end{aligned}
\]
Therefore
\[
 \sum_{p,p',q,q'}c_{pq}c_{pq'}c_{p'q}c_{p'q'}
 =\sum_{q,q'}t_{qq'}t_{q'q}
 =\tr(T^2)
 =\sum_{j\ge1}s_j^4.
\]
The other five four-cycle terms agree by relabeling, establishing the
identity used in \eqref{eq:H4-spectrum}.

\section{Direct proofs for the pairwise case}
\label{app:pairwise-proofs}

\subsection{Normal equations, range, and null space
  (Corollary~\ref{cor:pair-basic})}
\label{app:proof-pair-basic}
\begin{proof}
Listing the upper-triangular entries as an edge vector preserves inner
products: $\inner{\cdot}{\cdot}_E$ on $\cV_{n,2}$ becomes the usual
Euclidean inner product on $\R^{\binom{n}{2}}$.  For
$C\in\cV_{n,2}$, let $c=\operatorname{vec}_E(C)$.  Then
\[
 (Z^\mathsf Tc)_i=\sum_{j\ne i}C_{ij}.
\]
Thus $\ker(Z^\mathsf T)$ corresponds exactly to $\cH_{n,2}$.  Likewise,
$\operatorname{col}(Z)$ corresponds exactly to $\cL_{n,2}$ because
$(Zu)_{\{i,j\}}=u_i+u_j$.  Since $Z^\mathsf TZ$ is nonsingular, the matrix
$M_{n,2}$ in \eqref{eq:pair-projection} is the orthogonal projector onto
$\ker(Z^\mathsf T)$ and has null space $\operatorname{col}(Z)$.  Translating back
to arrays gives
\[
 \cV_{n,2}=\cL_{n,2}\oplus\cH_{n,2},\qquad
 \range(\cU_{n,2})=\cH_{n,2},\qquad
 \ker(\cU_{n,2})=\cL_{n,2}.
\]
The matrix of an orthogonal projector is symmetric and idempotent.  Every element of
$\cH_{n,2}$ has zero row sums, and symmetry gives zero column sums.  Finally,
the orthogonal decomposition of $A$ into its $\cL_{n,2}$ and $\cH_{n,2}$
components proves the stated existence and uniqueness characterization.
\end{proof}

\subsection{Pairwise overlap identity
  (Proposition~\ref{prop:pair-overlap})}
\label{app:proof-pair-overlap}
\begin{proof}
Expand both raw kernels using \eqref{eq:pair-hoeffding}.  In any product term,
if a sample label appears in a nonconstant Hoeffding component on only one
side, integrate over that label while holding the others fixed; canonicality
makes the expectation zero.  Consequently, with no shared label only the
overall means remain.  With one shared label, the pairing of the first
Hoeffding components also remains, and with two shared labels, the pairing of
the second components remains as well.  These three cases give the stated
identities and their weighted difference \eqref{eq:theta2-overlap}.
\end{proof}

\subsection{Conditional interaction interpretation
  (Proposition~\ref{prop:pair-conditional})}
\label{app:proof-pair-conditional}
\begin{proof}
Let $m_a(x)=\E\{a(x,X)\}$.  Conditional on $X_i,X_j$,
\[
 \E(A_{i\cdot}\mid X_i,X_j)
 =a(X_i,X_j)+(n-2)m_a(X_i),
\]
and similarly for $A_{j\cdot}$.  The ordered grand sum has conditional mean
\[
 2a(X_i,X_j)+2(n-2)\{m_a(X_i)+m_a(X_j)\}
 +(n-2)(n-3)a_0.
\]
Substitution into \eqref{eq:pairwise-u} gives
\eqref{eq:pair-conditional}.
\end{proof}

\section{\texorpdfstring{{Auxiliary proofs for the general $r$-subset theory}}{Auxiliary proofs for the general r-subset theory}}
\label{app:general-proofs}

\subsection{Overlap sums from margin cross-products
  (Proposition~\ref{prop:margin-inversion})}
\label{app:proof-margin-inversion}
\begin{proof}
By definition, $C_k(A,B)$ sums products of the two $k$-margins.  Expanding
these margins, each product $A_I B_J$ occurs once for every $k$-subset of
$I\cap J$, hence $\binom{|I\cap J|}{k}$ times.

Reversing the order of summation makes this count explicit:
\begingroup
\allowdisplaybreaks[0]
\begin{align*}
 C_k(A,B)
 &=\sum_{K\in\binom{[n]}k}
   \sum_{I\supseteq K}\sum_{J\supseteq K}A_I B_J\\
 &=\sum_{I,J\in\cI_{n,r}}
   \#\{K\in\tbinom{[n]}k:K\subseteq I\cap J\}\,A_I B_J\\
 &=\sum_{I,J\in\cI_{n,r}}\binom{|I\cap J|}{k}A_I B_J,
\end{align*}
\endgroup
which is \eqref{eq:Ck-Sigma}.  Equation
\eqref{eq:Sigma-Ck-inversion} follows from the binomial-inversion identity
\[
 \sum_{k=s}^t(-1)^{k-s}\binom{k}{s}\binom{t}{k}
 =\mathbf 1\{t=s\}.
\]
\end{proof}

\subsection{Closed form for the projection entries
  (Lemma~\ref{lem:p-s})}
\label{app:p-s-proof}
\begin{proof}
To prove \eqref{eq:p-s-closed}, recall that $\kappa=n-2r+1$ and write
$q=s-j$ in \eqref{eq:p-s-def} to obtain
\[
 p_{n,r,s}
 =(-1)^{r-s}\sum_{q=0}^s
 \frac{(-1)^q\binom{s}{q}}
      {\binom{\kappa+r-s+q}{r-s+q}}.
\]
For every integer $m\ge0$,
\[
\frac{1}{\binom{\kappa+m}{m}}
 =\kappa\int_0^1t^m(1-t)^{\kappa-1}\,dt,
\]
\begin{samepage}
so
\begin{align*}
 p_{n,r,s}
 &=(-1)^{r-s}\kappa\int_0^1
   t^{r-s}(1-t)^{\kappa-1}
   \sum_{q=0}^s(-1)^q\binom sqt^q\,dt\\
 &=(-1)^{r-s}\kappa\int_0^1
   t^{r-s}(1-t)^{\kappa+s-1}\,dt.
\end{align*}
\end{samepage}
The beta-integral identity gives
\[
 \kappa\int_0^1 t^{r-s}(1-t)^{\kappa+s-1}\,dt
 =\kappa B(r-s+1,\kappa+s)
 =\frac{\kappa(r-s)!(\kappa+s-1)!}{(\kappa+r)!},
\]
which proves \eqref{eq:p-s-closed}.  For the normalization identity
\eqref{eq:group-coeff}, note that
\[
 d_{n,r}=\binom nr\frac{\kappa}{n-r+1}.
\]
Substituting \eqref{eq:N-s} and \eqref{eq:p-s-closed}, and using
\[
 \binom{n-r}{r-s}(r-s)!
 =\frac{(n-r)!}{(n-2r+s)!},
\]
gives \eqref{eq:group-coeff} after cancellation.
\end{proof}

\subsection{Residual covariance
  (Theorem~\ref{thm:residual-covariance})}
\label{app:proof-residual-covariance}
\begin{proof}
For each $j<r$, the array
$I\mapsto\sum_{S\subseteq I,\,|S|=j}a_j(X_S)$ belongs to
$\cL_{n,r}$, so the Hoeffding decomposition gives
\begin{equation}\label{eq:projection-canonical-array}
 \cU_{n,r}A=\cU_{n,r}A^{(r)},\qquad A_I^{(r)}=a_r(X_I).
\end{equation}
The analogous identity holds for $B$.  Both canonical arrays have mean
zero.  If $I\ne J$, a label belongs to exactly one of $I,J$; integrating
over that observation and using canonicality shows
$\E\{a_r(X_I)b_r(Y_J)\}=0$.  For $I=J$ the expectation is
$\Theta_r(a,b)$.  Thus
\[
 \Cov(A^{(r)},B^{(r)})=\Theta_r(a,b)I_{\binom nr}.
\]
Multiplying on both sides by the symmetric idempotent matrix
$\cU_{n,r}$ proves the cross-covariance formula.  Taking $X=Y$, $a=b$
gives the variance formula.  All these products are integrable by
Cauchy--Schwarz and the assumed square-integrability.
\end{proof}

\subsection{Component estimation and mean-square error
  (Corollary~\ref{cor:general-conditional})}
\label{app:proof-general-conditional}
\begin{proof}
Use \eqref{eq:projection-canonical-array}.  Conditional on $X_I$,
canonicality gives $\E\{a_r(X_J)\mid X_I\}=0$ for every $J\ne I$.
All diagonal entries of $\cU_{n,r}$ are equal by permutation symmetry,
and their sum is its rank $d_{n,r}$.  Therefore each equals
$\lambda_{n,r}=d_{n,r}/\binom nr$, proving
\eqref{eq:general-conditional}.

For $s=|I\cap J|$, the closed form \eqref{eq:p-s-closed} gives
\[
 \frac{p_{n,r,s}}{\lambda_{n,r}}
 =\frac{(-1)^{r-s}}{\binom{n-r}{r-s}}.
\]
Partition the subsets $J$ by $S=I\cap J$ and write $J=S\cup K$, where
$K\subseteq[n]\setminus I$.  Expanding the projection entry then gives
\eqref{eq:complement-recovery}.

The covariance theorem and the conditional-mean identity yield
\[
 \E u_I^2=\lambda_{n,r}\sigma_a^2,\qquad
 \E\{u_Ia_r(X_I)\}=\lambda_{n,r}\sigma_a^2.
\]
Expanding $\E\{(cu_I-a_r(X_I))^2\}$ proves
\eqref{eq:scalar-recovery-risk}.  Substituting $c=1/\lambda_{n,r}$ and
$c=1$ gives \eqref{eq:rescaled-recovery-mse} and
\eqref{eq:unscaled-recovery-mse}, respectively.
\end{proof}

\subsection{Independence moments and comparison with population centering}
\label{app:proof-independence-moments}
The dependent example in Section~\ref{sec:independence-moments} follows
from the order-two kernel
$h(O_1,O_2)=(X_1-X_2)(Y_1-Y_2)/2$ for sample covariance.
Its first Hoeffding component is $h_1(X)=X(X^2-1)/2$, with variance $5/2$.
The second component is
$h_2(O_1,O_2)=-(X_1Y_2+X_2Y_1)/2$, with variance one.
Orthogonality of the Hoeffding components therefore gives
$\Var(T_{n,1})=4(5/2)/n+1/\binom n2$, as stated.

\begin{proof}[Proof of Proposition~\ref{prop:null-oracle-moments}]
For this proof abbreviate $\cU_{n,r}$ by $P$ and set
$\alpha=A^{(r)}$, $\beta=B^{(r)}$.  Under $X\perp Y$, these two mean-zero
arrays are independent and have covariance matrices
$\sigma_a^2I_M$ and $\sigma_b^2I_M$.  Consequently, for deterministic
$M\times M$ matrices $C,D$,
\begin{equation}\label{eq:independent-bilinear-moments}
 \E\{(\alpha^\mathsf TC\beta)(\alpha^\mathsf TD\beta)\}
 =\sigma_a^2\sigma_b^2\tr(CD^\mathsf T).
\end{equation}
Taking $(C,D)=(P/d,P/d)$, $(I_M/M,I_M/M)$, and
$(P/d,I_M/M)$ proves the variance and covariance identities, since
$\tr(P^2)=\tr(P)=d$.  Independence also gives
\[
 \E(\norm{u}^2\norm{v}^2)
 =\E\norm{u}^2\,\E\norm{v}^2
 =d^2\sigma_a^2\sigma_b^2,
\]
which proves unbiasedness of $\widehat V_0$.  Subtracting the corresponding
moments yields \eqref{eq:oracle-discrepancy}.

To identify the sample Hoeffding component, condition on $O_S$, the
observations with labels in $S$.  In
$T_{n,r}=d^{-1}\sum_{I,J}P_{I,J}a_r(X_I)b_r(Y_J)$, a term can survive
conditional integration only if both $I\subseteq S$ and $J\subseteq S$.
To see this, choose a label outside $S$ that belongs to $I$ or $J$.
If it belongs to only one of them, canonicality makes its integral zero.
If it belongs to both, the full independence $X\perp Y$ factors the
integration over that paired observation into separate $X$ and $Y$
integrals, each of which is zero by canonicality.
Thus
\[
 \E(T_{n,r}\mid O_S)=0\quad (|S|<r),\qquad
 \E(T_{n,r}\mid O_S)
 =\frac{a_r(X_S)b_r(Y_S)}{M}\quad (|S|=r),
\]
where the second identity uses $P_{S,S}/d=1/M$.  Indeed, relabeling before
or after applying $P$ gives the same result, and any $r$-subset can be
relabeled as any other.  All diagonal entries of $P$ are therefore equal.
Since their sum is $d$, each equals $d/M$.
The contribution indexed by a label set $S$ in the sample Hoeffding
decomposition of a statistic $T$ is
$\sum_{R\subseteq S}(-1)^{|S|-|R|}\E(T\mid O_R)$.
The conditional expectations above therefore identify the sum of its
order-$r$ contributions with $V_{n,r}$ in
\eqref{eq:oracle-component-statistic}.  Finally,
\[
 \frac1d-\frac1M
 =\frac1M\frac{r}{n-2r+1}=O(n^{-r-1})
\]
for fixed $r$, which proves the stated $L^2$ convergence.
\end{proof}

\begin{proof}[Proof of Corollary~\ref{cor:conditional-ols}]
Conditional on $X_{1:n}$, the vector $u$ is fixed, while independence and
Theorem~\ref{thm:residual-covariance} give
\[
 \E(v\mid X_{1:n})=0,\qquad
 \Cov(v\mid X_{1:n})=\sigma_b^2\cU_{n,r}.
\]
Since $\cU_{n,r}u=u$, the mean and variance of
$\widehat\beta=u^\mathsf Tv/\norm{u}^2$ follow immediately.
Let $H_u=uu^\mathsf T/\norm{u}^2$.  The residual sum of squares is
$v^\mathsf T(\cU_{n,r}-H_u)v$, and
\[
 \E(\mathrm{RSS}\mid X_{1:n})
 =\sigma_b^2\tr\{(\cU_{n,r}-H_u)\cU_{n,r}\}
 =\sigma_b^2(d-1).
\]
Dividing by $d-1$ proves the remaining assertion.
\end{proof}

\subsection{Permutation covariance
  (Theorem~\ref{thm:permutation-covariance})}
\label{app:proof-permutation-covariance}
\begin{proof}
Because $v\in\cH_{n,r}$, all its lower-order margins vanish, including
its grand sum.  Hence a uniform relabeling has $\E_\pi v_I^\pi=0$ for
every $I$.  In the notation of Proposition~\ref{prop:margin-inversion},
\[
 C_j(v,v)=0\quad(j<r),\qquad C_r(v,v)=\norm{v}^2.
\]
The inversion formula therefore gives
\[
 \Sigma_s(v,v)=(-1)^{r-s}\binom rs\norm{v}^2.
\]
If $|I\cap J|=s$, the ordered pair $(\pi(I),\pi(J))$ is uniform among
the $N_{n,r,s}$ ordered pairs with overlap $s$.  Using
\eqref{eq:group-coeff}, we obtain
\[
 \E_\pi(v_I^\pi v_J^\pi)
 =\frac{\Sigma_s(v,v)}{N_{n,r,s}}
 =\frac{\norm{v}^2}{d_{n,r}}p_{n,r,s},
\]
which proves the matrix identity.  Permutation preserves $\norm{v}$,
and $\cU_{n,r}u=u$, so multiplying the matrix identity on the left by
$u^\mathsf T$ and on the right by $u$ gives
\eqref{eq:permutation-statistic-variance} and, when both norms are
positive, \eqref{eq:permutation-correlation-moments}.
\end{proof}

\subsection{The Rademacher null limit}
\label{app:proof-rademacher-limit}
\begin{proof}
Use the distance kernels $a(x,x')=|x-x'|$ and $b(y,y')=|y-y'|$ under
independent symmetric Rademacher laws.  Their canonical components are
$a_2(x,x')=-xx'$ and $b_2(y,y')=-yy'$, with
$\sigma_a^2=\sigma_b^2=1$.  Let $Z_i=X_iY_i$, which are again i.i.d.
symmetric Rademacher variables.  The statistic in
\eqref{eq:oracle-component-statistic} is
\[
 V_{n,2}
 =\frac{(\sum_i Z_i)^2-n}{n(n-1)}.
\]
Proposition~\ref{prop:null-oracle-moments} gives
$n(T_{n,2}-V_{n,2})\to0$ in $L^2$.  The central limit theorem and the
continuous mapping theorem therefore yield
\[
 nT_{n,2}
 =\left(n^{-1/2}\sum_iZ_i\right)^2-1+o_p(1)
 \ \Longrightarrow\ \chi_1^2-1.
\]
The U-centered mean squares
$\norm{u}^2/d_{n,2}$ and $\norm{v}^2/d_{n,2}$ converge in probability to
one.  Indeed, each is the fixed-order unbiased U-statistic for the
corresponding highest-component variance, by
Theorem~\ref{thm:general-compression}; its kernel is bounded here, so the
U-statistic law of large numbers applies.  It follows that
\[
 nR_U
 =\frac{nT_{n,2}}
 {\{(\norm{u}^2/d_{n,2})(\norm{v}^2/d_{n,2})\}^{1/2}}
 \ \Longrightarrow\ \chi_1^2-1.
\]
In particular, the probability of a zero residual norm tends to zero,
and $R_U\to0$ in probability on the remaining event.  The probability
of $|R_U|=1$ consequently also tends to zero.  Since
$\sqrt{d_{n,2}-1}/n\to1/\sqrt2$, substitution into
\eqref{eq:edge-regression-t} proves \eqref{eq:rademacher-null-limit},
with the stated zero convention on the exceptional events.
\end{proof}

\subsection{Why the sample-size and nonnegativity conditions are necessary}
\label{app:proof-variance-sharpness}
\begin{proof}
For the Bernoulli product kernel, the highest canonical component is
$h_r(z_1,\ldots,z_r)=\prod_i(z_i-p)$, giving
\eqref{eq:sample-size-sharpness-target}.  The expectation of any statistic
of $n$ binary observations that is integrable for every $p\in(0,1)$ is
a polynomial in $p$ of degree at most $n$.  The target
$p^r(1-p)^r$ has degree $2r$, with nonzero leading coefficient.  Equality
for all $p\in(0,1)$ is therefore impossible when $n<2r$.

For a law on $\{-1,1\}$ with mean $\mu$, the product kernel has
$h_j(z_1,\ldots,z_j)=\mu^{r-j}\prod_{i=1}^j(z_i-\mu)$, which gives
\eqref{eq:lower-component-positivity-target}.  For $j<r$ the target is
zero at $\mu=0$, where every sample configuration has positive
probability.  A nonnegative unbiased statistic must then be zero at every
configuration.  Its expectation is consequently zero for every $\mu$,
contradicting positivity of the target for $0<|\mu|<1$.  The same
arguments cover additional parameter-independent randomization by first
taking its conditional expectation given the observations.
\end{proof}

\section{Proofs for two-sample regression and permutation identities}
\label{app:two-sample-nesting}

We retain the notation of Section~\ref{sec:edge-regression}.
The dependence regression follows from
Proposition~\ref{prop:general-edge-regression}, with $W=Z$, together with
Theorem~\ref{thm:pair-anova}.  We therefore focus on
Propositions~\ref{prop:energy-group-slope} and~\ref{prop:energy-anova-t}.
We first project the between-group pair indicator in the common-endpoint
model.  We then analyze the fully interacted model block by block, recover
its pooled studentization, and compare the two fitted spaces.

\noindent\emph{Projection of the between-group indicator and the common
slope.}
Recall that $M_{N,2}$ projects onto $\ker(Z_N^\mathsf T)$, the space
orthogonal to all additive endpoint effects.  To prove
\eqref{eq:ucentered-group-indicator}, we show that
$\ell_E/a_{n,m}^2$ belongs to this residual space and that its difference
from $c_E$ belongs to $\operatorname{col}(Z_N)$.

The $i$th entry of $Z_N^\mathsf T\ell_E$ is the sum of the entries of
$\ell_E$ over the edges incident to label $i$.  For an $X$-sample label and
a $Y$-sample label, respectively, these sums are
\[
 (n-1)\left\{-\frac1{n(n-1)}\right\}
 +m\left(\frac1{nm}\right)=0,
 \qquad
 (m-1)\left\{-\frac1{m(m-1)}\right\}
 +n\left(\frac1{nm}\right)=0.
\]
Hence $Z_N^\mathsf T\ell_E=0$, and therefore
$\ell_E\in\ker(Z_N^\mathsf T)=\range(M_{N,2})$.  Moreover, only the $nm$
between-group edges contribute to $c_E^\mathsf T\ell_E$, so
\[
 c_E^\mathsf T\ell_E=nm\left(\frac1{nm}\right)=1,
 \qquad
 \norm{\ell_E}^2=a_{n,m}^2.
\]

It remains to show that the difference between $c_E$ and
$\ell_E/a_{n,m}^2$ is an additive endpoint effect.  Set
\[
 \gamma_X^{(0)}=\frac{1}{2n(n-1)a_{n,m}^2},
 \qquad
 \gamma_Y^{(0)}=\frac{1}{2m(m-1)a_{n,m}^2}.
\]
On an $XX$ edge, the entry of $c_E-\ell_E/a_{n,m}^2$ is
$2\gamma_X^{(0)}$, and on a $YY$ edge it is $2\gamma_Y^{(0)}$.
On an $XY$ edge, the definition of $a_{n,m}^2$
gives
\[
 1-\frac{1}{nm a_{n,m}^2}=\gamma_X^{(0)}+\gamma_Y^{(0)}.
\]
Let $\gamma^{(0)}\in\R^N$ equal $\gamma_X^{(0)}$ on the $X$-sample labels
and $\gamma_Y^{(0)}$ on the $Y$-sample labels.  For an edge $\{i,j\}$, the
corresponding entry of $Z_N\gamma^{(0)}$ is
$\gamma_i^{(0)}+\gamma_j^{(0)}$.  The three edge-type calculations above
therefore give
\[
 c_E-\frac{\ell_E}{a_{n,m}^2}=Z_N\gamma^{(0)}.
\]
This is the orthogonal decomposition of $c_E$ into a vector in
$\range(M_{N,2})$ and a vector in
$\ker(M_{N,2})=\operatorname{col}(Z_N)$.  Applying $M_{N,2}$ proves
\eqref{eq:ucentered-group-indicator}.

Using $M_{N,2}c_E=\ell_E/a_{n,m}^2$,
$c_E^\mathsf T\ell_E=1$, and the symmetry of $M_{N,2}$, the
Frisch--Waugh--Lovell formula gives
\[
 \widehat\theta
 =\frac{c_E^\mathsf TM_{N,2}a_E}
        {c_E^\mathsf TM_{N,2}c_E}
 =\frac{\ell_E^\mathsf Ta_E/a_{n,m}^2}
        {1/a_{n,m}^2}
 =\ell_E^\mathsf Ta_E
 =\overline A_{XY}-\frac{\overline A_{XX}+\overline A_{YY}}2
 =\frac{\mathcal E_{n,m}}2.
\]
The penultimate equality follows by grouping the entries of
$\ell_E^\mathsf Ta_E$ over the $XX$, $XY$, and $YY$ blocks.

By Theorem~\ref{thm:pair-anova}, the edge vectors of $\widetilde A$ and
$\widetilde C$ are $M_{N,2}a_E$ and $M_{N,2}c_E$, respectively.  Because
$M_{N,2}$ is an orthogonal projector,
\[
 \inner{\widetilde A}{\widetilde C}_E
 =(M_{N,2}a_E)^\mathsf T(M_{N,2}c_E)
 =\frac{\ell_E^\mathsf Ta_E}{a_{n,m}^2}
 =\frac{\mathcal E_{n,m}}{2a_{n,m}^2},
\]
which proves \eqref{eq:energy-as-ucentered-label-covariance}.

The same calculation also gives the restricted-model residual sum of
squares.  Adding the residualized predictor $M_{N,2}c_E$ reduces the
endpoint-only residual sum of squares by
\[
 \frac{(c_E^\mathsf TM_{N,2}a_E)^2}
      {c_E^\mathsf TM_{N,2}c_E}
 =\frac{\widehat\theta^2}{a_{n,m}^2}.
\]
This proves the formula for $\mathrm{RSS}_R$ in
\eqref{eq:restricted-two-sample-rss}.  The matrix $Z_N$ has rank $N$, and
$M_{N,2}c_E=\ell_E/a_{n,m}^2\ne0$, so adding $c_E$ raises the fitted rank to
$N+1$.  The restricted residual degrees of freedom are therefore
$\binom N2-N-1$, as stated in \eqref{eq:restricted-two-sample-rss}.  Moreover,
$(c_E^\mathsf TM_{N,2}c_E)^{-1}=a_{n,m}^2$ is the least-squares variance
factor for $\widehat\theta$.  Provided $\mathrm{RSS}_R>0$, its conventional
homoskedastic standard error is therefore
$a_{n,m}\sqrt{\mathrm{RSS}_R/\nu_R}$, which gives
\eqref{eq:restricted-energy-t}.

\smallskip
\noindent\emph{Blockwise solution of the fully interacted model.}
Write the three block means as
\[
 \mu_{XX}=\eta+\delta,
 \qquad
 \mu_{XY}=\eta+\theta,
 \qquad
 \mu_{YY}=\eta-\delta.
\]
This is a one-to-one reparametrization because
\[
 \eta=\frac{\mu_{XX}+\mu_{YY}}2,
 \qquad
 \delta=\frac{\mu_{XX}-\mu_{YY}}2,
 \qquad
 \theta=\mu_{XY}-\frac{\mu_{XX}+\mu_{YY}}2.
\]
The $X$- and $Y$-group $\alpha$ effects occur only in the $XX$ and $YY$
blocks, respectively, while the $\beta$ effects occur only as row and column
effects in the $XY$ block.  Because the three blocks contain disjoint pair
entries, the full least-squares criterion is a sum of three criteria that can
be minimized separately.

Each centered endpoint-effect term has average zero in its block.  Indeed,
every $\alpha_i$ appears in $n-1$ edges of the $XX$ block, so the sum of the
endpoint terms over that block is
$(n-1)\sum_{i\le n}\alpha_i=0$.  The same argument applies to the $YY$ block.
In the $XY$ block, the corresponding sum is
$m\sum_{i\le n}\beta_i+n\sum_{j>n}\beta_j=0$.
The fitted block means are therefore
$\overline A_{XX}$, $\overline A_{XY}$, and $\overline A_{YY}$.  Substituting
them into the expression for $\theta$ gives
\[
 \widehat\theta
 =\overline A_{XY}
  -\frac{\overline A_{XX}+\overline A_{YY}}2
 =\ell_E^\mathsf Ta_E
 =\frac{\mathcal E_{n,m}}2.
\]

For the rectangular $XY$ block, fitting a mean together with centered row
and column effects is ordinary two-way additive ANOVA.  Double-centering
subtracts the fitted row and column effects, with the grand mean restored.
For the symmetric $XX$ and $YY$ blocks, fitting a mean and additive endpoint
effects is the pairwise ANOVA problem in Theorem~\ref{thm:pair-anova}.  The
column space of the endpoint design matrix already contains the constant vector, so
U-centering removes both the block mean and the endpoint effects.  The three
residuals are thus
\[
 \Pi_nA^{XY}\Pi_m,
 \qquad
 M_{n,2}\operatorname{vec}_E(A^{XX}),
 \qquad
 M_{m,2}\operatorname{vec}_E(A^{YY}).
\]
For the rectangular residual, summing the squares of all entries gives
its squared Frobenius norm.  For the two within-group residuals, the edge
vectors contain only upper-triangular entries.  Their squared norms
therefore add to \eqref{eq:stacked-full-rss} without any additional factor
of two.

The fitted dimensions of the $XX$, $XY$, and $YY$ blocks are, respectively,
$n$, $N-1$, and $m$.  Hence the corresponding residual spaces have
dimensions
\[
 d_{n,2}=\binom n2-n,
 \qquad
 (n-1)(m-1),
 \qquad
 d_{m,2}=\binom m2-m.
\]
Their sum is
\[
 \nu_F=(n-1)(m-1)+d_{n,2}+d_{m,2}
 =\binom N2-2N+1,
\]
which proves \eqref{eq:stacked-two-sample-df}.  Equivalently, the full fitted
space has rank $n+(N-1)+m=2N-1$.

\smallskip
\noindent\emph{Recovery of the pooled studentization.}
For the same dissimilarity convention as in
\citet{ChakrabortyZhang2021}, let $v_s=s(s-3)/2=d_{s,2}$, and let
$\widetilde A^{XX}$ and $\widetilde A^{YY}$ denote the U-centered
within-group arrays.  In their notation, the cross-block quantity and the two
within-sample squared distance variances are
\begin{align*}
 \operatorname{cdCov}_{n,m}^2(X,Y)
 &=\frac{\norm{\Pi_nA^{XY}\Pi_m}_F^2}{(n-1)(m-1)}
   =\frac{\mathrm{RSS}_{XY}}{(n-1)(m-1)},\\
 \widetilde{\mathcal D}_n^2(X,X)
 &=\frac{1}{n(n-3)}\sum_{i\ne j}(\widetilde A^{XX}_{ij})^2
   =\frac{\mathrm{RSS}_{XX}}{v_n},\\
 \widetilde{\mathcal D}_m^2(Y,Y)
 &=\frac{1}{m(m-3)}\sum_{i\ne j}(\widetilde A^{YY}_{ij})^2
   =\frac{\mathrm{RSS}_{YY}}{v_m}.
\end{align*}
The rectangular $XY$ array lists each cross edge once.  By contrast, the two
within-sample definitions sum over ordered pairs, so symmetry counts every
within-group edge twice.  Substituting the three identities above into the
pooled variance estimator of \citet[p.~5469]{ChakrabortyZhang2021} gives
\begin{equation}
 S_{n,m}
 =\frac{4(n-1)(m-1)\operatorname{cdCov}_{n,m}^2(X,Y)
          +4v_n\widetilde{\mathcal D}_n^2(X,X)
          +4v_m\widetilde{\mathcal D}_m^2(Y,Y)}
         {(n-1)(m-1)+v_n+v_m}.
 \label{eq:cz-pooled-variance}
\end{equation}
The three identities show that its numerator is
$4(\mathrm{RSS}_{XY}+\mathrm{RSS}_{XX}+\mathrm{RSS}_{YY})$ and its
denominator is $\nu_F$.  Hence $S_{n,m}=4\mathrm{MS}_F$.
The block-mean calculation above writes
$\widehat\theta=\ell_E^\mathsf Ta_E$, so $\ell_E$ is its least-squares
coefficient-weight vector.  Since $\norm{\ell_E}^2=a_{n,m}^2$, the usual
homoskedastic least-squares formula gives, provided $\mathrm{MS}_F>0$,
\[
 \widehat{\operatorname{se}}(\widehat\theta)
 =a_{n,m}\sqrt{\mathrm{MS}_F}.
\]
Consequently,
\[
 \widehat{\operatorname{se}}(\mathcal E_{n,m})
 =2a_{n,m}\sqrt{\mathrm{MS}_F}
 =a_{n,m}\sqrt{S_{n,m}},
\]
and dividing $\mathcal E_{n,m}$ by this quantity gives
\eqref{eq:energy-anova-t}.

\smallskip
\noindent\emph{Nesting of the two regressions.}
Consider a generic fitted value from the restricted model.  Define
\[
 \overline\gamma_X=\frac1n\sum_{i\le n}\gamma_i,
 \qquad
 \overline\gamma_Y=\frac1m\sum_{i>n}\gamma_i,
\]
and write
\[
 d_i=
 \begin{cases}
  \gamma_i-\overline\gamma_X,&i\le n,\\
  \gamma_i-\overline\gamma_Y,&i>n.
 \end{cases}
 \qquad
 \sum_{i\le n}d_i=\sum_{i>n}d_i=0.
\]
The restricted fitted values on $XX$, $XY$, and $YY$ edges are, respectively,
\[
 2\overline\gamma_X+d_i+d_j,
 \qquad
 \overline\gamma_X+\overline\gamma_Y+\theta+d_i+d_j,
 \qquad
 2\overline\gamma_Y+d_i+d_j.
\]
The fully interacted model reproduces these values by taking
$\alpha_i=\beta_i=d_i$,
$\eta=\overline\gamma_X+\overline\gamma_Y$, and
$\delta=\overline\gamma_X-\overline\gamma_Y$, while leaving $\theta$
unchanged.  These choices satisfy all four zero-sum constraints.  Thus the
restricted fitted space is contained in the fully interacted fitted space.

The restricted fitted rank is $N+1$, as shown above, while the full fitted
rank is $2N-1$.  The fully interacted model therefore adds $N-2$ fitted
directions, and
\[
 \nu_R-\nu_F=N-2.
\]
Let $\widehat f_R$ and $\widehat f_F$ denote the two fitted edge vectors.
Nested least squares gives the orthogonal decomposition
\[
 a_E-\widehat f_R
 =(a_E-\widehat f_F)+(\widehat f_F-\widehat f_R).
\]
Therefore
\[
 \mathrm{RSS}_R
 =\mathrm{RSS}_F+\mathrm{SS}_{\mathrm{int}},
 \qquad
 \mathrm{SS}_{\mathrm{int}}
 :=\norm{\widehat f_F-\widehat f_R}^2\ge0.
\]
The second term is the additional variation explained by allowing endpoint
effects to differ between within- and between-group pairs.  This proves
\eqref{eq:nested-rss-decomposition}.

Together, these calculations prove
Propositions~\ref{prop:energy-group-slope} and~\ref{prop:energy-anova-t}, as
well as the restricted-model formulas
\eqref{eq:restricted-two-sample-rss}--\eqref{eq:restricted-energy-t}.

\smallskip
\noindent\emph{The explicit interaction sum of squares.}
The centered effects in the full fit have the formulas
\begin{align*}
 \widehat\alpha_i^X
 &=\frac{\sum_{j\ne i}A^{XX}_{ij}-(n-1)\overline A_{XX}}{n-2},
 &\widehat\beta_i^X
 &=\frac1m\sum_{j=1}^m A^{XY}_{ij}-\overline A_{XY},\\
 \widehat\alpha_j^Y
 &=\frac{\sum_{i\ne j}A^{YY}_{ij}-(m-1)\overline A_{YY}}{m-2},
 &\widehat\beta_j^Y
 &=\frac1n\sum_{i=1}^n A^{XY}_{ij}-\overline A_{XY}.
\end{align*}
Each of these four vectors sums to zero.  For endpoint effects
$v\in\R^n$ with $\sum_i v_i=0$, the sum of squared fitted pair effects is
\[
 \sum_{i<j}(v_i+v_j)^2=(n-2)\norm v^2.
\]
For row and column effects $v\in\R^n$ and $w\in\R^m$, each summing to
zero, the corresponding rectangular array satisfies
\[
 \sum_{i=1}^n\sum_{j=1}^m(v_i+w_j)^2
 =m\norm v^2+n\norm w^2.
\]
The three block means can be fitted freely in the restricted model, just as
in the full model.  Restricting the $X$-group endpoint effects to a common
centered vector $v$ therefore increases the residual sum of squares by
the minimum of
\[
 (n-2)\norm{\widehat\alpha^X-v}^2
 +m\norm{\widehat\beta^X-v}^2.
\]
The minimizer is
$v=\{(n-2)\widehat\alpha^X+m\widehat\beta^X\}/(N-2)$, and the minimum is
$m(n-2)\norm{\widehat\alpha^X-\widehat\beta^X}^2/(N-2)$.
The same calculation for the $Y$-group effects proves
\eqref{eq:interaction-ss-explicit}.  If the common slope is nonzero and
both residual sums of squares are positive, comparing the denominators of
$T_F$ and $T_R$ gives $|T_F|>|T_R|$ if and only if
$\mathrm{MS}_F<\mathrm{MS}_R$.  Substitution of
$\mathrm{RSS}_R=\mathrm{RSS}_F+\mathrm{SS}_{\mathrm{int}}$ and
$\nu_R=\nu_F+N-2$ proves
\eqref{eq:interaction-studentization-comparison}.

\smallskip
\noindent\emph{Permutation moments.}
Fix a reference assignment and put $P_0=M_{N,2}$ and $u=P_0a_E$.
Permuting the pooled observations while keeping the reference assignment
fixed produces the same distribution of statistics as uniformly assigning
labels to groups of sizes $n,m$: each assignment is produced by exactly
$n!m!$ permutations.  Applying $P_0$ before or after relabeling gives the
same residual array.
Thus, by Theorem~\ref{thm:permutation-covariance} with $r=2$,
\begin{equation}
 \E_\pi u^\pi=0,\qquad
 \E_\pi[u^\pi(u^\pi)^\mathsf T]=\frac{Q_0}{d}P_0.
 \label{eq:pooled-permutation-covariance}
\end{equation}
Since $P_0\ell_E=\ell_E$ and $\norm{\ell_E}^2=a_{n,m}^2$, the slope
$\widehat\theta^\pi=\ell_E^\mathsf Tu^\pi$ has the moments in
\eqref{eq:two-sample-permutation-slope}.

Let $Q_R$ and $Q_F$ be the orthogonal residual projectors for the two
regressions with the reference assignment.  Their fitted spaces contain
the global endpoint space, so
\[
 Q_R=P_0Q_RP_0,\qquad Q_F=P_0Q_FP_0.
\]
Their ranks are $\nu_R=d-1$ and $\nu_F=d-N+1$, respectively.
For either projector $Q$, of rank $q$, equation
\eqref{eq:pooled-permutation-covariance} gives
\[
 \E_\pi\norm{Qu^\pi}^2
 =\tr\left\{Q\,\E_\pi[u^\pi(u^\pi)^\mathsf T]\right\}
 =\frac{Q_0}{d}\tr(QP_0)=q\frac{Q_0}{d}.
\]
Dividing by $q$ proves the two mean-square identities in
\eqref{eq:two-sample-permutation-mean-squares}.  The nesting of the fitted
spaces makes $Q_R-Q_F$ an orthogonal projector of rank $N-2$.
Its quadratic form is $\mathrm{SS}_{\mathrm{int}}$, which proves the
remaining identity and Proposition~\ref{prop:two-sample-permutation}.
The monotonicity statement follows from
\eqref{eq:restricted-pooled-monotonicity} on its stated domain.

\smallskip
\noindent\emph{Mean-square averaging under homogeneity.}
When $P=Q$, the pooled observations are i.i.d.  Theorem~\ref{thm:residual-covariance}
with $a=b=\rho$ and $r=2$ yields
$\Cov(P_0a_E)=\sigma_{\rho,2}^2P_0$ and $\E(P_0a_E)=0$.
Taking traces gives $\E Q_0=d\sigma_{\rho,2}^2$.
The other equalities in \eqref{eq:two-sample-null-mean-squares} follow by
averaging \eqref{eq:two-sample-permutation-mean-squares} over the pooled
sample.

For completeness, the conditional-expectation statement follows by
averaging over all permutations of the pooled observations.  If $f$ is
integrable, its permutation average $\overline f$ is unchanged by
relabeling.  For every bounded function $g$ that is also unchanged by relabeling,
exchangeability gives $\E(gf)=\E(g\overline f)$.  Hence
$\overline f=\E(f\mid\mathcal F_{\mathrm{sym}})$.
Apply this fact to either mean square and use
\eqref{eq:two-sample-permutation-mean-squares} to obtain
\eqref{eq:pooled-mean-square-rao-blackwell}.
The fourth moment assumption on $\rho(Z,Z')$ makes both mean squares
square-integrable, since they are quadratic forms in finitely many pairwise
values.  Applying the same averaging argument to their squares identifies
the conditional variance with $\Var_\pi(\mathrm{MS}_J^\pi)$.
The law of total variance now proves
\eqref{eq:pooled-mean-square-variance-reduction}, completing the proof of
Corollary~\ref{cor:pooled-mean-square-averaging}.

\section{Spectral computation and estimation proofs}
\label{app:spectral-computation}

This appendix proves Proposition~\ref{prop:cubic-spectral} and the
finite-sample boundary statements in Section~\ref{sec:example}.  The finite-array identities require only
real symmetric zero-diagonal pairwise arrays.  The population conclusions use the
bounded-kernel assumptions of Section~\ref{sec:example}.

\subsection{Why pairwise U-centering preserves the estimators}
Let $K^{(0)}$ and $L^{(0)}$ be the raw zero-diagonal pairwise arrays, and set
$K=\cU_{n,2}K^{(0)}$ and $L=\cU_{n,2}L^{(0)}$.
Thus $K\one=L\one=0$.  These are sample residual arrays, rather than
population-centered kernels.  We first justify their use in the exact
estimators.

For any numbers $u_1,\ldots,u_n$, changing the off-diagonal entries of
$K^{(0)}$ by $u_i+u_j$ changes its restriction to any four-subset $I$ by an
additive endpoint array.  The four-observation pairwise projector maps
this array to zero.  Hence the four-observation HSIC kernel $g(O_I)$ is unchanged,
as is the disjoint-subset estimator $\widehat Q$.
For a matching term, replacing $K^{(0)}_{ij}$ by
$K^{(0)}_{ij}+u_i+u_j$ produces additional terms of the forms
$K^{(0)}_{ij}u_k$ and $u_i u_k$.  Each involves at most three labels and
belongs to the lower-order subset ANOVA space.  Consequently,
\[
 \cU_{n,4}m_{K^{(0)}+(u_i+u_j)}=\cU_{n,4}m_{K^{(0)}}.
\]
Pairwise U-centering is a change of this form, so it preserves
$\widehat H_4$ as well as $\widehat Q$.  This is a deterministic identity;
the endpoint coefficients may themselves have been computed from the full
sample.

\subsection{Nine scalar quantities and matching margins}
Write $W=K\circ L$, $u=(K\circ K)\one$, $v=(L\circ L)\one$, and
$w=W\one$.  Here $\circ$ denotes entrywise multiplication.  Define
\begin{align*}
 \tau_1&=\norm K_F^2\norm L_F^2,
 &\tau_2&=(\one^{\mathsf T}W\one)^2,\\
 \tau_3&=u^{\mathsf T}v,
 &\tau_4&=\norm w^2,
 &\tau_5&=\norm W_F^2,\\
 \tau_6&=\sum_{i,j}\{K_{ij}^2(L^2)_{ij}+L_{ij}^2(K^2)_{ij}\},\\
 \tau_7&=\sum_{i,j}W_{ij}(KL)_{ij},
 &\tau_8&=\norm{KL}_F^2=\tr(K^2L^2),\\
 \tau_9&=\sum_{i,j}(KL)_{ij}(KL)_{ji}=\tr(KLKL).
\end{align*}
The symbols $K^2$ and $L^2$ denote matrix products, whereas $K_{ij}^2$
and $L_{ij}^2$ denote scalar squares.  In particular, $\tau_8$ and
$\tau_9$ generally differ.

For $M_j^K=D_{j,4}m_K$, summing over the remaining labels gives
\begin{align*}
 M_2^K(i,j)&=\frac{2K_{ij}^2-(K^2)_{ij}}3,
 &M_1^K(i)&=\frac{u_i}3,
 &M_0^K&=\frac{\norm K_F^2}{12},\\
 M_3^K(i,j,k)&=-\frac23
   (K_{ij}K_{ik}+K_{ij}K_{jk}+K_{ik}K_{jk}).
\end{align*}
For example, before imposing zero row sums the second margin is
\[
 \frac13\left[
 K_{ij}(E_K-R_i-R_j+K_{ij})
 +(R_i-K_{ij})(R_j-K_{ij})-(K^2)_{ij}\right],
\]
where $R_i=\sum_jK_{ij}$ and $E_K=\sum_{i<j}K_{ij}$.
Setting $R_i=E_K=0$ gives the displayed $M_2^K$ formula.  The relations
$M_1^K(i)=\frac13\sum_{j\ne i}M_2^K(i,j)$ and
$M_0^K=\frac14\sum_iM_1^K(i)$ give the next two formulas, using
$K^2\one=0$.  The third margin follows by summing a matching over its
fourth label.  The formulas for $L$ are identical.

Let $C_j=\sum_{|J|=j}M_j^K(J)M_j^L(J)$.  These identities yield
\begin{equation}\label{eq:spectral-matrix-margins}
 \begin{aligned}
 C_0&=\tau_1/144,& C_1&=\tau_3/9,\\
 C_2&=(4\tau_5-2\tau_6+\tau_8-\tau_3)/18,\\
 C_3&=(2\tau_4-2\tau_5+4\tau_7)/9,\\
 C_4&=(\tau_2/2+\tau_9-4\tau_4+2\tau_5)/36.
 \end{aligned}
\end{equation}
For $C_2$, use
$\sum_{i\ne j}(K^2)_{ij}(L^2)_{ij}=\tau_8-\tau_3$.
For $C_3$, multiplying the two three-term sums gives nine products: three
use the same two edges in $K$ and $L$, and six use different pairs of edges.
The first group sums to
$(\tau_4-\tau_5)/2$ over all unordered triples; the second sums to
$\tau_7$.  Multiplication by $4/9$ gives the stated result.

For $C_4$, pairs using the same matching contribute
\[
 \frac{E_W^2-\norm w^2+\norm W_F^2/2}{2},
 \qquad E_W=\sum_{i<j}W_{ij}.
\]
Different matchings form cycles through four labels, with $K$- and $L$-edges
alternating.  In $\tr(KLKL)$,
terms with equal adjacent indices vanish because the diagonal entries are
zero.  Repeated indices can therefore remain only at opposite vertices of
the cycle.  Removing those terms gives
$\tau_9-2\tau_4+\tau_5$.  Each ordered pair of different matchings is
counted four times, so its contribution is one quarter of this expression.
Adding the two groups and dividing by nine proves the final identity in
\eqref{eq:spectral-matrix-margins}.

The projection formula therefore gives
\begin{equation}\label{eq:spectral-H4-contraction}
 \widehat H_4=
 \frac{C_4-\dfrac{C_3}{n-6}
       +\dfrac{C_2}{\binom{n-5}{2}}
       -\dfrac{C_1}{\binom{n-4}{3}}
       +\dfrac{C_0}{\binom{n-3}{4}}}
      {\binom n4-\binom n3}.
\end{equation}
No triple array or triple loop is needed to compute the five $C_j$.

\subsection{Derivation of the coefficients for the disjoint HSIC product}
Put $(n)_m=n(n-1)\cdots(n-m+1)$ and $b_m=1/(n)_m$.
Then the exact disjoint-subset statistic is
\begin{equation}\label{eq:spectral-Q-contraction}
 \widehat Q=\sum_{j=1}^9\tau_j\sum_{m=4}^8 B_{jm}b_m,
\end{equation}
where Table~\ref{tab:spectral-coefficients} gives the integer coefficients.
\begin{table}[tb]
\centering
\caption{Exact coefficients in \eqref{eq:spectral-Q-contraction}.  The
sixth row applies to the two sums in $\tau_6$, obtained from each other
by exchanging $K$ and $L$; each has the displayed coefficient.}
\label{tab:spectral-coefficients}
\begin{tabular}{crrrrr}
\toprule
$j$ & $b_4$ & $b_5$ & $b_6$ & $b_7$ & $b_8$\\
\midrule
1&0&0&0&0&4\\
2&1&4&8&8&8\\
3&0&0&-4&-32&-96\\
4&-4&-24&-80&-160&-192\\
5&2&16&72&192&288\\
6&0&0&-4&-32&-64\\
7&0&-8&-56&-192&-256\\
8&0&0&4&16&32\\
9&0&0&4&16&16\\
\bottomrule
\end{tabular}
\end{table}

Here is a derivation of the table by counting terms.  Represent each product
of matrix entries by a graph.  Each factor contributes one edge, whose
color records whether the factor comes from $K$ or $L$.  Assign a sample label
from $\{1,\ldots,n\}$ to each vertex and multiply the matrix entries
specified by its edges.  For a graph $G$, let $\operatorname{inj}_n(G)$ sum
these products over assignments in which every vertex has a different
label; let $\operatorname{hom}_n(G)$ sum over all assignments, allowing
repeated labels.  Thus their difference consists of terms in which two or
more vertex labels coincide.  Parallel edges multiply factors on the same
pair; a doubled edge means two parallel edges.  Let $\mathsf A$ have a $K$-edge and an $L$-edge on the same two
vertices, $\mathsf B$ have a $K$-edge and an $L$-edge sharing one endpoint,
and $\mathsf C$ have disjoint $K$- and $L$-edges.  They have two, three,
and four vertices, respectively.

For a small example, consider the two-edge graph $\mathsf B$.  Its
unrestricted sum is
\[
 \operatorname{hom}_n(\mathsf B)
 =\sum_{i,j,k}K_{ij}L_{ik}
 =\sum_i\left(\sum_jK_{ij}\right)\left(\sum_kL_{ik}\right)=0,
\]
because both matrices have zero row sums.  Terms with $i=j$ or $i=k$
vanish because the matrices are zero-diagonal.  The only remaining repeated-label
terms have $j=k$, and their sum is $\sum_{i,j}K_{ij}L_{ij}$.  Removing them
therefore gives
\[
 \operatorname{inj}_n(\mathsf B)
 =\sum_{\substack{i,j,k\ \mathrm{pairwise\ distinct}}}K_{ij}L_{ik}
 =-\sum_{i,j}K_{ij}L_{ij}
 =-\one^{\mathsf T}W\one.
\]
The partition formula below carries out the same removal of repeated
labels for the larger graphs needed in $\widehat Q$.

The symbol $\sqcup$ denotes disjoint union: the two graphs use separate
vertex sets.  The four-observation HSIC kernel is the symmetrization of
$K_{12}L_{12}-2K_{12}L_{13}+K_{12}L_{34}$.
Expanding the product on two disjoint four-subsets consequently gives
\begin{equation}\label{eq:spectral-six-patterns}
\begin{aligned}
 \widehat Q={}&
 \frac{\operatorname{inj}_n(\mathsf A\sqcup\mathsf A)}{(n)_4}
 -\frac{4\operatorname{inj}_n(\mathsf A\sqcup\mathsf B)}{(n)_5}\\
 &+\frac{4\operatorname{inj}_n(\mathsf B\sqcup\mathsf B)
          +2\operatorname{inj}_n(\mathsf A\sqcup\mathsf C)}{(n)_6}\\
 &-\frac{4\operatorname{inj}_n(\mathsf B\sqcup\mathsf C)}{(n)_7}
 +\frac{\operatorname{inj}_n(\mathsf C\sqcup\mathsf C)}{(n)_8}.
\end{aligned}
\end{equation}
Unused labels in a lifted four-argument term cancel between the numerator
and denominator, which explains the differing falling-factorial orders.

To retain only assignments with distinct labels, we remove repeated-label
terms by inclusion--exclusion.  For a partition $\pi$ of the vertex set,
merge the vertices in each block and call the resulting graph $G/\pi$.
This graph may contain parallel edges or loops.  Inclusion--exclusion over
these partitions, also called partition-lattice M\"obius inversion, gives
\begin{equation}\label{eq:spectral-partition-certificate}
 \operatorname{inj}_n(G)=
 \sum_\pi\left\{\prod_{D\in\pi}
     (-1)^{|D|-1}(|D|-1)!\right\}
 \operatorname{hom}_n(G/\pi);
\end{equation}
see also \citet{Roth2017}.  This identity can be checked directly: for a
fixed label assignment, the sum of the displayed partition coefficients
over partitions that merge only vertices receiving the same label is one if
all labels are distinct and zero otherwise.

A merged graph containing a loop contributes zero because its product
contains $K_{ii}$ or $L_{ii}$.  A vertex belonging to just one edge also
gives zero, since summing its label uses a zero row sum of $K$ or $L$.
In each remaining graph, every vertex belongs to at least two edges,
counting parallel edges separately.  Since there are four edges, there
can be at most four vertices.  There are exactly ten graph patterns,
distinguishing $K$-edges from $L$-edges: two disjoint doubled
edges, with either equal or mixed colors on each edge; two adjacent doubled
edges with those same color patterns; two vertices joined by two $K$-edges
and two $L$-edges; three triangles with one doubled edge; and two four-cycles
with adjacent or alternating colors.  Their unrestricted sums are
$\tau_1,\ldots,\tau_9$, with the two exchanged same-color doubled
triangles combined in $\tau_6$.

Two graphs are color-isomorphic if their vertices can be renamed to match
all edges, including multiplicities and whether each is a $K$- or $L$-edge.
Write $|V(G)|$ for the number of vertices of $G$.
Every table entry can be obtained as follows.  For each of the six patterns $G$ in
\eqref{eq:spectral-six-patterns}, multiply its prefactor by
\[
 \sum_{\substack{\pi:\,G/\pi\text{ is color-isomorphic to class }j}}
       \prod_{D\in\pi}(-1)^{|D|-1}(|D|-1)!,
\]
and put the result in the column $m=|V(G)|$.  Sum contributions when two
patterns have the same $m$.  For row six, use either of the exchanged
triangle classes; their coefficients agree.  Enumerating the partitions
requires respectively $15,52,203,203,877,4140$ cases, totaling $5490$.
This exact finite rule yields Table~\ref{tab:spectral-coefficients}.

To express the partition sums explicitly, set
$I_{ab}=\operatorname{inj}_n(\mathsf G_a\sqcup\mathsf G_b)$, where
$\mathsf G_2=\mathsf A$, $\mathsf G_3=\mathsf B$, and
$\mathsf G_4=\mathsf C$.  The six partition sums reduce to
\begin{align*}
 I_{22}&=\tau_2-4\tau_4+2\tau_5,\\
 I_{23}&=-\tau_2+6\tau_4-4\tau_5+2\tau_7,\\
 I_{24}&=2\tau_2-16\tau_4+16\tau_5-8\tau_7,\\
 I_{33}&=\tau_2-\tau_3-12\tau_4+10\tau_5-\tau_6-10\tau_7
              +\tau_8+\tau_9,\\
 I_{34}&=-2\tau_2+8\tau_3+40\tau_4-48\tau_5+8\tau_6
              +48\tau_7-4\tau_8-4\tau_9,\\
 I_{44}&=4\tau_1+8\tau_2-96\tau_3-192\tau_4+288\tau_5
              -64\tau_6\\
       &\hspace{1em}-256\tau_7+32\tau_8+16\tau_9.
\end{align*}
Substitution into
$\widehat Q=b_4I_{22}-4b_5I_{23}
+b_6(2I_{24}+4I_{33})-4b_7I_{34}+b_8I_{44}$
gives the table entry by entry.

The three products $K^2,L^2,KL$ suffice for all nine scalar quantities.  Every
remaining operation is a vector or matrix-entry sum.  Equations
\eqref{eq:spectral-H4-contraction} and
\eqref{eq:spectral-Q-contraction} therefore establish the claimed
$O(n^3)$ arithmetic and $O(n^2)$ memory bounds, using ordinary matrix
multiplication after the pairwise kernels have been evaluated.

\subsection{Finite-sample limitations}
The requirement of eight observations cannot be lowered if unbiasedness
is required for every distribution in the stated class.  Take $X=Y\sim\operatorname{Bernoulli}(p)$
and the bounded linear kernels on $\{0,1\}$.  There is one singular value,
$s=p(1-p)$, and
\[
 H_4=H_2^2=S_4=[p(1-p)]^4.
\]
This polynomial has degree eight.  The expectation of a statistic based on
$n$ binary observations has degree at most $n$, so no statistic with $n<8$
can be unbiased for any of these targets for every $p$.

Nor can nonnegativity be guaranteed by a universally unbiased spectral
estimator.  Under independent symmetric Rademacher $X,Y$ with linear
kernels, all configurations in the finite sample space have positive
probability, but $H_4=H_2^2=S_4=0$.  A nonnegative unbiased statistic must
then be zero on every configuration.  It cannot have positive expectation
under a dependent distribution on the same support, for example one with
$\E(XY)=\rho\ne0$ and $|\rho|<1$.  This argument explains why the
cross-kernel spectral estimators can be negative even though their targets
are nonnegative.  It does not conflict with the nonnegative residual
mean square obtained when the same order-$r$ kernel array is used twice.

\subsection{Why equal singular values require care for effective rank}
At a nonzero cross-covariance operator $C$, put
$H_2=\norm C_{\mathrm{HS}}^2$ and $S_4=\tr\{(C^\ast C)^2\}$.
Viewed as a function of the operator $C=C_{XY}$, the ratio
$d_{\mathrm{eff}}=H_2^2/S_4$ has the following gradient with respect to the
Hilbert--Schmidt inner product:
\[
 \nabla_Cd_{\mathrm{eff}}
 =4\left\{\frac{H_2}{S_4}C
           -\frac{H_2^2}{S_4^2}C(C^\ast C)\right\}.
\]
If all $q$ nonzero singular values equal $s$, then
$H_2=qs^2$, $S_4=qs^4$, and $C(C^\ast C)=s^2C$, so this gradient
vanishes.  This includes rank one.  At these spectra the population ratio
does not change to first order under small changes in $C_{XY}$.  This is
one reason that we use effective rank to interpret the population spectra
and do not develop inference for it here.

\begin{samepage}
\section*{AI use statement}
During the preparation of this manuscript, the author used GPT-6,
GPT-5.1-Sol, and Opus 5.5 for language editing, exploratory proof
development, and mathematical review.
The author is responsible for the final selection and verification of all
definitions, statements, proofs, and computations and takes full
responsibility for the content of the manuscript.\par
\end{samepage}


\begin{thebibliography}{99}
\small

\bibitem[Bloznelis and G\"otze(2001)]{BloznelisGotze2001}
Bloznelis, M. and G\"otze, F. (2001).
Orthogonal decomposition of finite population statistics and its applications
to distributional asymptotics.
\emph{Annals of Statistics}, \textbf{29}, 899--917.
\url{https://doi.org/10.1214/aos/1009210694}.

\bibitem[Chakraborty and Zhang(2021)]{ChakrabortyZhang2021}
Chakraborty, S. and Zhang, X. (2021).
A new framework for distance and kernel-based metrics in high dimensions.
\emph{Electronic Journal of Statistics}, \textbf{15}, 5455--5522.
\url{https://doi.org/10.1214/21-EJS1889}.

\bibitem[Dunkl(1978)]{Dunkl1978}
Dunkl, C. F. (1978).
An addition theorem for Hahn polynomials: The spherical functions.
\emph{SIAM Journal on Mathematical Analysis}, \textbf{9}, 627--637.
\url{https://doi.org/10.1137/0509043}.

\bibitem[Filmus(2016)]{Filmus2016}
Filmus, Y. (2016).
An orthogonal basis for functions over a slice of the Boolean hypercube.
\emph{Electronic Journal of Combinatorics}, \textbf{23}(1), P1.23.
\url{https://doi.org/10.37236/4567}.

\bibitem[Gao and Shao(2023)]{GaoShao2023}
Gao, H. and Shao, X. (2023).
Two sample testing in high dimension via maximum mean discrepancy.
\emph{Journal of Machine Learning Research}, \textbf{24}(304), 1--33.
\url{https://www.jmlr.org/papers/v24/22-1136.html}.

\bibitem[Gottlieb(1966)]{Gottlieb1966}
Gottlieb, D. H. (1966).
A certain class of incidence matrices.
\emph{Proceedings of the American Mathematical Society}, \textbf{17},
1233--1237.
\url{https://doi.org/10.1090/S0002-9939-1966-0204305-9}.

\bibitem[Gretton et~al.(2005)Gretton, Bousquet, Smola, and Sch\"olkopf]{GrettonEtAl2005}
Gretton, A., Bousquet, O., Smola, A., and Sch\"olkopf, B. (2005).
Measuring statistical dependence with Hilbert--Schmidt norms.
In \emph{Algorithmic Learning Theory}, Lecture Notes in Computer Science,
\textbf{3734}, 63--77. Springer.
\url{https://doi.org/10.1007/11564089_7}.

\bibitem[Gretton et~al.(2007)Gretton, Fukumizu, Teo, Song, Sch\"olkopf, and Smola]{GrettonEtAl2008}
Gretton, A., Fukumizu, K., Teo, C. H., Song, L., Sch\"olkopf, B., and Smola,
A. J. (2007).
A kernel statistical test of independence.
In \emph{Advances in Neural Information Processing Systems 20 (NIPS 2007)}, 585--592.

\bibitem[Halmos(1946)]{Halmos1946}
Halmos, P. R. (1946).
The theory of unbiased estimation.
\emph{Annals of Mathematical Statistics}, \textbf{17}, 34--43.
\url{https://doi.org/10.1214/aoms/1177731020}.

\bibitem[Hoeffding(1948)]{Hoeffding1948}
Hoeffding, W. (1948).
A class of statistics with asymptotically normal distribution.
\emph{Annals of Mathematical Statistics}, \textbf{19}, 293--325.
 \url{https://doi.org/10.1214/aoms/1177730196}.

\bibitem[Huo and Sz\'ekely(2016)]{HuoSzekely2016}
Huo, X. and Sz\'ekely, G. J. (2016).
Fast computing for distance covariance.
\emph{Technometrics}, \textbf{58}, 435--447.
\url{https://doi.org/10.1080/00401706.2015.1054435}.

\bibitem[Kenny and La~Voie(1984)]{KennyLaVoie1984}
Kenny, D. A. and La Voie, L. (1984).
The social relations model.
\emph{Advances in Experimental Social Psychology}, \textbf{18}, 141--182.
\url{https://doi.org/10.1016/S0065-2601(08)60144-6}.

\bibitem[Mantel(1967)]{Mantel1967}
Mantel, N. (1967).
The detection of disease clustering and a generalized regression approach.
\emph{Cancer Research}, \textbf{27}, 209--220.

\bibitem[Peccati(2004)]{Peccati2004}
Peccati, G. (2004).
Hoeffding--ANOVA decompositions for symmetric statistics of exchangeable
observations.
\emph{Annals of Probability}, \textbf{32}, 1796--1829.
\url{https://doi.org/10.1214/009117904000000405}.

\bibitem[Roth(2017)]{Roth2017}
Roth, M. (2017).
Counting restricted homomorphisms via M\"obius inversion over matroid lattices.
In \emph{25th Annual European Symposium on Algorithms (ESA 2017)},
Leibniz International Proceedings in Informatics, \textbf{87}, 63:1--63:14.
\url{https://doi.org/10.4230/LIPIcs.ESA.2017.63}.

\bibitem[Sejdinovic et~al.(2013)Sejdinovic, Sriperumbudur, Gretton, and Fukumizu]{SejdinovicEtAl2013}
Sejdinovic, D., Sriperumbudur, B., Gretton, A., and Fukumizu, K. (2013).
Equivalence of distance-based and RKHS-based statistics in hypothesis testing.
\emph{Annals of Statistics}, \textbf{41}, 2263--2291.
\url{https://doi.org/10.1214/13-AOS1140}.

\bibitem[Serfling(1980)]{Serfling1980}
Serfling, R. J. (1980).
\emph{Approximation Theorems of Mathematical Statistics}.
Wiley, New York.

\bibitem[Shao and Zhang(2014)]{ShaoZhang2014}
Shao, X. and Zhang, J. (2014).
Martingale difference correlation and its use in high-dimensional variable
screening.
\emph{Journal of the American Statistical Association}, \textbf{109},
1302--1318.
\url{https://doi.org/10.1080/01621459.2014.887012}.

\bibitem[Song and Chen(2024)]{SongChen2024}
Song, H. and Chen, H. (2024).
Generalized kernel two-sample tests.
\emph{Biometrika}, \textbf{111}, 755--770.
\url{https://doi.org/10.1093/biomet/asad068}.

\bibitem[Song et~al.(2012)Song, Smola, Gretton, Bedo, and Borgwardt]{SongEtAl2012}
Song, L., Smola, A., Gretton, A., Bedo, J., and Borgwardt, K. (2012).
Feature selection via dependence maximization.
\emph{Journal of Machine Learning Research}, \textbf{13}, 1393--1434.

\bibitem[Sz\'ekely and Rizzo(2013)]{SzekelyRizzo2013}
Sz\'ekely, G. J. and Rizzo, M. L. (2013).
The distance correlation $t$-test of independence in high dimension.
\emph{Journal of Multivariate Analysis}, \textbf{117}, 193--213.
\url{https://doi.org/10.1016/j.jmva.2013.02.012}.

\bibitem[Sz\'ekely and Rizzo(2014)]{SzekelyRizzo2014}
Sz\'ekely, G. J. and Rizzo, M. L. (2014).
Partial distance correlation with methods for dissimilarities.
\emph{Annals of Statistics}, \textbf{42}, 2382--2412.
\url{https://doi.org/10.1214/14-AOS1255}.

\bibitem[Sz\'ekely et~al.(2007)Sz\'ekely, Rizzo, and Bakirov]{SzekelyRizzoBakirov2007}
Sz\'ekely, G. J., Rizzo, M. L., and Bakirov, N. K. (2007).
Measuring and testing dependence by correlation of distances.
\emph{Annals of Statistics}, \textbf{35}, 2769--2794.
\url{https://doi.org/10.1214/009053607000000505}.

\bibitem[Szab\'o and Sriperumbudur(2018)]{SzaboSriperumbudur2018}
Szab\'o, Z. and Sriperumbudur, B. K. (2018).
Characteristic and universal tensor product kernels.
\emph{Journal of Machine Learning Research}, \textbf{18}(233), 1--29.
\url{https://jmlr.org/papers/v18/17-492.html}.

\bibitem[Wang and Lindsay(2014)]{WangLindsay2014}
Wang, Q. and Lindsay, B. (2014).
Variance estimation of a general U-statistic with application to
cross-validation.
\emph{Statistica Sinica}, \textbf{24}, 1117--1141.
\url{https://doi.org/10.5705/ss.2012.215}.

\bibitem[Warner et~al.(1979)Warner, Kenny, and Stoto]{WarnerKennyStoto1979}
Warner, R. M., Kenny, D. A., and Stoto, M. (1979).
A new round robin analysis of variance for social interaction data.
\emph{Journal of Personality and Social Psychology}, \textbf{37}(10),
1742--1757.
\url{https://doi.org/10.1037/0022-3514.37.10.1742}.

\bibitem[Zhang et~al.(2018)Zhang, Yao, and Shao]{ZhangYaoShao2018}
Zhang, X., Yao, S., and Shao, X. (2018).
Conditional mean and quantile dependence testing in high dimension.
\emph{Annals of Statistics}, \textbf{46}, 219--246.
\url{https://doi.org/10.1214/17-AOS1548}.

\bibitem[Zhu et~al.(2020)Zhu, Zhang, Yao, and Shao]{ZhuZhangYaoShao2020}
Zhu, C., Zhang, X., Yao, S., and Shao, X. (2020).
Distance-based and RKHS-based dependence metrics in high dimension.
\emph{Annals of Statistics}, \textbf{48}, 3366--3394.
\url{https://doi.org/10.1214/19-AOS1934}.

\end{thebibliography}
\end{document}